\documentclass[twoside]{amsart}
\usepackage{amsmath,amssymb,amscd,amsthm,amsxtra,amsfonts}
\usepackage{mathtools,mathrsfs,dsfont,xparse}
\usepackage{graphicx,epstopdf}
\usepackage{multirow}
\usepackage{cases}
\usepackage{enumerate}
\usepackage{xcolor}
\usepackage{tikz,pgfplots}
\usetikzlibrary{matrix}
\usepackage{mathtools}
\usepackage{cjhebrew}

\usepackage[numbers,sort]{natbib}
\usepackage{multicol,bbm}
\usepackage[colorlinks=true, pdfstartview=FitV, linkcolor=blue, citecolor=blue, urlcolor=blue]{hyperref}

\usepackage[margin=1 in]{geometry}

\allowdisplaybreaks[4]

\theoremstyle{plain}
\newtheorem{thm}{Theorem}[section]
\newtheorem{prop}[thm]{Proposition}
\newtheorem{lem}[thm]{Lemma}
\newtheorem{cor}[thm]{Corollary}

\theoremstyle{definition}
\newtheorem{defn}[thm]{Definition}
\newtheorem{rmk}[thm]{Remark}

\numberwithin{equation}{section}
\numberwithin{figure}{section}
\newtheorem{claim}[thm]{Claim}

\newcommand{\case}[2]{ \noindent \textbf{Case #1:} #2 \\}
\newcommand{\scase}[2]{ \noindent \textbf{Sub-case #1:} #2 \\}

\newcommand{\M}{\mathcal{M}}
\newcommand{\OO}{\mathcal{O}}

\newcommand{\oo}{\mathfrak{O}}
\newcommand{\C}{\mathbb{C}}
\newcommand{\N}{\mathbb{N}}
\newcommand{\R}{\mathbb{R}}

\newcommand{\Z}{\mathbb{Z}}
\newcommand{\T}{\mathbb{T}}

\DeclareMathOperator{\im}{Im}

\newcommand{\wt}{\widetilde}

\newcommand{\dd}{\delta}
\newcommand{\DD}{\Delta}
\newcommand{\w}{\omega}
\newcommand{\ee}{\varepsilon}

\DeclareDocumentCommand{\abs}{s m}{
  \operatorname{}
  \IfBooleanTF{#1}{#2}{\left|#2\right|}}

\DeclareDocumentCommand{\norm}{s m}{
  \operatorname{}
  \IfBooleanTF{#1}{#2} {\left\| #2\right\|}}

\DeclareDocumentCommand{\inner}{s m}{
  \operatorname{}
  \IfBooleanTF{#1}{#2}{\left \langle#2\right \rangle}}

\DeclareDocumentCommand{\parenthese}{s m}{
  \operatorname{}
  \IfBooleanTF{#1}{#2}{\left(#2\right)}}

\DeclareDocumentCommand{\square}{s m}{
  \operatorname{}
  \IfBooleanTF{#1} {#2}{\left[#2\right]}}

\DeclareDocumentCommand{\bracket}{s m}{
  \operatorname{}
  \IfBooleanTF{#1}{#2}{\left\{#2\right\}}}

\begin{document}
\title[GWP for the quintic NLS]{On the global well-posedness for the periodic quintic nonlinear Schr\"odinger equation}

\author[Yu and Yue]{Xueying Yu and Haitian Yue}

\address{Xueying Yu
\newline \indent Department of Mathematics, MIT\indent 
\newline \indent  77 Massachusetts Ave, Cambridge, MA 02139.\indent}
\email{xueyingy@mit.edu}

\address{Haitian Yue
\newline \indent Department of Mathematics, University of Southern California\indent 
\newline \indent  3620 S. Vermont Ave., Los Angeles, CA 90089.\indent}
\email{haitiany@usc.edu}

\begin{abstract}
In this paper, we consider the initial value problem for the quintic, defocusing nonlinear Schr\"odinger equation on $\Bbb T^2$ with general data in the critical Sobolev space $H^{\frac{1}{2}} (\Bbb T^2)$. We show that if a solution remains bounded in $H^{\frac{1}{2}} (\Bbb T^2)$ in its maximal interval of existence, then the solution is globally well-posed in $\Bbb T^2$.

\vspace{0.1cm}
\noindent
\textbf{Keywords}: NLS, global well-posedness, compact manifold\\

\noindent
\textit{Mathematics Subject Classification (2020):} 35Q55, 35R01, 37K06, 37L50\\
\end{abstract}

\maketitle

\nocite{*}
\setcounter{tocdepth}{1}
\tableofcontents

\parindent = 10pt     
\parskip = 8pt

\section{Introduction}

We consider the initial value problem for the quintic,  defocusing nonlinear Schr\"odinger equation (NLS) on the two dimensional torus $\T_{\lambda}^2$:
\begin{align}\label{NLS}
\begin{cases}
(i \partial_t  + \Delta_{\T^2}) u = \abs{u}^4 u, & t \in \R , x \in \T_{\lambda}^2,\\
u(0,x) = u_0(x), & 
\end{cases}
\end{align}
where $u=u(t,x)$ is a complex-value function in spacetime $\R \times \T_{\lambda}^2$. Here $\T_{\lambda}^2$ is a general rectangular tori of dimension two, that is, 
\begin{align*}
\T_{\lambda}^2 : = \R^2/(\lambda_1 \Z \times \lambda_2 \Z), \quad \lambda= (\lambda_1, \lambda_2),
\end{align*}
where $\lambda_1, \lambda_2 \in  \R^+$. $\T_{\lambda}^2$ is called a rational (or irrational) tori if the ratio $\lambda_1/ \lambda_2$ is a rational (or irrational) number. This paper, in fact,  deals with general tori since our proof does not see any difference from being a rational torus or irrational one. Due to this reason, we will denote $\T_{\lambda}^2$ by $\T^2$ for convenience.

\subsection{Setup of NLS}
In general, the nonlinear Schr\"odinger equation with the power type nonlinearity is given by
\begin{align}\label{pNLS}
\begin{cases}
(i \partial_t  + \Delta_{\M^d}) u = \pm \abs{u}^{p-1} u, & t \in \R , x \in \M^d,\\
u(0,x) = u_0(x) & 
\end{cases}
\end{align}
where $\M^d$ is a d-dimensional manifold. In this paper, we only focus on the cases when $\M=\R$ or $\T$.  The \eqref{pNLS} is known as {\it defocusing} if the sign of the nonlinearity is positive and  {\it focusing} if negative.

Solutions of \eqref{pNLS} conserve both the mass:
\begin{align}\label{intro Mass}
M(u(t)) : = \int_{\M^d} \abs{u(t,x)}^2 \, dx = M(u_0) ,
\end{align}
and the energy:
\begin{align}\label{intro Energy}
E(u(t)) :  = \int_{\M^d} \frac{1}{2} \abs{\nabla u(t,x)}^2 \pm \frac{1}{p+1} \abs{u(t,x)}^{p+1} \, dx = E(u_0) .
\end{align}
Conservation laws of mass and energy give the control of $L^2$- and $\dot{H}^1$-norms of the solutions. 

The equation \eqref{pNLS} also conserves momentum,
\begin{align}\label{intro Momentum}  
\mathcal{P}(u(t)): =  \int_{\M^d} \im [\bar{u}(t,x) \nabla u(t,x)] \, dx = \mathcal{P}(u_0) ,
\end{align}
which plays a fundamental role in the derivation of the interaction Morawetz estimates \cite{CKSTT1, PV, CGT1, CGT2}. It should be noted that momentum is naturally associated to the regularity, however it does not control the $\dot{H}^{\frac{1}{2}}$-norm globally in time.

In Euclidean spaces,  the equation \eqref{pNLS} enjoys a {\it scaling symmetry}, which defines the criticality (relative to scaling) for this equation, and this {\it critical scaling} exponent{\footnote{More precisely,  the scaling of  \eqref{pNLS} is $u(t,x) \mapsto \lambda^{\frac{2}{p-1}} u( \lambda^2 t , \lambda x)$, under which the only invariant homogeneous $L_x^2$-based Sobolev norm  of initial data that is the $\dot{H}^{s_c} (\R^d)$-norm.}}  of \eqref{pNLS} on $\R^d$ is given by 
\begin{align*}
s_c:=\frac{d}{2} - \frac{2}{p-1} . 
\end{align*}
Accordingly, the problem NLS can be classified as {\it subcritical} or {\it critical} depending on whether the regularity of the initial data is below or equal to the scaling $s_c$ of \eqref{pNLS}. We will adopt the language in the scaling context in other manifolds  $\M^d$.

Our goal in this paper is to study the global well-posedness{\footnote{By {\it global/local well-posedness}, we mean the global/local in time existence, uniqueness and continuous dependence of the data to solution map.}} of the initial value problem \eqref{NLS}.  Before stating our main result, let us first review related works in the well-posedness theory of NLS. 
\subsection{History and related works}

\subsubsection{Well-posedness results in $\R^d$}

In both sub-critical regime (data in $H^s(\R^d)$, $s> s_c$) and critical regime (data in $\dot{H}^{s_c}(\R^d)$), the local well-posedness follows from the Strichartz estimates (see \cite{GV,Y,KT} for the Strichartz estimates in the Euclidean spaces $\R^d$) and a fixed point argument (see \cite{CW1, CW2, CW3, Ca}). In the former case, the time of existence depends solely on the $H^s$-norm of the data while in the latter one it depends also on the profile of the data.
Due to this reason, in the energy-subcritical regime ($s_c <1$), the conservation of energy gives global well-posedness in $H^1$ by iteration, in the energy-critical regime (see \cite{CW3, Ca}), a similar iteration, however, only gives global in time existence and scattering{\footnote{In general terms, with {\it scattering} we intend that the nonlinear solution as time goes to infinity approaches a linear one in some space $\mathcal{H}$.   To be precise,  there exists a solution $u_{\pm} (t)$ to the corresponding linear dispersive  equation s.t. $\lim_{t \to \pm \infty} \norm{u(t) - u_{\pm} (t)}_{\mathcal{H}} =0$.}} for small (in energy) initial data.  It is because for large initial data the time of existence depends also on the profile of the data (not a conserved quantity), which may lead to a shrinking interval of existence when iterating the local well-posedness theorem. 
Similarly, in the mass-subcritical regime ($s_c < s=0$), the conservation of mass gives global well-posedness in $L^2$ while if $s_c =0$, $p = 1+ \frac{4}{d}$, one cannot extend the local given solution to a global solution by iteration due to the same reason.

In the energy-critical case, Bourgain \cite{Bour99} first proved global existence and scattering in three dimensions for large finite energy data which is assumed to be radial. A different proof of the same result is given by Grillakis in \cite{G3}. Then, a breakthrough was made by Colliander-Keel-Staffilani-Takaoka-Tao \cite{CKSTT2}. They removed the radial assumption and proved global well-posedness and scattering of the energy-critical problem in three dimensions for general large data. Later \cite{RV, V} extended the result in \cite{CKSTT2} to higher dimensions. 
In \cite{KM1} Kenig and Merle proposed a road map to tackle critical problems. In fact, using a contradiction argument they first proved the existence of a critical element such that the global well-posedness and scattering fail. Then relying on a concentration compactness argument they showed that this critical element enjoys a compactness property up to the symmetries of this equation. In the second step, the problem is reduced to a rigidity theorem that precluded the existence of such critical element (also known as {\it  minimal blow-up solutions}).
It is worth mentioning that the concentration compactness method that they applied was first introduced in the context of Sobolev embeddings in \cite{Ge}, nonlinear wave equations in \cite{BG} and Schr\"odinger equations in \cite{MV, K1, K2}.
For the mass-critical regime, the radial case was first studied in \cite{TVZ1, KTV}. Then Dodson proved the global well-posedness and scattering in any dimension for nonradial data \cite{D1, D2, D3}. A key ingredient in Dodson's work is to prove a long time Strichartz estimate to rule out the minimal blow-up solutions.

For the global well-posedness at a level in which conservation laws are not available ($s_c \neq 0,1$), there is no control in the growth in time of the $\dot{H}^{s_c}$-norm of the solutions. Kenig and Merle \cite{KM3} showed for the first time that if a solution of the defocusing cubic NLS in three dimensions remains bounded in the critical norm $\dot{H}^{\frac{1}{2}}$ in the maximal time of existence, then the interval of existence is infinite and the solution scatters using concentration compactness and rigidity argument. \cite{Mu2} extended the $\dot{H}^{\frac{1}{2}}$ critical result in \cite{KM3} to dimensions four and higher. 
The analogue of the $\dot{H}^{\frac{1}{2}}$-critical result in dimensions two was obtained by the first author in \cite{Yu}. It is worth mentioning that the main difficulty in this low dimension setting is  the weaker time decay, which results in (1) a significantly different interaction Morawetz estimates in two dimensions from dimensions three and above, and (2) the failure of the endpoint Strichartz estimate.

\subsubsection{Well-posedness results in $\T^d$ and more}
Similar as in $\R^d$ settings, the local theory in sub-critical regimes in $\T^d$ also relies on the Strichartz estimates. In contrast, the Strichartz estimates in compact manifold are significantly different from the ones in the Euclidean spaces because the dispersion in such settings is weaker than unbounded manifold. This same reason also leads to nonexistence of scattering in $\T^d$. On rational tori $\T^d$, the Strichartz estimates were initially developed by Bourgain \cite{Bour93}. The proof of \cite{Bour93} employed a number theoretical related lattice counting argument, which works better in the rational tori than irration tori. A recent work of Bourgain-Demeter \cite{BD} proved the optimal Strichartz estimates for both rational and irrational tori using a completely different approach, which does not rely on the lattice counting lattice{\footnote{In fact, \cite{BD} proved the $l^2$ decoupling conjecture, and the full range of expected $L_{t,x}^p$ Strichartz estimates on both rational and irrational tori was derived as a consequence of this decoupling conjecture.}.

More results on the study of the Strichartz estimates on tori and well-posedness problem in sub-critical regimes can be found in \cite{Bour07, DPSN, CaW, Bour13, GOW, KV4, De, DGG, FSWW}. On more general compact manifolds, Burq-Gerard-Tzvetkov derived the Strichartz type estimates and applied these estimates to the global well-posedness of NLS on compact manifolds in a series of their papers \cite{BGT1, BGT2, BGT3, BGT4}. We also refers to \cite{Z, GP, H1, H2} and references therein for the other results of global existence sub-critical NLS on compact manifolds.

In the critical regime, due to the vital difference (a logarithmic loss) in the Strichartz estimates, one cannot close the fixed point argument and obtain a local theory in the same spirit the critical problems were treated in $\R^d$. Herr-Tataru-Tzvetkov \cite{HTT1} first proved the global well-posedness with small initial data in $H^1$ for the energy-critical NLS on $\T^3$. To overcome the logarithmic loss, they implemented a crucial trilinear Strichartz type of estimates in the context of the critical atomic spaces $U^p$ and $V^p$, which were originally developed in unpublished work on wave maps by Tararu. These atomic spaces were systematically formalized by Hadac-Herr-Koch \cite{HHK} (see also \cite{KT1, HTT2}) and now the atomic spaces $U^p$ and $V^p$ are widely used in the field of the critical well-posedness theory of nonlinear dispersive equations.

For large data in the critical regime, the first result in compact manifolds is by Ionescu-Pausader \cite{IP1}, where they showed that the global well-posedness of the energy-critical NLS on rational $\T^3$. In a series of papers, Ionescu-Pausader \cite{IP1, IP2} and Ionescu-Pausader-Staffilani \cite{IPS} developed a method (the {\it `black-box' trick}) to obtain energy-critical large data global well-posedness in more general manifolds based on the corresponding results on the Euclidean spaces in the same dimensions. More precisely, using the Kenig-Merle road map and an ad hoc profile decomposition technique, the authors of \cite{IP1, IP2, IPS}  successfully transferred the already available energy-critical global well-posedness results in $\R^d$ into the  $\T^3$, $\T^3\times\R$, and $\mathbb{H}^3$ settings{\footnote{Here, the product spaces $\T^m \times \R^n$ ($m,n \in \Z^+$) are known as waveguides, and $\mathbb{H}^d$ are $d$-dimensional hyperbolic spaces.}}. This methodology now has been widely applied to many manifolds (see \cite{PTW, St, Z1, Z2}). Let us also point out two recent results \cite{Yue, YYZ} towards the understanding the long time behavior of the focusing NLS on the tori and waveguide manifolds using {\it `black-box' method}. On compact manifolds, one does not expect the scattering effect, however, on some product spaces, some nice long time behaviors (such as scattering, modified scattering) of the solutions have been studied \cite{TV1, HP, TV2, HPTV, GPT, CGYZ, Z1, Z2}.

Now we are ready to state our main result in this work.
\subsection{Main result}
\begin{thm}[Global well-posedness of \eqref{NLS}]\label{thm Main}
Let $u_0 \in H^{\frac{1}{2}}(\T^2)$ and   $u : I \times \R^2 \to \C$  be a maximal-lifespan solution to \eqref{NLS} with initial data $u_0$ such that for any $T>0$, 
\begin{align*}
\sup\limits_{t \in I \cap [-T,T]  } \norm{u(t)}_{ H_x^{\frac{1}{2}}(\T^2)} < \infty .
\end{align*}
Then the solution is globally well-posed. In addition, the mapping $u_0 \mapsto  u$ extends to a continuous mapping from $H^{\frac{1}{2}} (\T^2)$ to the $X^{\frac{1}{2}} ([-T, T])$ space{\footnote{Note that the space $X^{\frac{1}{2}} (I) \subset C(I: H^{\frac{1}{2}} (\T^2))$ is the adapted atomic space (see Definition \ref{defn X^s Y^s}).}}.
\end{thm}

In this paper, we prove the large data global wellposedness result of defocusing $\dot{H}^{\frac{1}{2}}$-critical NLS on the both rational and irrational tori in two dimensions. To the authors' best knowledge, this is the first global wellposedness result on compact manifolds in the critical settings at a level in which conservation laws are not available.
We remark that our proof does not differentiate rational and irrational tori, because the Strichartz type estimates (Lemma \ref{lem Strichartz1}) by Bourgain-Demeter \cite{BD} that we use in this paper work both for rational or irrational tori.

\subsection{Outline of the proof}

In our proof, we follow the Kenig-Merle's road map \cite{KM1} and  the `black-box' strategy \cite{IP1, IP2, IPS}, and we also rely on atomic spaces introduced by \cite{HHK, HTT1, HTT2}. Now let us summarize below the main steps and ingredients in the proof of Theorem \ref{thm Main}. 

\subsubsection*{\bf Step 1}
A local well-posedness argument and a stability theory of \eqref{NLS} in $\T^2$. 

The key ingredients in this step are:
\vspace{-0.3cm}
\begin{itemize}
\item[$-$]
the construction of suitable solution spaces $X^s$ and $Y^s$ (Definition \ref{defn X^s Y^s})
\item[$-$]
a weaker critical solution norm $Z$-norm, which plays a similar role as the $L^{10}_{t,x}$-norm in \cite{CKSTT2} (Definition \ref{defn Z}),
\item[$-$]
a trilinear estimate (Proposition \ref{prop Loc Trilinear}, Lemma \ref{lem Trilinear} and Remark \ref{rmk Trilinear}).
\end{itemize}
As we mentioned above, due to a logarithmic loss in Strichartz estimates, it is not enough to close the fixed point argument in the critical local theory. Hence a trilinear estimate is needed. First, we adapt Herr-Tataru-Tzvetkov's idea \cite{HTT1, HTT2} and introduce the adapted critical spaces $X^s$ and $Y^s$, which are frequency localized modification of atomic spaces $U^p$ and $V^p$, as our solution spaces and nonlinear spaces. Then we define a weaker critical norm $Z$  (see Definition \ref{defn Z}) of the solution, which plays a similar role as the $L^{10}_{t,x}$-norm in \cite{CKSTT2}. Next we prove a trilinear estimate for frequency-localized functions in the atomic spaces defined above, which greatly makes up for the logarithmic loss in Strichartz estimates and hence helps to close the fixed point argument in our local theory (Proposition \ref{prop LWP}). 
Moreover, using this $Z$-norm based local theory, we also obtain a stability theory (Proposition \ref{prop Stability}) and a fact that the boundedness of $Z$-norms implies the well-posedness of the solutions.

\subsubsection*{\bf Step 2} 

Preparation for the rigidity argument. 

The key ingredients in this step are:
\vspace{-0.3cm}
\begin{itemize}
\item[$-$]
understanding the behaviors of Euclidean profiles (Section \ref{sec Euclidean profiles}),
\item[$-$]
linear and nonlinear profile decompositions of the minimal blow-up solutions (Section \ref{sec Profile decomp}).
\end{itemize}
An important step  in the `black-box' trick is to  decompose the minimal blow-up solutions into two types of profiles: Euclidean profiles and scale-one profiles. Here the Euclidean profiles are those that concentrate in  space and time, while the scale-one profiles are those that spread out more in space. Moreover the corresponding nonlinear Euclidean profiles (which are obtained by running the Euclidean profiles along the NLS flow as initial data) are Euclidean-like solutions, which can be viewed as solutions in the Euclidean space $\R^2$. And the interaction between these two types of profiles can be showed to be very weak, hence almost orthogonal.

\subsubsection*{\bf Step 3}

Rigidity argument. 

In this step, we prove the main theorem (Theorem \ref{thm Main}) by contradiction, following the road map set up in \cite{KM3}. We assume that $Z$-norms are not bounded for all solutions. Notice in {\bf Step 1}, we proved that at least for the solutions with small data in $H^{\frac{1}{2}} (\T^2)$, their $Z$-norms are bounded globally. Combining the assumption with {\bf Step 1}, we are able to find the minimal blow-up solutions and decompose them into Euclidean profiles and scale-one profiles as we discussed above. With understanding of their behaviors, we can preclude the existence of such minimal blow-up solutions and hence complete the proof of Theorem \ref{thm Main}. More precisely, by contradiction argument, we construct a sequence of initial data which implies a sequence of solutions and leads the $Z$-norm towards infinity. Then following the profile decomposition idea, we perform a linear profile decomposition of the sequence of initial data with a scale-one profile and a series of Euclidean profiles that concentrate at space-time points.  We get nonlinear profiles by running the linear profiles along the NLS flow as initial data. By the contradiction condition, the scattering properties of nonlinear Euclidean profiles and the defect of interaction between different profiles show that there is actually at most one profile (which is the Euclidean profile). Moreover its corresponding nonlinear Euclidean profile is just the Euclidean-like solution, which can be interpreted in some sense as solutions in the Euclidean space $\R^2$. However, these kind of concentration as a Euclidean-like solution is prevented by the global well-posedness theory on the Euclidean space $\R^2$ in \cite{Yu}.

\subsection{Organization of the paper}

In Section \ref{sec Preliminaries}, we introduce the notations and the functional spaces with their properties that we will need in this paper. In Section \ref{sec LWP}, we prove a trilinear estimate and based on which we present a local well-posedness result and a stability argument. In Section \ref{sec Euclidean profiles}, we study the behavior of Euclidean-like solutions to the linear and nonlinear equation concentrating to a point in space and time. In Section \ref{sec Profile decomp}, we prove a profile decomposition that will be used in Section \ref{sec Rigidity}. Finally in Section \ref{sec Rigidity}, we prove the main theorem (Theorem \ref{thm Main}).

\subsection*{Acknowledgements}
X.Y. is funded in part by the Jarve Seed Fund and an AMS-Simons travel grant. Both authors would like to thank Gigliola Staffilani for suggesting this problem.

\section{Preliminaries}\label{sec Preliminaries}
In this section, we  discuss  notations used in the rest of the paper, and introduce the function spaces (atomic spaces) that we will be working with.

\subsection{Notations}
We define
\begin{align*}
\norm{f}_{L_t^q L_x^r (I \times \M)} : = \square{\int_I \parenthese{\int_{\M} \abs{f(t,x)}^r \, dx}^{\frac{q}{r}} dt}^{\frac{1}{q}},
\end{align*}
where $I$ is a time interval.

For $x\in \R$, we set $\inner{x} = (1 + \abs{x}^2)^{\frac{1}{2}}$. We adopt the usual notation that $A \lesssim  B$ or $B \gtrsim A$ to denote an estimate of the form $A \leq C B$ , for some constant $0 < C < \infty$ depending only on the {\it a priori} fixed constants of the problem. We write $A \sim B$ when both $A \lesssim  B $ and $B \lesssim A$.

\subsection{Definition of solutions}

\begin{defn}[Solutions]
Given an interval $I \subset \R$, we call $u \in C (I : H^{\frac{1}{2}} (\T^2))$ a strong solution of \eqref{NLS} if $u \in X^{\frac{1}{2}} (I)$ and $u$ satisfies that for all $t,s \in I$,
\begin{align*}
u(t) = e^{i(t-s)\Delta} u(s) - i \int_s^t e^{i(t-t' ) \Delta} (\abs{u}^4 u ) (t'  ) \, dt' .
\end{align*} 
\end{defn}

\subsection{Atomic spaces}
In this subsection, we recall some basic definitions and properties of atomic spaces.

The atomic spaces were first introduced in partial differential equations in \cite{KT1}, and then applied to KP-II equations in \cite{HHK} and nonlinear Schr\"odinger equations in \cite{KT2, KT3, HTT1, HTT2}. They are useful in many different settings, and it is worth mentioning that they are quite helpful in fixing the lack of Strichartz endpoint issues. 

First, we recall some basic definitions and properties of the atomic spaces. 
\begin{defn}[$U_{\Delta}^p$-spaces, Definitions 2.1, 2.15 in \cite{HHK}]
Let $1 \leq p  < \infty$. Let $U_{\Delta}^p$ be an atomic space whose atoms are piecewise solutions to the linear equation. $u_{\lambda}$ is a $U_{\Delta}^p$-atom if there exists an increasing sequence $\{ t_k\}_{k=1}^N$, $N$ may be finite or infinite, $t_0 =-\infty$, $N$ is finite, $t_{N+1} = + \infty$, and
\begin{align*}
u_{\lambda} = \sum_{k=0}^N \mathbf{1}_{[t_k , t_{k+1})} e^{it\Delta} u_k,  \quad \text{ and }  \quad \sum_k \norm{u_k}_{L^2}^p =1. 
\end{align*}
If $J \subset \R$ is an interval, then we say that $u_{\lambda}$ is a $U_{\Delta}^p (J)$-atom if $t_k \in J$ for all $1 \leq k \leq N$. Then for any $1 \leq p < \infty$, let 
\begin{align*}
\norm{u}_{U_{\Delta}^p (J \times \R^2)} = \inf \bracket{ \sum_{\lambda} \abs{c_{\lambda}} : u= \sum_{\lambda} c_{\lambda} u_{\lambda}, u_{\lambda} \text{ are }  U_{\Delta}^p(J)-\text{atoms}} .
\end{align*} 
\end{defn}

For any $1 \leq p < \infty$, $U_{\Delta}^p (J \times \R^2) \subset L^{\infty} L^2 (J \times \R^2)$. Additionally, $U_{\Delta}^p$-functions are continuous except at countably many points and right-continuous everywhere.

\begin{defn}[$V_{\Delta}^p$-spaces, Definitions 2.3, 2.15 in \cite{HHK}]
Let $1 \leq p < \infty$. Then $V_{\Delta}^p$ is the space of functions $u \in L^{\infty} (L^2)$ such that 
\begin{align*}
\norm{v}_{V_{\Delta}^p}^p = \norm{v}_{L_t^{\infty} L_x^2 }^p + \sup_{\{ t_k \} \nearrow} \sum_k \norm{e^{-it_k} v(t_k) - e^{-it_{k+1} \Delta} v(t_{k+1})}_{L_x^2}^p.
\end{align*}
The supremum is taken over increasing sequences $t_k$. If $J \subset \R$, then
\begin{align*}
\norm{v}_{V_{\Delta}^p (J \times \R^2)}^p = \norm{v}_{L_t^{\infty} L_x^2  (J \times \R^2)}^p + \sup_{\{ t_k \}\nearrow} \sum_k \norm{e^{-it_k} v(t_k) - e^{-it_{k+1} \Delta} v(t_{k+1})}_{L_x^2}^p.
\end{align*}
where each $t_k$ lies in $J$. Note that $\{ t_k \}$ may be a finite or infinite sequence.
\end{defn}

\begin{prop}[Embedding, Proposition 2.2, Proposition 2.4 and Corollary 2.6 in \cite{HHK}]\label{prop Embed}
For $1\leq p \leq q < \infty$,
\begin{align*}
U_{\Delta}^p(\R, L^2) \hookrightarrow V_{\Delta}^p (\R, L^2) \hookrightarrow U_{\Delta}^q(\R, L^2) \hookrightarrow  L^{\infty}(\R,L^2) .
\end{align*}
\end{prop}

\begin{prop}[Duality, Theorem 2.8 in \cite{HHK}]\label{prop Duality}
Let $DU_\Delta^p$ be the space of functions
\begin{align*}
DU_\Delta^p = \bracket{(i\partial_t + \Delta) u : u \in U_\Delta^p },
\end{align*}
and the $DU_\Delta^p = (V_\Delta^{p' })^{\ast}$, with $\frac{1}{p} + \frac{1}{p' } = 1$. Then $(0 \in J )$
\begin{align*}
\norm{\int_0^t e^{i(t-t' )\Delta}F(u)(t' )\, dt' }_{U_\Delta^p (J\times \R^d)} \lesssim \sup\bracket{\int_J \inner{ v, F}  \,  dt: \norm{v}_{V_\Delta^{p' } }=1 }.
\end{align*}
\end{prop}

\begin{prop}[Transfer Principle, Proposition 2.19 in \cite{HHK}]\label{prop TransferPrinciple}
Let 
\begin{align*}
T_0 : L^2 \times \dots \times L^2 \to L_{loc}^{1}
\end{align*}
be an m-linear operator. Assume that for some $1\leq p,\ q \leq \infty$
\begin{align*}
\norm{ T_0 (e^{it\Delta} \phi_1,\dots, e^{it\Delta} \phi_m)}_{L^p(\R, L_x^q)} \lesssim \prod_{i=1}^m\norm{\phi_i}_{L^2}.
\end{align*}
Then, there exists an extension $T : U_{\Delta}^p \times \dots \times U_{\Delta}^p \to L^p(\R, L_x^q)$ satisfying
\begin{align*}
\norm{T(u_1, \dots, u_m)}_{L^p(\R, L_x^q)} \lesssim \prod_{i=1}^m \norm{u_i}_{U_{\Delta}^p},
\end{align*}
and  such that $T(u_1, \dots , u_m) (t, \cdot) = T_0 (u_1(t), \dots , u_m(t))(\cdot)$, a.e.
\end{prop}

\begin{defn}[$X^s$ and $Y^s$ spaces]\label{defn X^s Y^s}
We define $X^s ([0,1])$ and $Y^s ([0,1])$ to be the spaces of all functions $u: [0,1] \to H^s (\T^d)$ such that for every $\xi \in \Z^d$ the map $ t \to \widehat{e^{-it\Delta} u(t)} (\xi)$ is in $U^2([0,1] ; \C)$ and $V_{rc} ([0,1]; \C)$, respectively, with norms given by
\begin{align*}
\norm{u}_{X^s ([0,1])} : = \parenthese{\sum_{\xi \in \Z^d} \inner{\xi}^{2s} \norm{\widehat{e^{-it\Delta}u(t)}(\xi)}_{U^2}^2}^{\frac{1}{2}},\\
\norm{u}_{Y^s ([0,1])} : = \parenthese{\sum_{\xi \in \Z^d} \inner{\xi}^{2s} \norm{\widehat{e^{-it\Delta}u(t)}(\xi)}_{V^2}^2}^{\frac{1}{2}}.
\end{align*}

For intervals $I \subset \R$, we define $X^s (I)$, $s \in \R$, in the usual way as restriction norms; thus
\begin{align*}
X^s(I) : = \{ u \in C(I : H^s) : \norm{u}_{X^s (I)} : = \sup_{J \subset I, \abs{J} \leq 1}  \inf_{v|_J (t) = u|_J (t)} \norm{v}_{\widetilde{X}^s} < \infty\} .
\end{align*}

The spaces $Y^s(I)$ are defined in a similar way.

\begin{prop}\label{prop Duality X}
Let $s \geq 0$ and $T >0$. For $f \in L^1 ([0,1] ; H^s (\T^2))$, we have
\begin{align*}
\int_0^t e^{t(t-s) \Delta} f (s) \, ds \in X^s ([0,1]),
\end{align*}
and
\begin{align*}
\norm{\int_0^t e^{t(t-s) \Delta} f (s) \, ds}_{X^s ([0,1])} \leq \sup_{v \in Y^{-s} ([0,1]) : \norm{v}_{Y^{-s}} =1} \abs{ \int_0^T \int_{\T^2} f(t,x) \overline{v(t,x)} \, dx dt }
\end{align*}
\end{prop}

The norm controlling the inhomogeneous term on an interval $I =(a,b)$ is then defined as
\begin{align*}
\norm{h}_{N(I)} : = \norm{\int_a^t e^{i(t-s) \Delta} h(s) \, ds}_{X^{\frac{1}{2}} (I)}.
\end{align*}
\end{defn}

We also need the following weaker critical norm. 
\begin{defn}[$Z$-norm]\label{defn Z}
\begin{align*}
\norm{u}_{Z(I)} : = \sum_{p \in \{ p_0, p_1\} }\sup_{J \subset I , \abs{J} \leq 1} \parenthese{ \sum_N N^{4 -\frac{p}{2}} \norm{P_N u(t)}_{L_{t,x}^p (J \times \T^2)}^p}^{\frac{1}{p}} .
\end{align*}
where $p_0 = 4.1$, $p_1 = 82$. Note that $\frac{2}{p_0} + \frac{1}{p_1} = \frac{1}{2}$ and the reason  why we choose these two values of $p$'s in this way is that we will use them in an $L^2$ estimate later in the proof.
\end{defn}

\begin{rmk}\label{rmk Z<X}
Moreover,
\begin{align*}
\norm{u}_{Z(I) } \lesssim \norm{u}_{X^{\frac{1}{2}}(I)},
\end{align*}
thus $Z$ is indeed a weaker norm.  In fact, using Definition \ref{defn Z} and Strichartz estimates (Lemma \ref{lem Strichartz2}) presented in the next section, we see that for $p=p_0, p_1$
\begin{align*}
\sup_{J \subset I , \abs{J} \leq 1} \parenthese{ \sum_N N^{4 -\frac{p}{2}} \norm{P_N u(t)}_{L_{t,x}^p (J \times \T^2)}^p}^{\frac{1}{p}}  \lesssim \sup_{J \subset I , \abs{J} \leq 1} \parenthese{ \sum_N N^{\frac{p}{2} }  \norm{P_N u(t)}_{X^{0} (J \times \T^2)}^p}^{\frac{1}{p}}  \lesssim \norm{u}_{X^{\frac{1}{2}}(I)}.
\end{align*}
\end{rmk}

\begin{rmk}\label{rmk Embed}
We have the continuous embedding 
\begin{align*}
U_{\Delta}^2 H^s \hookrightarrow X^s \hookrightarrow Y^s \hookrightarrow V_{\Delta}^2 H^s .
\end{align*}
We also note that 
\begin{align*}
\norm{u}_{L_t^{\infty} H_x^s ([0,1] \times \T^2)} \lesssim \norm{u}_{X^s ([0,1])}
\end{align*}
and
\begin{align}\label{eq Embed}
\norm{\int_0^t e^{i(t-s) \Delta} F(s) \, ds }_{ X^s ([0,1])} = \norm{F}_{N([0,1])} \lesssim \norm{F}_{L_t^1 H_x^s ([0,1] \times \T^2)}.
\end{align}

\end{rmk}

\section{Local well-posedness and stability}\label{sec LWP}
In this section, we present a local theory  via trilinear estimates and a stability argument. 

Before we work on trilinear estimates, let us first review the Strichartz estimates on $\T^2$. 
\subsection{Strichartz estimates}
Recall the Strichartz estimates obtained by Bourgain \cite{Bour93} and Bourgain-Demeter \cite{BD}.
\begin{lem}[Strichartz estimates \cite{Bour93, BD}]\label{lem Strichartz1}
If $p>4$, then Strichartz estimates on $\T^2$ read
\begin{align*}
\norm{P_N e^{it\Delta} f}_{L_{t,x}^p([0,1] \times \T^2)} \lesssim_p N^{1- \frac{4}{p}} \norm{P_N f}_{L_x^2 (\T^2)}
\end{align*}
and
\begin{align*}
\norm{P_C e^{it\Delta} f}_{L_{t,x}^p([0,1] \times \T^2)} \lesssim_p N^{1 - \frac{4}{p}} \norm{P_N f}_{L_x^2 (\T^2)}
\end{align*}
where $C$ is a cube of side length $N$.
\end{lem}

\begin{lem}\label{lem Strichartz2}
Using Proposition \ref{prop TransferPrinciple} and Remark \ref{rmk Embed}, we can write the Strichartz estimates:
\begin{align*}
\norm{P_{\leq N} u}_{L_{t,x}^p ([0,1] \times \T^2)} \lesssim N^{1 -\frac{4}{p}} \norm{P_{\leq N} u}_{U_{\Delta}^p L^2} \lesssim N^{1 - \frac{4}{p}} \norm{P_{\leq N} u}_{Y^0 ([0,1])}
\end{align*}
for any $p > 4$ and $N \geq 1$. 

In particular, due to the Galilean invariance of solutions to the linear Schr\"odinger equation,
\begin{align*}
\norm{P_C u}_{L_{t,x}^p ([0,1] \times \T^2)} \lesssim N^{1 - \frac{4}{p}} \norm{P_C u}_{Y^0 ([0,1])}
\end{align*}
for any $p > 4$ and any cube $C \subset \R^2$ of side-length $N \geq 1$.
\end{lem}

Besides the Strichartz estimates (with frequencies cut into cubes of side length $N$) above,  in the proof of trilinear estimates in Subsection \ref{ssec Trilinear}, we also need the following Strichartz estimates,  where we cut frequencies in rectangles with side  length $2N$ and side width $2M$.

Let $\mathcal{R}_M (N)$ be the collection of all sets in $\Z^2$ which are given as the intersection of a cube of side length $2N$ for solutions with strips of width $2M$, that is the collection of all sets of the form
\begin{align*}
(\xi_0 + [-N, N]^2) \cap [\xi \in \Z^2 : \abs{a \cdot \xi -A} \leq M]
\end{align*}
which some $\xi_0 \in \Z^2$, $a \in \R^2$, $\abs{a} \leq 1$, $A \in \R$.
\begin{prop}\label{prop Strichartz L^infty}
For all $1 \leq M \leq N$ and $R \in \mathcal{R}_M (N)$ we have
\begin{align*}
\norm{P_R e^{it\Delta} \phi}_{L_{t,x}^{\infty} (\T \times \T^2)} \lesssim M^{\frac{1}{2}} N^{\frac{1}{2}} \norm{P_R \phi}_{L_x^2 (\T^2)}.
\end{align*}
\end{prop}
\begin{proof}[Proof of Proposition \ref{prop Strichartz L^infty}]
For fixed $t$ we the uniform bound
\begin{align*}
\norm{P_R e^{it\Delta} \phi}_{L_{t,x}^{\infty} (\T^2)} \leq \sum_{\xi \in \Z^2 \cap R} \abs{\widehat{\phi} (\xi)} \leq (\# (\Z^2 \cap R))^{\frac{1}{2}} \norm{P_R \phi}_{L_x^2 (\T^2)}.
\end{align*}
We cover the rectangle by approximately $\frac{N}{M}$ cubes of side length $M$ and faces parallel to the coordinate planes. Each such cube contains approximately $M^2$ lattice points. Altogether, we arrive at the upper bound
\begin{align*}
\# (\Z^2 \cap R) \lesssim MN ,
\end{align*}
which shows the claim.
\end{proof}

\begin{cor}
Let $p > 4$ and $0 < \delta < \frac{1}{2} - \frac{2}{p}$. For all $1 \leq M \leq N$ and $R \in \mathcal{R}_M (N)$ we have 
\begin{align*}
\norm{P_R e^{it\Delta} \phi}_{L_{t,x}^p (\T \times \T^2)} \lesssim N^{1-\frac{4}{p}} (\frac{M}{N})^{\delta} \norm{P_R \phi}_{ L_x^2 (\T^2)}.
\end{align*}
\end{cor}

\subsection{Trilinear estimates}\label{ssec Trilinear}

\begin{prop}[Frequency localized trilinear estimates]\label{prop Loc Trilinear}
There exists $\delta > 0$ such that for any $N_1 \geq N_2 \geq N_3 \geq 1$ and any interval $I \subset [0, 2\pi]$ we have
\begin{align}\label{eq Trilinear}
\norm{  P_{N_1} u_1  P_{N_2} u_2 P_{N_3} u_3  }_{L_{t,x}^2 (I \times \T^2)} \lesssim (N_2 N_3)^{\frac{1}{2}} \bigg( \frac{N_3}{N_1} + \frac{1}{N_2} \bigg)^{\delta}\cdot \prod_{j=1}^3 \norm{P_{N_j} u_j}_{Y^0(I)}.
\end{align}
\end{prop}

\begin{proof}[Proof of Proposition \ref{prop Loc Trilinear}]
The proof of this trilinear estimates is adapted  from the proof of Proposition 3.5 in \cite{HTT1}.
We begin the proof with some simplifications. First of all, it is enough to consider the case $I = [0, 2 \pi]$. Second, given a partition $\Z^2 = \cup C_j$ int cubes $C_j \in \mathcal{C}_{N_2}$, the outputs $P_{N_1} P_{C_j} u_1 P_{N_2} u_2 P_{N_3} u_3$ are almost orthogonal because the spatial Fourier support of $u_2 P_{N_3} u_3$ is contained in at most finitely many cubes of side length $N_2$. Additionally, we have
\begin{align*}
\norm{u}_{Y^0 (I)}^2 = \sum_j \norm{P_{C_j} u}_{Y^0(I)}^2,
\end{align*}
which implies that it suffices to prove \eqref{eq Trilinear} in the case when the first factor is further restricted to a cube $C \in \mathcal{C}_{N_2}$,
\begin{align}\label{eq Trilinear P_C}
\norm{P_C P_{N_1} u_1 P_{N_2} u_2 P_{N_3} u_3}_{L_{t,x}^2(I \times \T^2)} \lesssim (N_2 N_3)^{\frac{1}{2}} ( \frac{N_3}{N_1} + \frac{1}{N_2} )^{\delta} \prod_{j=1}^3 \norm{P_{N_j} u_j}_{Y^0(I)}.
\end{align}
Due to Remark \ref{rmk Embed}, we are allowed to replace $Y^0$ by $V_{\Delta}^2 L^2$. Then \eqref{eq Trilinear P_C} follows by interpolation, from the following two trilinear estimates:
\begin{align}\label{eq Interpolation1}
\norm{P_C P_{N_1} u_1 P_{N_2} u_2 P_{N_3} u_3}_{L_{t,x}^2 (I \times \T^2)} \lesssim (N_2 N_3)^{\frac{1}{2}} ( \frac{N_3}{N_1} )^{\delta' } \prod_{j=1}^3 \norm{P_{N_j} u_j}_{U_{\Delta}^4  (I, L^2)}
\end{align}
for $0 < \delta'  < \frac{1}{2}$, respectively,
\begin{align}\label{eq Interpolation2}
\norm{P_C P_{N_1} u_1 P_{N_2} u_2 P_{N_3} u_3}_{L_{t,x}^2(I \times \T^2)} \lesssim (N_2 N_3)^{\frac{1}{2}} ( \frac{N_3}{N_1} + \frac{1}{N_2} )^{\delta^{\prime \prime}} \prod_{j=1}^3 \norm{P_{N_j} u_j}_{U_{\Delta}^2  (I, L^2)}
\end{align}
for some $\delta^{\prime \prime} >0$.

The first bound \eqref{eq Interpolation1} follows from H\"older's inequality with $4 < p < 5$ and $q$ such that $\frac{2}{p} + \frac{1}{q} = \frac{1}{2}$ and  Strichartz estimates
\begin{align*}
& \quad \norm{P_C P_{N_1} u_1 P_{N_2} u_2 P_{N_3} u_3}_{L_{t,x}^2(I \times \T^2)}  \\
& \leq \norm{P_C P_{N_1} u_1}_{L_{t,x}^p(I \times \T^2)} \norm{ P_{N_2} u_2}_{L_{t,x}^p(I \times \T^2)}  \norm{P_{N_3} u_3}_{L_{t,x}^q (I \times \T^2)} \\
& \lesssim N_2^{2-\frac{8}{p}} N_3^{1-\frac{4}{q}} \norm{P_{N_1} u_1}_{U_{\Delta}^p  (I, L^2)} \norm{ P_{N_2} u_2}_{U_{\Delta}^p  (I, L^2)}  \norm{P_{N_3} u_3}_{U_{\Delta}^q  (I, L^2)}
\end{align*}
and the embedding $U_{\Delta}^4 L^2 \hookrightarrow U_{\Delta}^p L^2 \hookrightarrow U_{\Delta}^q L^2$.

For the second bound \eqref{eq Interpolation2} we can use the transfer principle (Proposition \ref{prop TransferPrinciple}) to reduce the problem to the similar estimate for the produce of three solutions to the linear Schr\"odinger equation:
\begin{align}\label{eq Interpolation22}
\norm{P_C P_{N_1} u_1 P_{N_2} u_2 P_{N_3} u_3}_{L_{t,x}^2(I \times \T^2)}  \lesssim (N_2 N_3)^{\frac{1}{2}} ( \frac{N_3}{N_1} + \frac{1}{N_2} )^{\delta^{\prime \prime}} \prod_{j=1}^3 \norm{P_{N_j} \phi_j}_{ L^2 (\T^2)}.
\end{align}
Let $\xi_0$ be the center of $C$. We partition $C = \cup R_k$ into almost disjoint strips of width $M = \max \{ \frac{N_2^2 }{N_1} ,1 \}$, which are orthogonal to $\xi_0$:
\begin{align*}
R_k = \{ \xi \in C : \xi \cdot \xi_0 \in [\abs{\xi_0} MK , \abs{\xi_0} M(K+1)) \}, \quad \abs{ k} \approx \frac{N_1}{M}.
\end{align*}
Each $R_k$ is the intersection of a cube of side length $2N_2$ with a strip of width $M$, and we decompose 
\begin{align*}
P_C P_{N_1} u_1 P_{N_2} u_2 P_{N_3} u_3 = \sum_k P_{R_k} P_{N_1} u_1 P_{N_2} u_2 P_{N_3} u_3
\end{align*}
and claim that the summands are almost orthogonal in $L^2 (\T \times \T^2)$. This orthogonality no longer comes from the spatial frequencies, but from the time frequency. Indeed, for $\xi_1 \in R_k$ we have
\begin{align*}
\xi_1^2 = \frac{1}{\abs{\xi_0}^2} \abs{\xi_1 \cdot \xi_0}^2 + \abs{\xi_1 - \xi_0}^2 - \frac{1}{\abs{\xi_0}^2}  \abs{(\xi_1 - \xi_0) \cdot \xi_0}^2 = M^2 K^2 + O(M^2 K)
\end{align*}
since $N_2^2 \lesssim M^2 K$. The second and third factors alter the time frequency by at most $O(N_2^2)$. Hence the expressions $P_{R_k} P_{N_1} u_1 P_{N_2} u_2 P_{N_3} u_3$ are localized at time frequency $M^2 K + O(M^k)$ and thus are almost orthogonal:
\begin{align*}
\norm{P_{C} P_{N_1} u_1 P_{N_2} u_2 P_{N_3} u_3}_{L^2(I \times \T^2)}^2 \lesssim \sum_k \norm{P_{R_k} P_{N_1} u_1 P_{N_2} u_2 P_{N_3} u_3}_{L^2(I \times \T^2)}^2.
\end{align*}

On the other hand, Strichartz estimates yields
\begin{align*}
\norm{P_{R_k} P_{N_1} u_1 P_{N_2} u_2 P_{N_3} u_3}_{L^2(I \times \T^2)} \lesssim N_2^{2-\frac{8}{p}} N_3^{1-\frac{4}{p}} (\frac{M}{N_2})^{\varepsilon} \norm{P_{R_k} P_{N_1} \phi_1}_{L^2 (\T^2)} \norm{P_{N_2} \phi_2}_{L^2 (\T^2)} \norm{ P_{N_3} \phi_3}_{L^2 (\T^2)},
\end{align*}
with $4 < p <5$ and $q$ such that $\frac{2}{p} + \frac{1}{q} =\frac{1}{2}$ and $0 < \varepsilon < \frac{1}{2} -\frac{2}{p}$. Then \eqref{eq Interpolation22} follows by summing up the squares with respect to $k$. 

We finish the proof of Proposition \ref{prop Loc Trilinear}.
\end{proof}

Before stating the trilinear estimates, we first introduce 
\begin{align*}
\norm{u}_{Z' (I)} = \norm{u}_{Z(I)}^{\frac{1}{2}} \norm{u}_{X^{\frac{1}{2}} (I)}^{\frac{1}{2}}.
\end{align*}

\begin{lem}\label{lem Trilinear}
For any $N_1, N_2, N_3 \geq 1$ and any interval $I \subset [0, 2\pi]$ we have
\begin{align}\label{eq:trilinear2}
\norm{  P_{N_1} u_1  P_{N_2} u_2 P_{N_3} u_3  }_{L_{t,x}^2 (I \times \T^2)}  \lesssim   \norm{P_{N_1}u_1}_{Y^0 (I)} \norm{P_{N_2}u_2}_{Z  (I)} \norm{P_{N_3}u_3}_{Z  (I)}.
\end{align}
\end{lem}

\begin{proof}[Proof of Lemma \ref{lem Trilinear}]
This proof is adapted from Lemma 3.1 in \cite{IP1}. 
Without loss of generality, we assume that $N_2\geq N_3$. 
Then we only need to consider two cases: 
\begin{enumerate}
\item
$N_1\geq N_2\geq N_3$,
\item
$N_2 \geq N_1$ and $N_2\geq N_3$.
\end{enumerate}
First in the first case  ($N_1\geq N_2\geq N_3$), in order to prove \eqref{eq:trilinear2} we observe that if $\{ C_k \}_{k \in \Z}$ is a partition of $\Z^2$ in cubes of size $N_2$ then the functions $(P_{C_k} u_1) u_2 u_3$ are almost orthogonal in $L_x^2$. By the almost orthogonality, H\"older's inequality and Strichartz estimates (Lemma \ref{lem Strichartz2}), we have 
\begin{align*}
& \quad \norm{P_{N_1} u_1 P_{N_2} u_2 P_{N_3} u_3}_{L_{t,x}^2(I \times \T^2)}^2 \\
& \lesssim \sum_k \norm{(P_{C_k} P_{N_1} u_1) P_{N_2} u_2 P_{N_3} u_3}_{L_{t,x}^2(I \times \T^2)}^2\\
& \lesssim \sum_k \norm{P_{C_k} P_{N_1} u_1}_{L_{t,x}^{p_0}(I \times \T^2)}^2 \norm{P_{N_2} u_2}_{L_{t,x}^{p_0}(I \times \T^2)}^2  \norm{P_{N_3} u_3}_{L_{t,x}^{\frac{2p_0}{p_0-4}}(I \times \T^2)}^2 \\
&\lesssim  \parenthese{\frac{N_3}{N_2}}^{\frac{16}{p_0}-3} \cdot\norm{P_{N_2} u_2}_{Z(I)}^2  \norm{P_{N_3} u_3}_{Z(I)}^2 \sum_k \norm{P_{C_k} P_{N_1} u_1}_{U_{\Delta}^{p_0} (I, L^2)}^2\\
& \lesssim  \norm{P_{N_2} u_2}_{Z(I)}^2  \norm{P_{N_3} u_3}_{Z(I)}^2 \sum_k \norm{P_{C_k} P_{N_1} u_1}_{Y^0 (I)}^2,
\end{align*}
where $p_0 = 4.1$, $\frac{2p_0}{p_0-4} = p_1 =82$ as defined in Definition \ref{defn Z}. Note that in the last inequality above, we used the embedding  $Y^0 (I) \hookrightarrow U_{\Delta}^{p_0}(I , L^2)$ and the fact that the $Y^0$-norm is square-summable. Now we finish the calculation of \eqref{eq:trilinear2} when  $N_1\geq N_2\geq N_3$.

In the second case ($N_2\geq N_1$ and $ N_2\geq N_3$), by H\"older's inequality and Strichartz estimates (Lemma \ref{lem Strichartz2}), we have 
\begin{align*}
\norm{P_{N_1} u_1 P_{N_2} u_2 P_{N_3} u_3}_{L_{t,x}^2(I \times \T^2)}^2 
& \lesssim \norm{P_{N_1} u_1}_{L_{t,x}^{p_0}(I \times \T^2)}^2 \norm{P_{N_2} u_2}_{L_{t,x}^{p_0}(I \times \T^2)}^2  \norm{P_{N_3} u_3}_{L_{t,x}^{\frac{2p_0}{p_0-4}}(I \times \T^2)}^2 \\
&\lesssim  \parenthese{\frac{N_1}{N_2}}^{2-\frac{8}{p_0}}
\parenthese{\frac{N_3}{N_2}}^{\frac{16}{p_0}-3}  \norm{P_{N_2} u_2}_{Z(I)}^2  \norm{P_{N_3 } u_3}_{Z(I)}^2 \norm{P_{N_1} u_1}_{U_{\Delta}^{p_0} (I, L^2)}^2\\
& \lesssim  \norm{P_{N_2} u_2}_{Z(I)}^2  \norm{P_{N_3} u_3}_{Z(I)}^2 \norm{P_{N_1} u_1}_{Y^0 (I)}^2.
\end{align*}
Similarly as in the first case, we obtain  \eqref{eq:trilinear2}  for the second case.
\end{proof}

\begin{rmk}\label{rmk Trilinear}
Interpolating Proposition  \ref{prop Loc Trilinear} and Lemma \ref{lem Trilinear}, we have that if $N_1 \geq N_2 \geq N_3 \geq 1$
\begin{align*}
\norm{ P_{N_1} u_1  P_{N_2} u_2  P_{N_3} u_3}_{L_{t,x}^2 (I \times \T^2)} & \lesssim (N_2 N_3)^{\frac{1}{2}} \bigg( \frac{N_3}{N_1} + \frac{1}{N_2} \bigg)^{\delta'} \norm{P_{N_1}u_1}_{Y^0 (I)} \norm{P_{N_2}u_2}_{Z  (I)} \norm{P_{N_3}u_3}_{Z  (I)} , \\
\norm{ P_{N_1} u_1  P_{N_2} u_2  P_{N_3} u_3}_{L_{t,x}^2 (I \times \T^2)}  &  \lesssim (N_2 N_3)^{\frac{1}{2}} \bigg( \frac{N_3}{N_1} + \frac{1}{N_2} \bigg)^{\delta'}\norm{P_{N_1}u_1}_{Z (I)} \norm{P_{N_2}u_2}_{Y^0  (I)} \norm{P_{N_3}u_3}_{Z  (I)},
\end{align*}
where $0 < \delta' < \delta$.
\end{rmk}

\begin{lem}[Nonlinear estimates]\label{lem Nonlinear est}
For $u_k \in X^{\frac{1}{2}} (I)$, $k=1,2,\cdots, 5$, $\abs{I} \leq 1$, the estimate
\begin{align*}
\norm{\prod_{k=1}^5 \widetilde{u}_k}_{N(I)} \lesssim \sum_{j=1}^5 \norm{u_j}_{X^{\frac{1}{2}} (I)} \prod_{\substack{k\in\{1,2,\cdots, 5\} \\k\neq j}} \norm{u_k}_{Z' (I)}
\end{align*}
holds, where $\widetilde{u}_k \in \{ u_k , \overline{u}_k \}$.
Furthermore, if there exist constants $A,B>0$, such that $u_k = P_{>A} u_k$ for $k=1,\cdots, 4$ and $u_5=P_{<B} u_5$, then  
\begin{equation}\label{lem3.8 2}
\norm{\prod_{k=1}^5 \widetilde{u}_k}_{N(I)} \lesssim\norm{u_5}_{Z'(I)}  \bigg( \sum_{j=1}^4 \norm{u_j}_{X^{\frac{1}{2}} (I)} \prod_{\substack{k\in\{1, 2,3, 4\} \\k\neq j}} \norm{u_k}_{Z' (I)}\bigg).
\end{equation}
\end{lem}

\begin{proof}[Proof of Lemma \ref{lem Nonlinear est}]
This proof is adapted from Lemma 3.2  in \cite{IP1} and Lemma 3.2 in \cite{IP2}.

Let $N \geq 1$ be given. By Proposition \ref{prop Duality X}, we have that
\begin{align*}
\norm{P_{\leq N} \prod_{i=1}^5 \widetilde{u}_k}_{N(I)} \leq \sup_{  v \in Y^{- \frac{1}{2} } (I) ;  \, \norm{v}_{Y^{-\frac{1}{2}}} =1} \int_0^{2 \pi} \int_{\T^2} P_{\leq N} \prod_{k=1}^5 \widetilde{u}_k \overline{v} \, dxdt .
\end{align*}
We denote $u_0 = P_{\leq N} v$. Then we need to prove the multilinear estimate
\begin{align}\label{eq nonlinear}
\abs{\int_{I \times \T^2} \prod_{k=0}^5  \wt{u}_k  \, dxdt } \lesssim \norm{u_0}_{Y^{-\frac{1}{2}} (I)}   \norm{u_1}_{X^{\frac{1}{2}} (I)} \prod_{k=2}^5 \norm{u_k}_{Z' (I)}
\end{align}
Once \eqref{eq nonlinear} is obtained, the claimed estimate follows by letting $N \to \infty$. We consider extensions to $\R$ of $u_k$, which we also denote with $u_k$ in the sequel of the proof, $k= 0, \cdots,5$ and \eqref{eq nonlinear}  reduces to 
\begin{align*}
\abs{\int_{I \times \T^2} \prod_{k=0}^5  \wt{u}_k  \, dxdt } \lesssim \norm{u_0}_{Y^{-\frac{1}{2}} (I)}   \norm{u_1}_{X^{\frac{1}{2}} (I)} \prod_{k=2}^5 \norm{u_k}_{Z' (I)}
\end{align*}
To prove this, we dyadically decompose 
\begin{align*}
\wt{u_k} = \sum_{N_k} P_{N_k} \wt{u}_k .
\end{align*}
and after relabeling, we need to estimate
\begin{align*}
S = \sum_{\underline{N}} \norm{P_{N_1} u_1 P_{N_3} u_3 P_{N_5} u_5}_{L_{t,x}^2 (I \times \T^2)} \norm{P_{N_0} u_0 P_{N_2} u_2 P_{N_4} u_4}_{L_{t,x}^2 (I \times \T^2)}
\end{align*}
where we denote by $\underline{N} = ( N_0, N_1,\cdots ,N_5)$ the $6$-tuple of $2^n$ numbers, $ n \in \N$ satisfying
\begin{align*}
N_1 \sim \max \{N_2, N_0 \} \geq N_2 \geq N_3 \geq N_4 \geq N_5 .
\end{align*}
Then we will consider the following two cases:
\begin{enumerate}[$-$]
\item
{\bf Case 1}: the $6$-tuple $\underline{N_1} = ( N_0, N_1,\cdots ,N_5)$ such that $N_1 \sim N_0 \gtrsim N_2$
\item
{\bf Case 2}: the $6$-tuple $\underline{N_2} = ( N_0, N_1,\cdots ,N_5)$ such that $N_1 \sim N_2 \gtrsim N_0$
\end{enumerate}

We first consider the case $N_0 \sim N_1$. 
By interpolating Remark \ref{rmk Trilinear}, for some $0<\delta'<\delta$ we have that
\begin{align*}
S_1 & \lesssim  \sum_{\underline{N_1}}  \bigg(\frac{N_5}{N_1} + \frac{1}{N_3}\bigg)^{\delta'} \bigg(\frac{N_4}{N_0} + \frac{1}{N_2}\bigg)^{\delta'} \norm{P_{N_0} u_0}_{Y^0 (I)} \norm{P_{N_1} u_1}_{Y^0 (I)} \prod_{j=2}^5 \norm{P_{N_j} u_j}_{Z'  (I)} \\
& \lesssim  \norm{u_0}_{Y^{-\frac{1}{2}} (I)} \norm{u_1}_{X^{\frac{1}{2}} (I)} \prod_{j=2}^5 \norm{u_j}_{Z'  (I)}.
\end{align*}

Then we consider the case $N_0 \leq N_2$. 
Similarly by interpolating Remark \ref{rmk Trilinear}, for some $0<\delta'<\delta$ we have that
\begin{align*}
S_2 & \lesssim  \sum_{\underline{N_2}}  \bigg(\frac{N_5}{N_1} + \frac{1}{N_3}\bigg)^{\delta'} \bigg(\frac{\min(N_0,N_4)}{N_2} + \frac{1}{\max(N_0, N_4)}\bigg)^{\delta'}\cdot \norm{P_{N_0} u_0}_{Y^0 (I)} \norm{P_{N_1} u_1}_{Y^0 (I)} \prod_{j=2}^5 \norm{P_{N_j} u_j}_{Z'  (I)} \\
& \lesssim  \norm{u_0}_{Y^{-\frac{1}{2}} (I)} \norm{u_1}_{X^{\frac{1}{2}} (I)} \prod_{j=2}^5 \norm{u_j}_{Z'  (I)}.
\end{align*}

In particular, if there exist constants $A,B>0$, such that $u_k = P_{>A} u_k$ for $k=1,2,3, 4$ and $u_5=P_{<B} u_5$, then without loss of generality, we only need to consider $N_1, N_2, N_3, N_4\gtrsim N_5$, so we can prove (\ref{lem3.8 2}) by using a similar way (but $u_5$ is in the low frequency position and hence $u_5$ is in $Z'$-norm for all cases).

Now we finish the proof of Lemma \ref{lem Nonlinear est}.
\end{proof}

\subsection{Local well-posedness}
\begin{prop}[Local well-posedness]\label{prop LWP}
\begin{enumerate}[\bf (1)]
\item
Given $E >0$, there exists $\delta_0 = \delta_0 (E) > 0$ such that if $\norm{\phi}_{H_x^{\frac{1}{2}} (\T^2)} \leq E$ and 
\begin{align*}
\norm{e^{it\Delta} \phi}_{Z'(I)} \leq \delta_0
\end{align*}
on some interval $I \ni 0$, $\abs{I} \leq 1$, then there exists a unique solution $u \in X^{\frac{1}{2}} (I)$ of \eqref{NLS} satisfying $u(0) = \phi$. Besides,
\begin{align*}
\norm{u - e^{it\Delta} \phi}_{X^{\frac{1}{2}}(I)} \lesssim _E \norm{e^{it\Delta} \phi}_{Z'(I)}^3 .
\end{align*}
The quantity $E(u) = \norm{u}_{L_t^{\infty} \dot{H}_x^{\frac{1}{2}} (I \times \T^2)}$ is bounded on $I$.

\item
If $u \in X^{\frac{1}{2}} (I)$ is a solution of \eqref{NLS} on some interval $I$ and 
\begin{align*}
\norm{u}_{Z(I)} < + \infty,
\end{align*}
then $u$ can be extended as a nonlinear solution to a neighborhood of $\bar{I}$ and
\begin{align*}
\norm{u}_{X^{\frac{1}{2}}(I)} \leq C(E(u), \norm{u}_{Z(I)})
\end{align*}
for some constant $C$ depending on $E(u)$ and $\norm{u}_{Z(I)}$.
\end{enumerate}
\end{prop}

\begin{proof}[Proof of Proposition \ref{prop LWP}]
\begin{enumerate}[\bf (1)]
\item
The proof in {\bf Part (1)} is adapted from Proposition 3.2 in \cite{IP2}.
We proceed by a standard fixed point argument. For $E ,a > 0$, we consider the space
\begin{align*}
S = \{ u \in X^{\frac{1}{2}} (I): \norm{u}_{X^{\frac{1}{2}} (I)} \leq 2 E , \norm{u}_{Z' (I)} \leq a \}
\end{align*}
and the mapping
\begin{align*}
\Phi (v) := e^{it\Delta} u_0 - i \int_0^t e^{i (t-s) \Delta} \abs{v}^4 v (s) \, ds .
\end{align*}
We then see that for $u,v \in S$, by Lemma \ref{lem Nonlinear est}
\begin{align*}
\norm{\Phi (u) - \Phi(v)}_{X^{\frac{1}{2}} (I)} & \lesssim \parenthese{ \norm{u}_{X^{\frac{1}{2}} (I)} + \norm{v}_{X^{\frac{1}{2}} (I)} } \parenthese{ \norm{u}_{Z'(I)} + \norm{v}_{Z' (I)} }^3 \norm{u-v}_{X^{\frac{1}{2}} (I)} \\
& \lesssim E a^3 \norm{u-v}_{X^{\frac{1}{2}} (I)}  .
\end{align*}
Similarly, using Lemma \ref{lem Nonlinear est}, we also obtain that
\begin{align*}
\norm{\Phi (u)}_{X^{\frac{1}{2}} (I)} & \lesssim \norm{\Phi (0)}_{X^{\frac{1}{2}} (I)} + C \norm{\Phi (u) - \Phi(0)}_{X^{\frac{1}{2}} (I)}  \lesssim \norm{u_0}_{H^{\frac{1}{2}}} + CE a^4 ,\\
\norm{\Phi (u)}_{Z' (I)} & \lesssim \norm{\Phi (0)}_{Z' (I)} + C \norm{\Phi (u) - \Phi(0)}_{X^{\frac{1}{2}} (I)}  \lesssim \delta + CE a^4  .
\end{align*}
Now, we choose $a = 2 \delta$ and we let $\delta_0 = \delta_0 (E)$ be small enough. We then see that $\Phi$ is a contraction on $S$. Hence it possesses a unique fixed point $u$. Finally, using Lemma \ref{lem Trilinear} and taking $\delta_0$ smaller if necessary,
\begin{align*}
\norm{u - e^{it\Delta} u_0}_{X^{\frac{1}{2}} (I)} \lesssim \norm{u}_{X^{\frac{1}{2}}(I) }  \norm{u}_{Z'(I)}^4 \leq E \delta^4 \leq \delta^3 .
\end{align*}
In order to justify the uniqueness in $X^{\frac{1}{2}} (I)$, assume that $u,v \in X^{\frac{1}{2}}$ satisfy $u(0) =v(0)$ and choose $E = \max \{ \norm{u}_{X^{\frac{1}{2}} (I)} , \norm{v}_{X^{\frac{1}{2}} (I)}  \}$. Then there exists a possibly smaller open interval $0 \in J$ such that $u, v \in S_J$. By uniqueness of the fixed point in $S$, $u|_J = v|_J$ and hence $\{ u=v \}$ is closed and open in $I$. This finishes the proof.

\item
The proof in  {\bf Part (2)} is adapted from Lemma 3.4 in \cite{IP2}.
It suffices to consider the case $I = (0,T)$. Choose $\varepsilon > 0 $ sufficiently small and $T_1 \geq T-1$ such that
\begin{align*}
\norm{u}_{Z'(T_1 , T)} \leq \varepsilon .
\end{align*}
Now, let
\begin{align*}
h(s) = \norm{e^{i(t -T_1)\Delta} u(T_1)}_{Z' (T_1 ,T_1 +s)} .
\end{align*}
Clearly, $h$ is continuous function of $s$ satisfying $h(0) =0$. In addition, using {\bf Part (1)}, as long as $h(s) \leq \delta_0$, we have that
\begin{align*}
\norm{u(t) - e^{i (t- T_1)\Delta} u(T_1)}_{X^{\frac{1}{2}} (T_1 , T_1 +s)} \leq h(s)^3 ,
\end{align*}
and in particular, from Duhamel formula.
\begin{align*}
\norm{ e^{i (t- T_1)\Delta} u(T_1)}_{Z(T_1 , T_1 +s)} & \leq \norm{u}_{Z (T_1 , T_1 +s)} + C \norm{u(t) - e^{i (t- T_1)\Delta} u(T_1)}_{X^{\frac{1}{2}} (T_1 , T_1 +s)} \\
& \leq \varepsilon + C h(s)^3 .
\end{align*}
On the other hand, by definition, there holds that
\begin{align*}
h(s) \leq \norm{ e^{i (t- T_1)\Delta} u(T_1)}_{Z(T_1 , T_1 +s)}^{\frac{1}{2}} \norm{ e^{i (t- T_1)\Delta} u(T_1)}_{X^{\frac{1}{2}}(T_1 , T_1 +s)}^{\frac{1}{2}} \leq (\varepsilon + C h(s)^3)^{\frac{1}{2}} \norm{u(T_1)}_{H^{\frac{1}{2}}}^{\frac{1}{2}} .
\end{align*}
Now, if we take $\varepsilon >0$ sufficiently small, we see by continuity that
\begin{align*}
h(s) \leq C \sqrt{\varepsilon} E^{\frac{1}{2}} \leq \frac{1}{2} \delta_0 .
\end{align*}
for all $s \leq T- T_1$. Consequently, there exists a larger time $T_2 > T$ such that
\begin{align*}
\norm{e^{i(t -T_1)\Delta} u(T_1)}_{Z' (T_1 ,T_2)} \leq \frac{3}{4} \delta_0  .
\end{align*}
Applying {\bf (1)}, we see that $u$ can be extended to $(0, T_2)$. 
\end{enumerate}
Now we finish the proof of Proposition \ref{prop LWP}.
\end{proof}

\subsection{Stability}
\begin{prop}[Stability]\label{prop Stability}
Assume that $I$ is an open bounded interval, $\rho \in \{0, 1\}$, and $\widetilde{u} \in X^{\frac{1}{2}} (I)$ satisfies the approximate Schr\"odinger equation
\begin{align}\label{eq Approx NLS}
(i\partial_t + \Delta_{\T^2}) \widetilde{u} = \rho \abs{\widetilde{u}}^4 \widetilde{u} + e \quad \text{ on } I \times \T^2.
\end{align}
Assume in addition that
\begin{align*}
\norm{\widetilde{u}}_{Z(I)} + \norm{\widetilde{u}}_{L_t^{\infty} (I , H^{\frac{1}{2}} (\T^2))} 
\leq M ,
\end{align*}
for some $M \in [0,\infty)$. Assume that $t_0 \in I$, and assume that $u_0 \in H^{\frac{1}{2}}(\T^2)$ such that the smallness condition
\begin{align*}
\norm{u_0 -\widetilde{u}(t_0)}_{H^{\frac{1}{2}} (\T^2)} + \norm{e}_{N(I)} \leq \varepsilon
\end{align*}
holds for some $0 < \varepsilon < \varepsilon_1$, where $\varepsilon_1 \leq 1$ is a small constant $\varepsilon_1 = \varepsilon_1 (M) >0$.

Then there exists a strong solution $u \in X^{\frac{1}{2}} (I)$ of the Schr\"odinger equation
\begin{align}\label{eq NLS}
(i\partial_t + \Delta_{\T^2}) u = \rho \abs{u}^4 u ,
\end{align}
such that $u(t_0) = u_0$ and
\begin{align*}
\norm{u}_{X^{\frac{1}{2}}(I)} + \norm{\widetilde{u}}_{X^{\frac{1}{2}} (I)} \leq C(M) ,\\
\norm{u-\widetilde{u}}_{X^{\frac{1}{2}}(I)} \leq C(M)\varepsilon .
\end{align*}

\end{prop}

\begin{proof}[Proof of Proposition \ref{prop Stability}]
This proof is adapted from Proposition 3.4 in \cite{IP1} and Proposition 3.5 in \cite{IP2}.

Without loss of generality, we may assume that $\abs{I} \leq 1$. We proceed in several steps.
First we need the following variation on the local existence result. We claim that there exists $\delta_1 = \delta_1 (M)$ such that if, for some interval $J \ni t_0$,
\begin{align*}
\norm{e^{i (t- t_0)\Delta} \wt{u}(t_0)}_{Z' (J)} + \norm{e}_{N(J)} \leq \delta_1,
\end{align*} 
then there exists a unique solution $\wt{u}$ of \eqref{eq Approx NLS} on $J$ and
\begin{align*}
\norm{\wt{u} - e^{i(t-t_0) \Delta} \wt{u} (t_0)}_{X^{\frac{1}{2}} (J)} \leq \norm{e^{i(t-t_0) \Delta} \wt{u} (t_0)}_{Z'(J)}^3 + 2 \norm{e}_{N(J)}.
\end{align*}
The proof is very similar to the proof of {\bf Part (1)} in Proposition \ref{prop LWP}.

Now we claim that there exists $\varepsilon_1 = \varepsilon_1 (M)$ such that if the inequalities 
\begin{align}\label{eq Small1}
\begin{aligned}
& \norm{e}_{N(I_k)} \leq \varepsilon_1 , \\
& \norm{\wt{u}}_{Z(I_k)} \leq \varepsilon \leq \varepsilon_1 ,
\end{aligned}
\end{align}
hold on $I_k = ( T_k ,T_{k+1} )$, then
\begin{align}\label{eq Claim1}
\begin{aligned}
& \norm{e^{i(t-t_0) \Delta} \wt{u} (T_k)}_{Z' (I_k)} \leq C (1 +M ) ( \varepsilon + \norm{e}_{N(I_k)})^{\frac{1}{2}} ,\\
&\norm{\wt{u}}_{Z' (I_k)} \leq C (1 +M ) ( \varepsilon + \norm{e}_{N(I_k)})^{\frac{1}{2}} .
\end{aligned}
\end{align}
Define $h(s) = \norm{e^{i(t-T_k) \Delta} \wt{u}(T_k)}_{Z' (T_k ,T_k +s)}$. Let $J_k = [ T_k, T') \subset I_k$ be the largest interval such that $h(s) \leq \frac{1}{2} \delta_1$, where $\delta_1 = \delta_1 (M)$ is as in the claim above. Then on the one hand, we see from Duhamel formula that
\begin{align*}
\norm{e^{i(t - T_k) \Delta} \wt{u}(T_k)}_{Z(T_k ,T_k + s)} & \leq \norm{\wt{u}}_{Z(T_k, T_k +s)} +  \norm{\wt{u} - e^{i(t-T_k) \Delta} \wt{u}(T_k)}_{X^{\frac{1}{2}} (T_k, T_k +s)} \\
& \leq \varepsilon + h(s)^3 + 2 \norm{e}_{N(I_k)}.
\end{align*}
On the other hand, we also have that
\begin{align*}
h(s) & \leq \norm{e^{i(t - T_k) \Delta} \wt{u}(T_k)}_{Z(T_k ,T_k + s)}^{\frac{1}{2}} \norm{e^{i(t - T_k) \Delta} \wt{u}(T_k)}_{X^{\frac{1}{2}} (T_k ,T_k + s)}^{\frac{1}{2}} \\
& \leq \parenthese{\varepsilon + h(s)^3 + 2 \norm{e}_{N(I_k)}}^{\frac{1}{2}} M^{\frac{1}{2}} \\
& \leq C (1 + M) (\varepsilon + \norm{e}_{N(I_k)})^{\frac{1}{2}}) + C(1+M)h(s)^{\frac{3}{2}} .
\end{align*}
The claim \eqref{eq Claim1} follows provided that $\varepsilon_1$ is chosen small enough.

Now we consider an interval $I_k = (T_k ,T_{k+1})$ on which we assume that
\begin{align}\label{eq Small2}
\begin{aligned}
&\norm{e^{i(t -T_k) \Delta}\wt{u} (T_k)}_{Z' (I_k)} \leq \varepsilon \leq \varepsilon_0 ,\\
&\norm{\wt{u}}_{Z' (I_k)} \leq \varepsilon \leq \varepsilon_0 ,\\
&\norm{e}_{N(I_k)} \leq \varepsilon_0 . 
\end{aligned}
\end{align}
for some constant $\varepsilon_0$ sufficiently small. We can control $\wt{u}$ on $I_k$ as follows
\begin{align*}
\norm{\wt{u}}_{X^{\frac{1}{2}}(I_k)} \leq \norm{e^{i(t-T_k) \Delta}\wt{u} (T_k)}_{X^{\frac{1}{2}}(I_k)}  + \norm{\wt{u} - e^{i(t-T_k) \Delta}\wt{u} (T_k)}_{X^{\frac{1}{2}}(I_k)}  \leq M+1 .
\end{align*}
Now, we let $u$ be an exact strong solution of \eqref{eq NLS} defined on an interval $J_u$ such that
\begin{align*}
a_k = \norm{\wt{u} (T_k) -u(T_k)}_{H^{\frac{1}{2}} (\T^2)} \leq \varepsilon_0 ,
\end{align*}
and we let $J_k = [T_k ,T_k +s] \cap I_k \cap J_u$ be the maximal interval such that
\begin{align}\label{eq omega1}
\norm{\omega}_{Z'(J_k)} \leq 10 C \varepsilon_0 \leq \frac{1}{10 (M+1)} ,
\end{align}
where $\omega : = u - \wt{u}$. Such an interval exists and is nonempty since $s \mapsto \norm{\omega}_{Z'(T_k , T_k +s)} $ is finite and continuous on $J_u$ and vanishes for $s=0$. Then we see that $\omega$ solves
\begin{align*}
(i \partial_t  + \Delta ) \omega = \rho ( \abs{\wt{u} +\omega}^4 (\wt{u} +\omega) - \abs{\wt{u}}^4 \wt{u})-e ,
\end{align*}
and consequently, using Lemma \ref{lem Nonlinear est}, we get that
\begin{align}\label{eq omega2}
\norm{\omega}_{X^{\frac{1}{2}} (J_k)} & \leq \norm{e^{i(t-T_k)\Delta} (u(T_k) - \wt{u}(T_k))}_{X^{\frac{1}{2}} (J_k)} + \norm{ \abs{\wt{u} +\omega}^4 (\wt{u} +\omega) - \abs{\wt{u}}^4 \wt{u}}_{N(J_k)} + \norm{e}_{N(J_k)} \notag\\
& \leq \norm{u(T_k) - \wt{u}(T_k)}_{H^{\frac{1}{2}}} + C \norm{\omega}_{X^{\frac{1}{2}}(J_k)} ( \norm{\wt{u}}_{X^{\frac{1}{2}}(J_k)} \norm{\wt{u}}_{Z'(J_k)}^3 +  \norm{\omega}_{X^{\frac{1}{2}}(J_k)} \norm{\omega}_{Z'(J_k)}^3) + \norm{e}_{N(J_k)} \notag\\
& \leq \norm{u(T_k) - \wt{u}(T_k)}_{H^{\frac{1}{2}}} + C \varepsilon_0  \norm{\omega}_{X^{\frac{1}{2}}(J_k)} + \norm{e}_{N(J_k)} .
\end{align}
where the last line holds due to \eqref{eq omega1}. Consequently,
\begin{align*}
\norm{\omega}_{Z'(J_k)} \leq C \norm{\omega}_{X^{\frac{1}{2}}(J_k)} \leq 4 C  (\norm{u(T_k) - \wt{u}(T_k)}_{H^{\frac{1}{2}}} + \norm{e}_{N(J_k)}) \leq 8 C \varepsilon_0 .
\end{align*}
provided that $\varepsilon_0$ is small enough. Consequently $J_k = I_k \cap J_u$ and \eqref{eq omega2} holds on $I_k \cap J_u$. Therefore, it follows from (ii) in Proposition \ref{prop LWP} that the solution $u$ can be extended to the entire interval $I_k$, and the bounds  \eqref{eq omega1} and  \eqref{eq omega2} hold with $J_k = I_k$.

Now we can finish the proof. We take $\varepsilon_2 (M ) \leq \varepsilon_1 (M)$ sufficiently small and split $I$ into $N = O (\norm{\wt{u}}_{Z(I)} / \varepsilon_2)^6$ intervals such that
\begin{align*}
&\norm{\wt{u}}_{Z(I_k)} \leq \varepsilon_2 ,\\
&\norm{e}_{N(I_k)} \leq \kappa \varepsilon_2 .
\end{align*}
Then, on each interval, we have that \eqref{eq Small1} holds, so that we also have from \eqref{eq Claim1} that \eqref{eq Small2} holds. As a consequence, the bounds  \eqref{eq omega1} and  \eqref{eq omega2} both hold on each interval $I_k$. The conclusion of  Proposition \ref{prop Stability} follows.
\end{proof}

\section{Euclidean profiles}\label{sec Euclidean profiles}
In this section, we  will study the behaviors of  Euclidean profiles and their corresponding  nonlinear profiles (obtained by evolving the Euclidean profiles along the NLS flow as initial data).

We fix a spherically symmetric function $\eta \in C_0^{\infty} (\R^2)$ supported in the ball of radius $2$ and equal to $1$ in the ball of radius $1$. Given $\phi \in \dot{H}^{\frac{1}{2}} (\R^2)$ and a real number $N \geq 1$ we define
\begin{align}\label{eq Euclidean profile}
\begin{aligned}
Q_N \phi \in H^{\frac{1}{2}} (\R^2) , &\qquad (Q_N \phi)(x) = \eta(\frac{x}{N^{\frac{1}{2}}}) \phi (x),\\
\phi_N   \in H^{\frac{1}{2}} (\R^2),  & \qquad \phi_N (x)  = N^{\frac{1}{2}} (Q_N \phi) (Nx) ,\\
f_N  \in H^{\frac{1}{2}}(\T^2) , & \qquad f_N (y)  = \phi_N (\Psi^{-1} (y)),
\end{aligned}
\end{align} 
where $\Psi : \{ x \in \R^2 : \abs{x} < 1 \} \to O_0 \subset \T^2$, $\Psi (x) =x$. Thus $Q_N \phi$ is a compactly supported modification of the profile $\phi$, $\phi_N$ is an $\dot{H}^{\frac{1}{2}}$-invariant rescaling of $Q_N \phi$, and $f_N$ is the function obtained by transferring $\phi_N$ to a neighborhood of zero in $\T^2$. 

Recall that we denote
\begin{align*}
E_{\R^2} (u) = \norm{u}_{L_t^{\infty}\dot{H}_x^{\frac{1}{2}} }.
\end{align*}
Note the this $H^{\frac{1}{2}}$-norm is not the energy nor a conservation quantity in our setting. We just adopt the notation $E$ from the energy-critical scenario.

We use the main theorem of \cite{Yu} in the following form.
\begin{thm}\label{thm Yu}
Assume that $\psi \in \dot{H}^{\frac{1}{2}}(\R^2)$. Then there is a unique global solution $v \in C(\R : \dot{H}^{\frac{1}{2}} (\R^2))$ of the initial-value problem
\begin{align*}
\begin{cases}
(i\partial_t + \Delta_{\R^2}) v = \abs{v}^4 v, \\
 v(0) =\psi ,
\end{cases}
\end{align*}
and 
\begin{align*}
\norm{\abs{\nabla_{\R^2}}^{\frac{1}{2}} v}_{L_t^{q} L_x^r (\R \times \R^2)} \leq \widetilde{C} (E_{\R^2} (\psi)),
\end{align*}
where $(q,r)$ is any Strichartz admissible pair{\footnote{By {\it Strichartz admissible pair}, we mean the pair $(q,r)$ satisfying $\frac{1}{q} + \frac{1}{r} = \frac{1}{2}$ and $q \in (2,\infty]$, $r \in [2, \infty)$.}}.

Moreover, this solution scatters in the sense that there exists $\psi^{\pm \infty} \in \dot{H}^{\frac{1}{2}} (\R^2)$ such that 
\begin{align*}
\lim_{t \to \pm \infty} \norm{v(t) -e^{it\Delta_{\R^2}} \psi^{\pm \infty}}_{\dot{H}^{\frac{1}{2}} (\R^2)} = 0.
\end{align*}
Besides, if $\psi \in H^5(\R^2)$, then $v \in C(\R: H^5 (\R^2))$ and 
\begin{align*}
\sup_{t \in \R} \norm{v(t)}_{H^5(\R^2)} \lesssim_{\norm{\psi}_{H^5 (\R^2)}} 1.
\end{align*}
\end{thm}

\begin{lem}\label{lem Compare Profile}
Assume that $\phi \in \dot{H}^{\frac{1}{2}} (\R^2)$, $T_0 \in (0, \infty)$, and $\rho \in \{ 0, 1\}$ are given, and define $f_N$ as in \eqref{eq Euclidean profile}. Then the following conclusions hold:
\begin{enumerate}[\bf (1)]
\item
There is $N_0 = N_0(\phi, T_0)$ sufficiently large such that for any $N \geq N_0$ there is a unique solution $U_N \in C ( (-T_0 N^{-2} , T_0 N^{-2}) : H^{\frac{1}{2}} (\T^2) )$ of the initial-value problem
\begin{align*}
(i\partial_t + \Delta_{\T^2} ) U_N = \rho \abs{U_N}^4 U_N, \quad U_N(0) = f_N.
\end{align*}
Moreover, for any $N \geq N_0$,
\begin{align*}
\norm{U_N}_{X^{\frac{1}{2}} ((-T_0 N^{-2} , T_0 N^{-2}))} \lesssim_{E_{\R^2} (\phi)} 1.
\end{align*}

\item
Assume that $\varepsilon_1 \in ( 0,1]$ is sufficiently small (depending only on $E_{\R^2} (\phi)$), $\phi'  \in H^5 (\R^2)$, and $\norm{\phi - \phi' }_{\dot{H}^{\frac{1}{2}} (\R^2)} \leq \varepsilon_1$. Let $v'  \in C(\R : H^5)$ denote the solution of the initial-value problem
\begin{align*}
(i\partial_t + \Delta_{\R^2}) v'  = \rho \abs{v' }^4 v' , \quad v' (0)= \phi' .
\end{align*}
For $R, N \geq 1$ we define
\begin{align*}
\begin{aligned}
v_R'  (t,x)  = \eta(\frac{x}{R}) v' (t,x), &\qquad  (t,x) \in (-T_0 , T_0) \times \R^2 ,\\
v_{R, N}'  (t,x) = N^{\frac{1}{2}} v_R' (Nx, N^2 t), & \qquad (t,x) \in (-T_0N^{-2} , T_0 N^{-2}) \times \R^2, \\
V_{R,N} (t,y)  = v_{R,N}'  (t , \Psi^{-1} (y) ), &\qquad  (t,y) \in  (-T_0N^{-2} , T_0 N^{-2}) \times \R^2.
\end{aligned}
\end{align*}
Then there is $R_0 \geq 1$ (depending on $T_0, \phi' $ and $\varepsilon_1$) such that, for any $R \geq R_0$,
\begin{align*}
\limsup_{N \to \infty} \norm{U_N - V_{R,N}}_{X^{\frac{1}{2}} ((-T_0 N^{-2} , T_0 N^{-2}))} \lesssim_{E_{\R^2} (\phi)} \varepsilon_1 .
\end{align*}
\end{enumerate}

\end{lem}

\begin{rmk}\label{rmk U,V}
Notice that there are two types of solutions in this lemma. In fact, we can think of $U_{R,N}$ as the solution of \eqref{NLS} with the cutoff and rescaled initial data while $V_{R,N}$ can be considered as the solution of \eqref{NLS} being cutoff and rescaled ($V_{R,N}$ is not a solution to \eqref{NLS}, but solves an approximate NLS).
\end{rmk}

\begin{proof}[Proof of Lemma \ref{lem Compare Profile}]
This proof is adapted from Lemma 4.2 in \cite{IP1} and Lemma 4.2 in \cite{IPS}.

The proof of this lemma is based on Proposition \ref{prop Stability}. By Theorem \ref{thm Yu}, we have $v'$ exists globally ans satisfies the following two inequalities
\begin{align}\label{eq H^5}
\begin{aligned}
\norm{ \abs{\nabla_{\R^2}}^{\frac{1}{2}} v'}_{L_t^{q} L_x^r  (\R \times \R^2)} \lesssim 1\\
\sup_{t \in \R} \norm{v' (t)}_{H^5 (\R^2)} \lesssim_{ \norm{\phi'}_{H^5 (\R^2)}} 1 ,
\end{aligned}
\end{align}
for  $(q,r)$  any Strichartz admissible pair.

We first consider $v'_{R} (t,x) = \eta(\frac{x}{R}) v'(t,x)$. By plugging $v'_{R} (t,x) $ into \eqref{NLS}, we see that $v'_{R} (t,x)$ solves the following approximate equation
\begin{align*}
(i \partial_t  + \Delta_{\R^2}) v'_{R} = \rho \abs{v'_{R}}^4 v'_{R} + e_R (t,x) ,
\end{align*}
where
\begin{align*}
e_R (t,x) = \rho (\eta (\frac{x}{R}) - \eta^5 (\frac{x}{R}) ) \abs{v'}^4 v' + R^{-2} v'(t,x) (\Delta_{\R^2} \eta) (\frac{x}{R}) + 2R^{-1} \sum_{j=1}^2 \partial_j v'(t,x) \partial_j \eta (\frac{x}{R}) .
\end{align*}
Then $v'_{R,N} (t,x) $ solves
\begin{align*}
(i \partial_t  + \Delta_{\R^2}) v'_{R,N} = \rho \abs{v'_{R,N}}^4 v'_{R,N} + e_{R,N} (t,x) ,
\end{align*}
where
\begin{align*}
e_{R,N}(t,x) = N^{\frac{5}{2}} e_R (N^2t, Nx) .
\end{align*}
Finally, as we explained in Remark \ref{rmk U,V}, $V_{R,N} (t,y) = v'_{R,N} (t, \Psi^{-1} (y))$ with $N \geq 10 R$ solves
\begin{align*}
(i \partial_t  + \Delta_{\R^2}) V_{R,N} (t,y)= \rho \abs{V_{R,N}}^4  V_{R,N} + E_{R,N} (t,y) ,
\end{align*}
where
\begin{align*}
E_{R,N} (t,y) = e_{R,N} (t, \Psi^{-1} (y)).
\end{align*}
To be able to apply Proposition \ref{prop Stability}, we need to check first the following three conditions in Proposition \ref{prop Stability}.
\begin{enumerate}[\bf (1)]
\item
$\norm{V_{R,N}}_{L_t^{\infty} H_x^{\frac{1}{2}} ([-T_0 N^{-2},T_0 N^{-2}] \times \T^2)} + \norm{V_{R,N}}_{Z([-T_0 N^{-2},T_0 N^{-2}])} \lesssim 1$,
\item
$\norm{f_N - V_{R,N}(0)}_{H_x^{\frac{1}{2}}} \leq \varepsilon$,
\item
$\norm{E_{R,N}}_{N([-T_0 N^{-2}, T_0 N^{-2}])} \leq \varepsilon$.
\end{enumerate}

Now we check the conditions listed above one by one.
\begin{enumerate}[\bf (1)]
\item
First, we know that $V_{R,N}(t,y)$ exists globally, thanks to the global existence of $v'(t,x)$. 
\begin{align*}
& \quad \sup_{t \in [-T_0 N^{-2} , T_0 N^{-2}]} \norm{V_{R,N} (t)}_{H_x^{\frac{1}{2}} (\T^2)} \leq \sup_{t \in [-T_0 N^{-2} , T_0 N^{-2}]} \norm{v'_{R,N} (t)}_{H_x^{\frac{1}{2}} (\T^2)} \\
& = \sup_{t \in [-T_0 N^{-2} , T_0 N^{-2}]} \norm{N^{\frac{1}{2}} v'_R (N^2t, Nx)}_{H_x^{\frac{1}{2}} (\R^2)} \\
& = \sup_{t \in [-T_0 N^{-2} , T_0 N^{-2}]} \frac{1}{N^{\frac{1}{2}}} \norm{v'_R (N^2t)}_{L_x^2 (\R^2)} + \norm{v'_R (N^2 t)}_{\dot{H}_x^{\frac{1}{2}} (\R^2)} \\
& \leq \sup_{t \in [-T_0 , T_0 ]} \norm{v'_R}_{H_x^{\frac{1}{2}} (\R^2)} = \sup_{t \in [-T_0 , T_0 ]}  \norm{\eta (\frac{x}{R}) v'(t,x)}_{H_x^{\frac{1}{2}} (\R^2)}\\
& \leq \sup_{t \in [-T_0 , T_0 ]} \norm{\eta(\frac{x}{R}) v'(t,x) }_{L_x^2(\R^2)} + \norm{\abs{\nabla}^{\frac{1}{2}} \eta(\frac{x}{R}) v'(t,x)}_{L_x^2 (\R^2)} + \norm{\eta (\frac{x}{R}) \abs{\nabla}^{\frac{1}{2}} v'(t,x)}_{L_x^2 (\R^2)} \\
& \leq 2 \norm{v'(t,x)}_{H_x^{\frac{1}{2}} (\R^2)} \leq 2 \norm{\phi' (t)}_{H^5 (\R^2)} .
\end{align*}

By Littlewood-Paley decomposition and Sobolev embedding, we obtain that for $p >4$, in particular $p= p_0, p_1$
\begin{align*}
& \quad \norm{V_{R,N}}_{Z([-T_0 N^{-2} ,T_0 N^{-2}])} = \sup_{J \subset [-T_0 N^{-2}, T_0 N^{-2}]} \parenthese{\sum_{M \, dyadic} M^{4- \frac{p}{2}} \norm{P_M V_{R,N}}_{L_{t,x}^p(J \times \T^2)}^p}^{\frac{1}{p}} \\
& \lesssim \sup_{J \subset [-T_0 N^{-2}, T_0 N^{-2}]} \norm{ \inner{\nabla}^{\frac{4}{p}-\frac{1}{2}} V_{R,N}}_{L_{t,x}^p(J \times \T^2)} \lesssim \sup_{J \subset [-T_0 N^{-2}, T_0 N^{-2}]} \norm{ \inner{\nabla}^{\frac{1}{2}} V_{R,N}}_{L_t^p L_x^{\frac{2p}{p-2}}(J \times \T^2)} \\
& \lesssim \norm{\inner{\nabla}^{\frac{1}{2}} V_{R,N}}_{L_t^p L_x^{\frac{2p}{p-2}}([-T_0 N^{-2}, T_0 N^{-2}] \times \R^2)} \lesssim \sup_t \norm{v'(t)}_{H_x^5}.
\end{align*}
The  last inequality above is due to \eqref{eq H^5}. Hence we proved the first condition.

\item
\begin{align*}
& \quad \norm{f_N - V_{R,N} (0)}_{H_x^{\frac{1}{2}} (\T^2)} \leq \norm{\phi_N (\Psi^{-1} (y) - \phi'_{R,N} (\Psi^{-1} (y)}_{H_x^{\frac{1}{2}}(\T^2)} \\
&\leq \norm{\phi_N -\phi'_{R,N}}_{\dot{H}_x^{\frac{1}{2}} (\T^2)}  = \norm{\eta(\frac{x}{N^{\frac{1}{2}}}) \phi (x) - \eta(\frac{x}{N^{\frac{1}{2}}}) \phi' (x) }_{\dot{H}_x^{\frac{1}{2}} (\R^2)} \\
& \leq \norm{\eta(\frac{x}{N^{\frac{1}{2}}}) \phi (x) - \phi(x)}_{\dot{H}_x^{\frac{1}{2}} (\R^2)} + \norm{\phi - \phi'}_{\dot{H}_x^{\frac{1}{2}} (\R^2)} + \norm{\phi' - \eta(\frac{x}{N^{\frac{1}{2}}}) \phi' (x) }_{\dot{H}_x^{\frac{1}{2}} (\R^2)}.
\end{align*}
With the assumption $N \geq 10 R$, and by taking $R > R_0$, $R_0$ large enough, we can make the sum of the three norms above small than $\varepsilon_1$.

\item
\begin{align*}
& \quad \norm{E_{R,N}}_{N([-T_0 N^{-2} ,T_0 N^{-2}])} = \norm{\int_0^t e^{i(t-s) \Delta} E_{R,N} (s) \, ds}_{X^{\frac{1}{2}} ([-T_0 N^{-2} ,T_0 N^{-2}])} \\
& \leq \sup_{\norm{u_0}_{Y^{-\frac{1}{2}}} =1} \abs{\int_{[-T_0 N^{-2} ,T_0 N^{-2}] \times \T^2} \bar{u}_0 \cdot E_{R,N} \, dx dt} \\
& \leq \sup_{\norm{u_0}_{Y^{-\frac{1}{2}}} =1} \norm{u_0}_{Y^{-\frac{1}{2}}} \norm{\abs{\nabla}^{\frac{1}{2}} E_{R,N}}_{L_t^1 L_x^2 ([-T_0 N^{-2} ,T_0 N^{-2}] \times \T^2)} \\
& \leq \norm{\abs{\nabla_{\R^2} }^{\frac{1}{2}}e_{R,N}}_{L_t^1 L_x^2 ([-T_0 N^{-2} ,T_0 N^{-2}] \times \R^2)} = \norm{\abs{\nabla_{\R^2}}^{\frac{1}{2}} e_{R}}_{L_t^1 L_x^2 ([-T_0 ,T_0 ] \times \R^2)} .
\end{align*}
To estimate $\norm{\abs{\nabla_{\R^2}}^{\frac{1}{2}} e_{R}}_{L_t^1 L_x^2 ([-T_0 ,T_0 ] \times \R^2)}$, we compute the following two terms first.
\begin{align*}
\abs{ e_{R} (t,x) } & = \abs{   \rho (\eta (\frac{x}{R}) - \eta^5 (\frac{x}{R}) ) \abs{v'}^4 v' + R^{-2} v'(t,x) (\Delta_{\R^2} \eta) (\frac{x}{R}) + 2R^{-1} \sum_{j=1}^2 \partial_j v'(t,x) \partial_j \eta (\frac{x}{R}) }\\
& \lesssim \abs{\rho (\eta (\frac{x}{R}) - \eta^5 (\frac{x}{R}) ) v'} + R^{-2} \abs{(\Delta_{\R^2} \eta) (\frac{x}{R}) v'} + R^{-1} \sum_{j=1}^2 \abs{\partial_j v'(t,x) \partial_j \eta (\frac{x}{R})}   .
\end{align*}
Then
\begin{align*}
&\lim_{R \to 0} \norm{ e_R}_{L_t^1 L_x^2 ([-T_0, T_0] \times \R^2)} =0 .
\end{align*}
Notice that
\begin{align*}
& \sum_{k=1}^2 \abs{\partial_k e_{R} (t,x)} \lesssim_{\norm{\phi'}_{H^5(\R^2)}} \mathbf{1}_{[R, 2R]} (\abs{x}) \cdot \parenthese{\abs{v'(t,x)} + \sum_{k=1}^2 \abs{\partial_k v' (t,x)} + \sum_{k=1}^2 \sum_{j=1}^2 \abs{\partial_k \partial_j  v'(t,x)}}  ,
\end{align*}
therefore,
\begin{align*}
\lim_{R \to \infty} \norm{ \nabla e_R}_{L_t^1 L_x^2 ([-T_0, T_0] \times \R^2)} =0 .
\end{align*}
Interpolating the estimates above, we see that
\begin{align*}
\lim_{R \to \infty} \norm{ \abs{\nabla}^{\frac{1}{2}} e_R}_{L_t^1 L_x^2 ([-T_0, T_0] \times \R^2)} =0 .
\end{align*}
Then
\begin{align*}
\lim_{R \to \infty} \norm{ \abs{\nabla}^{\frac{1}{2}} E_{R,N}}_{L_t^1 L_x^2 ([-T_0 N^{-2}, T_0 N^{-2}] \times \T^2)} =0 .
\end{align*}
\end{enumerate}
Now we proved the three conditions listed, hence finish the proof of Lemma \ref{lem Compare Profile}.
\end{proof}

\begin{lem}[Extinction lemma]\label{lem Extinction}
Let $\phi \in \dot{H}^{\frac{1}{2}} (\R^2)$, and define $f_N$ as in \eqref{eq Euclidean profile}. For any $\varepsilon > 0$, there exist $T = T( \phi, \varepsilon)$ and $N_0 (\phi, \varepsilon)$ such that for all $N \geq N_0$, there holds that
\begin{align*}
\norm{e^{it\Delta} f_N}_{Z([TN^{-2}, T^{-1}])} \lesssim \varepsilon .
\end{align*}
\end{lem}

\begin{proof}[Proof of Lemma \ref{lem Extinction}]
This proof is adapted from Lemma 4.3 in \cite{IP1}.

For $M \geq 1$, we define
\begin{align*}
K_M (t,x) = \sum_{\xi \in \Z^2} e^{-i[t \abs{\xi}^2 + x \cdot \xi]} \eta^2 (\frac{\xi}{M}) = e^{it \Delta} P_{\leq M} \delta_0.
\end{align*}

We note from Lemma 3.18 in \cite{Bour93} that $K_M$ satisfies
\begin{align}\label{eq K_M}
\abs{K_M (t,x)} \lesssim \prod_{i=1}^2 \square{\frac{M}{\sqrt{q_i} (1 + M \abs{\frac{t}{\lambda_i} -\frac{ai}{q_i}}^{\frac{1}{2}})}} ,
\end{align}
if $a_i$ and $q_i$ satisfying $\frac{t}{\lambda_i} = \frac{a_i}{q_i} + \beta_i$, where $q_i \in \{ 1, \dots , M \}$, $a_i \in \Z$, $(a_i , q_i) =1$ and $\abs{\beta_i} \leq (M q_i)^{-1}$ for all $i = 1, 2$.

From this, we conclude that for any $1 \leq S \leq M$,
\begin{align}\label{eq K_M est}
\norm{K_M (t,x)}_{ L_{t,x}^{\infty} (\T^2 \times [SM^{-2} , S^{-1}])} \lesssim S^{-1} M^2 . 
\end{align}
This follows from \eqref{eq K_M} and Dirichlet's lemma. 
\begin{lem}[Dirichlet's lemma]
For any real number $\alpha$, and positive integer $N$, there exist integers $p$ and $q$ such that $1 \leq q \leq N$ and $\abs{q \alpha - p} < \frac{1}{N}$.
\end{lem}

Indeed, assume that $\abs{t} \leq \frac{1}{S}$, and write $\frac{t}{\lambda_i} = \frac{a_i}{q_i} + \beta_i$ and $\abs{\beta_i} \leq \frac{1}{M q_i} \leq \frac{1}{M} \leq \frac{1}{S}$. So we obtain
\begin{align*}
\abs{\frac{a_i}{q_i}} \leq \frac{2}{S}  \implies q_i \geq \frac{a_i}{2}.
\end{align*}
Therefore if $a_i \geq 1$, then $q_i \geq \frac{S}{2}$ and 
\begin{align*}
\frac{M}{\sqrt{q_i} (1 + M \abs{\frac{t}{\lambda_i} -\frac{ai}{q_i}}^{\frac{1}{2}})}  \leq \frac{M}{\sqrt{q}} \lesssim S^{-\frac{1}{2}} M .
\end{align*} 
If $a_i =0$ for $i=1,2$, then
\begin{align*}
\frac{M}{\sqrt{q_i} (1 + M \abs{\frac{t}{\lambda_i} -\frac{ai}{q_i}}^{\frac{1}{2}})}  \lesssim \frac{M}{\sqrt{q_i}  M \abs{t}^{\frac{1}{2}}}  \lesssim \abs{t}^{-\frac{1}{2}} \leq S^{-\frac{1}{2}} M.
\end{align*}
The last inequality holds due to the restriction $t \in [SM^{-1} , S^{-1}]$.

In view of the Strichartz estimates in Lemma \ref{lem Strichartz1} and Lemma \ref{lem Strichartz2}, to prove the lemma we may assume that $\phi \in C_0^{\infty} (\R^2)$. In this case, from the definition, 
\begin{align}\label{eq Claim2}
\norm{f_N}_{L^1(\T^2)} \lesssim_{\phi} N^{-\frac{3}{2}} , \quad \norm{P_K f_N}_{L^2 (\T^2)} \lesssim_{\phi} (1 + \frac{K}{N})^{-10} N^{-\frac{1}{2}} .
\end{align}

Now we can estimate $\norm{e^{it \Delta} f_N}_{Z(TN^{-2} , T^{-1})} $ for $1 \leq T \leq N$ and $p \in (4, \infty]$. To this end, we decompose the sum in frequencies $K$ into the following two pieces:
\begin{align*}
K \in [NT^{-\frac{1}{100}} , NT^{\frac{1}{100}}] \quad \text{ and } K \notin [NT^{-\frac{1}{100}} , NT^{\frac{1}{100}}] .
\end{align*}

Using the Strichartz estimates in Lemma \ref{lem Strichartz1} and Lemma \ref{lem Strichartz2}, we obtain, for $p \in ( 4 , \infty]$,
\begin{align}\label{eq Strichartz}
\norm{e^{it\Delta} P_K f_N}_{L_{t,x}^p ([-1, 1] \times \T^2 )} \lesssim_{\phi, p} K^{1-\frac{4}{p}} (1 + \frac{K}{N})^{-10} N^{-\frac{1}{2}} .
\end{align}
Therefore, the second piece becomes
\begin{align*}
\sum_{K \notin [NT^{-\frac{1}{100}} , NT^{\frac{1}{100}}]} K^{4-\frac{p}{2}} \norm{e^{it \Delta} f_N}_{L_{t,x}^p ([-1,1] \times \T^2)}^p \lesssim_{\phi} T^{-\frac{1}{100}} .
\end{align*}
To estimate the remaining sum over $K \in [NT^{-\frac{1}{100}} , NT^{\frac{1}{100}}]$ we use the first bound in \eqref{eq Claim2} and \eqref{eq K_M est} (with $M \approx \max (K,N)$, $S \approx T$). If follows that for all $K$,
\begin{align}\label{eq Claim3}
\norm{e^{it\Delta} P_K f_N}_{L_{t,x}^{\infty} ([TN^{-2} ,T^{-1}] \times \T^2) } \lesssim_{\phi} T^{-1} (K+N)^2 N^{-\frac{3}{2}} .
\end{align}
Interpolating with \eqref{eq Strichartz}, for $K \in [NT^{-\frac{1}{100}} , NT^{\frac{1}{100}}]$,
\begin{align*}
\norm{e^{it\Delta} P_K f_N}_{L_{t,x}^{p} ([TN^{-2} ,T^{-1}] \times \T^2)} \lesssim_{\phi} T^{-\frac{1}{100}} N^{\frac{1}{2} -\frac{p}{4}} .
\end{align*}
By setting $T = T(\varepsilon ,\phi)$ sufficiently large and putting the estimates of these two pieces, we complete the proof of Lemma \ref{lem Extinction}.
\end{proof}

\begin{defn}
Given $f \in L^2(\T^2)$, $t_0 \in \R$, and $x_0 \in \T^2$ we define
\begin{align*}
(\pi_{x_0} f) (x): &= f(x-x_0)\\
(\Pi_{t_0, x_0}) f(x) : &= (e^{-it_0 \Delta} f) (x- x_0) = (\pi_{x_0} e^{-it_0 \Delta} f) (x) .
\end{align*}
As in \eqref{eq Euclidean profile}, given $\phi \in \dot{H}^1 (\R^2)$ and $N \geq 1$, we define
\begin{align*}
T_N \phi(x) : = N^{\frac{1}{2}} \widetilde{\phi} (N \Psi^{-1} (x)), \text{ where } \widetilde{\phi} (y) := \eta(\frac{y}{N^{\frac{1}{2}}}) \phi(y),
\end{align*}
and observe that
\begin{align*}
T_N : \dot{H}^{\frac{1}{2}} (\R^2) \to H^{\frac{1}{2}}(\T^2) \text{ is a linear operator with } \norm{T_N \phi}_{H^1(\T^2)} \lesssim \norm{\phi}_{\dot{H}^1(\R^2)}.
\end{align*}
Let $\widetilde{\mathcal{F}}_e$ denote the set of renormalized Euclidean frames
\begin{align*}
\widetilde{\mathcal{F}}_e : = \{ (N_k, t_k, x_k)_{k \geq 1}: N_k \in [1, \infty), t_k \to 0, x_k \in \T^2 , N_k \to \infty, \\
 \text{ and either } t_k =0 \text{ for any } k \geq 1 \text{ or } \lim_{k \to \infty} N_k^2 \abs{t_k} = \infty  \} .
\end{align*} 
\end{defn}

\begin{prop}[Euclidean profiles]\label{prop Euclidean profiles}
Assume that $\OO = (N_k , t_k, x_k)_k \in \widetilde{\mathcal{F}}_e$, $\phi \in \dot{H}^{\frac{1}{2}} (\R^2)$, and let $U_k (0) = \Pi_{t_k, x_k} (T_{N_k} \phi)$.
\begin{enumerate}[\bf (1)]
\item
There exists $\tau = \tau(\phi)$ such that for $k$ large enough (depending only on $\phi$ and $\OO$ there is a nonlinear solution $U_k \in X^{\frac{1}{2}} (-\tau , \tau)$ of \eqref{NLS} with initial data $U_k (0)$, and
\begin{align}\label{eq LWP}
\norm{U_k}_{X^{\frac{1}{2}} ((-\tau , \tau))} \lesssim_{E_{\R^2} (\phi)} 1.
\end{align}

\item
There exists a Euclidean solution $u \in C(\R: \dot{H}^{\frac{1}{2}} (\R^2))$ of 
\begin{align}\label{eq ENLS}
(i\partial_t + \Delta_{\R^2}) u = \abs{u}^4 u
\end{align}
with scattering data $\phi^{\pm \infty}$ such that the following holds, up tp a subsequence: for any $\varepsilon > 0$, there exists $T(\phi, \varepsilon)$ such that for any $T \geq T(\phi ,\varepsilon)$ there exists $R(\phi, \varepsilon , T)$ such taht for all $R \geq R(\phi, \varepsilon ,T)$, there holds that
\begin{align}\label{eq:tildeu}
\norm{U_k - \widetilde{u}_k}_{X^{\frac{1}{2}} (\{ \abs{t-t_k} \leq TN_k^{-2} \} \cap \{ \abs{t} \leq T^{-1} \}  )} \leq \varepsilon
\end{align} 
for $k$ large enough, where
\begin{align*}
(\pi_{-x_k} \widetilde{u}_k) (t,x) = N_k^{\frac{1}{2}} \eta (\frac{N_k}{R} \Psi^{-1} (x)) u (N_k^2 (t-t_k , N_k \Psi^{-1}(x))).
\end{align*}

In addition, up to a subsequence,
\begin{align*}
\norm{U_k (t) - \Pi_{t_k -t , x_k} T_{N_k} \phi^{\pm \infty}}_{X^{\frac{1}{2}} (\{ \pm (t-t_k) \leq TN_k^{-2} \} \cap \{ \abs{t} \leq T^{-1} \}  )} \leq \varepsilon,
\end{align*}
for $k$ large enough (depending on $\phi, \varepsilon , T, R$).
\end{enumerate}
\end{prop}

\begin{proof}[Proof of Proposition \ref{prop Euclidean profiles}]
This proof is adapted from Proposition 4.4 in \cite{IP1}.

Without loss of generality, we may assume that $x_k =0$. 

\begin{enumerate}[\bf (1)]
\item
First, for $k$ large enough, we can make
\begin{align*}
\norm{\phi - \eta \parenthese{ \frac{x}{N_k^{\frac{1}{2}}}} \phi }_{\dot{H}_x^{\frac{1}{2}} (\R^2)} \leq \varepsilon_1.
\end{align*}
For each $N_k$, we choose $T_{0, N_k} = \tau N_k^2$ ($T_{0, N_k}$ is the coefficient in Lemma \ref{lem Compare Profile}). For each $T_{0, N_k}$, to apply  Lemma \ref{lem Compare Profile}, we can choose $R_k$ very large enough. Note that $R_k$ is determined by $T_{0, N_k}$ in the proof of Lemma \ref{lem Compare Profile}. 

\item
First we consider the case $t_k=0$, for any $k$, then we directly have \eqref{eq LWP} from Lemma \ref{lem Compare Profile}. The proposition follows from Theorem \ref{thm Yu}, Lemma \ref{lem Compare Profile} and Lemma \ref{lem Extinction}. In fact, we let $u$ be the nonlinear Euclidean solution of \eqref{eq ENLS} with $u(0) =\phi$ and notice that for any $\delta >0$ there is $T_{\phi, \delta}$ such that
\begin{align*}
\norm{\abs{\nabla_{\R^2}}^{\frac{1}{2}} u}_{L_{t,x}^{4} (\{ \abs{t} > T(\phi, \delta) \times \R^2 \})} \leq \delta .
\end{align*}
By Theorem \ref{thm Yu}, we obtain
\begin{align*}
\norm{u(\pm T(\phi, \delta)) - e^{\pm i T(\phi, \delta) \Delta} \phi^{\pm \infty}}_{\dot{H}^{\frac{1}{2}} (\R^2)} \leq \delta,
\end{align*}
which implies
\begin{align*}
\norm{U_{N_k} (\pm TN_k^{-2}) - \Pi_{\pm {T, x_k} }\phi^{\pm \infty} }_{H^{\frac{1}{2}} (\T^2)} \leq \delta .
\end{align*}
By Remark  \ref{rmk Z<X} and Remark \ref{rmk Embed}, we have
\begin{align}\label{eq Diff}
\norm{e^{it\Delta} (U_{N_k}(\pm TN_k^{-2}) - \Pi_{\pm {T,x_k}} \phi^{\pm \infty})}_{X^{\frac{1}{2}}(\abs{t} \leq T^{-1})} \leq \delta.
\end{align}
Proposition \ref{prop LWP} implies that
\begin{align*}
\norm{U_{N_k} -e^{it\Delta} U_{N_k} (\pm TN_k^{-2})}_{X^{\frac{1}{2}}} \leq \delta,
\end{align*}
and combining \eqref{eq Diff}, we have
\begin{align*}
\norm{U_{N_k} - \Pi_{\pm {t,x_k}} \phi^{\pm \infty})}_{X^{\frac{1}{2}}( \{\pm t > \pm T N_k^{-2}) \} \cap \{ \abs{t} \leq T^{-1} \}} \leq \varepsilon.
\end{align*}

Now let us consider the second case: $\abs{N_k}^2 \abs{t_k} \to \infty$. 
\begin{align*}
U_k (0) = \Pi_{t_k ,0}(T_{N_k} \phi) = e^{-it_k \Delta} \parenthese{N_k^{\frac{1}{2}} \widetilde{\phi}(N_k \Psi^{-1} (x)) }.
\end{align*}
By the existence of wave operator of NLS, we know that the following initial value problem is global well-posed, so there exists $v$ satisfying:
\begin{align*}
\begin{cases}
(i \partial_t + \Delta_{\R^2}) v = \abs{v}^4 v,\\
\lim_{t \to - \infty} \norm{v(t) - e^{-t\Delta} \phi}_{\dot{H}^{\frac{1}{2}} (\R^2)} =0.
\end{cases}
\end{align*}

We set 
\begin{align*}
\widetilde{v}_k (t) = N_k^{\frac{1}{2}} \eta \parenthese{\frac{N_k \Psi^{-1}}{R}} v (N_k^2 t , N_k \Psi^{-1} (x)  ),
\end{align*}
then
\begin{align*}
\widetilde{v}(-t_k) = N_k^{\frac{1}{2}} \eta \parenthese{\frac{N_k \Psi^{-1}}{R}} v (-N_k^2 t , N_k \Psi^{-1} (x) ).
\end{align*}

For $k$ and $R$ large enough,
\begin{align*}
& \quad \norm{\widetilde{v}_k (-t_k) - e^{-it_k \Delta} N_k^{\frac{1}{2}} \eta (N_k^{\frac{1}{2}} \Psi^{-1}) \phi (N_k \Psi^{-1})}_{\dot{H}_x^{\frac{1}{2}} (\T^2)} \\
& \leq \norm{\eta(\frac{x}{N_k^{\frac{1}{2}}}) v(-N_k^2 t_k , x) - e^{it_k N_k^2 \Delta} \eta(\frac{x}{N_k^{\frac{1}{2}}}) \phi(x) }_{\dot{H}_x^{\frac{1}{2}} (\R^2)} \leq \varepsilon.
\end{align*}

So $V_k(t)$ solves initial value problem \eqref{NLS} on $\T^2$, with initial data $V_k(0) =\widetilde{V}_k (0)$, which implies $V_k(t)$ exists in $[-\delta, \delta]$, and $\norm{V_k(t) - \widetilde{V}_k(t)}_{X^{\frac{1}{2}} (-\delta , \delta)} \lesssim \varepsilon$. By the stability, $\norm{U_k - V_k}_{X^{\frac{1}{2}} ((-\delta , \delta))} \to 0$, as $k \to \infty$.
\end{enumerate}
Now we finish the proof of Proposition \ref{prop Euclidean profiles}.
\end{proof}

The following corollary decompose the nonlinear Euclidean profiles $U_k$ defined in Proposition \ref{prop Euclidean profiles}. This corollary follows closely in a part of the proof of Lemma 6.2 in \cite{IP1}.

\begin{cor}[Decomposition of the nonlinear Euclidean profiles $U_k$]\label{cor Decomp nonlinear Euclidean profiles}
Consider $U_k$ is the nonlinear Euclidean profiles with respect to $\OO = (N_k , t_k, x_k)_k \in \widetilde{\mathcal{F}}_e$ defined above. For any $\theta > 0$, there exists $T_{\theta}^0$ sufficiently large such that for all $T_{\theta}$, $T_{\theta} \geq T_{\theta}^0$ and there exists $R_{\theta}$ sufficiently large such that for all $k$ large enough (depending on $R_{\theta}$) we can decompose $U_k$ as following:
\begin{align*}
\mathbf{1}_{(-T_{\theta}^{-1} , T_{\theta}^{-1})} (t) U_k = \omega_k^{\theta , -\infty} + \omega_k^{\theta , +\infty} + \omega_k^{\theta} + \rho_k^{\theta} ,
\end{align*}
and $\omega_k^{\theta , \pm \infty}$, $\omega_k^{\theta}$ and $\rho_k^{\theta}$ satisfying the following conditions:
\begin{align}\label{eq cor4.8}
\begin{aligned}
& \norm{\omega_k^{\theta, \pm \infty}}_{Z'  (-T_{\theta}^{-1} , T_{\theta}^{-1})} + \norm{\rho_k^{\theta}}_{X^{\frac{1}{2}} (-T_{\theta}^{-1} , T_{\theta}^{-1})} \leq \theta,\\
& \norm{\omega_k^{\theta, \pm \infty}}_{X^{\frac{1}{2}} (-T_{\theta}^{-1} , T_{\theta}^{-1})} + \norm{\omega_k^{\theta}}_{X^{\frac{1}{2}} (-T_{\theta}^{-1} , T_{\theta}^{-1})} \lesssim 1,\\
& \omega_k^{\theta , \pm \infty} = P_{\leq R_{\theta} N_k \omega_k^{\theta , \pm \infty}} ,\\
& \abs{\nabla_x^m \omega_k^{\theta}} + (N_k)^{-2} \mathbf{1}_{S_k^{\theta}} \abs{\partial_t \nabla_x^m \omega_k^{\theta} } \leq R_{\theta} (N_k)^{\abs{m} +\frac{1}{2}} \mathbf{1}_{S_k^{\theta}}, \quad m \in \N^2, \, 0 \leq \abs{m} \leq 10,
\end{aligned}
\end{align}
where $S_k^{\theta} : = \{ (t,x) \in  (-T_{\theta}^{-1} , T_{\theta}^{-1}) \times \T^2 : \abs{t - t_k} \leq T_{\theta} N_k^{-2} , \abs{x-x_k} \leq R_{\theta} N_k^{-1} \}$.
\end{cor}

\begin{proof}[Proof of Corollary \ref{cor Decomp nonlinear Euclidean profiles}]
By Proposition \ref{prop Euclidean profiles}, there exists $T(\phi, \frac{\theta}{4})$, such that for all $T \geq T(\phi, \frac{\theta}{4})$, there exists $R(\phi, \frac{\theta}{4} , T)$ such that for all $R \geq R(\phi, \frac{\theta}{2} ,T)$, there holds that
\begin{align*}
\norm{U_k -\widetilde{u}_k}_{X^{\frac{1}{2}}(\{ \abs{t- t_k} \leq TN_k^{-2} \}  \cap \{\abs{t} \leq T^{-1} \})} \leq \frac{\theta}{2},
\end{align*} 
for $k$ large enough, where
\begin{align*}
(\pi_{-x_k} \widetilde{u}_k) (t,x) = N_k \eta(N_k \Psi^{-1}(x) R^{-1}) u(N_k^2(t-t_k) , N_k \Psi^{-1} (x)),
\end{align*}
where $u$ is a solution of \eqref{NLS} with scattering data $\phi^{\pm}$.

In addition, up to subsequence,
\begin{align*}
\norm{U_k - \Pi_{t_k-t, x_k} T_{N_k} \phi^{\pm}}_{X^{\frac{1}{2}} (\{ \pm (t-t_k) \geq TN_k^{-2}\} \cap \{ \abs{t} \leq T^{-1} \})} \leq \frac{\theta}{4},
\end{align*}
for $k$ large enough (depending on $\phi, \theta, T$ and $R$).

Choose a sufficiently large $T_{\theta} > T(\phi, \frac{\theta}{4})$ based on the extinction lemma (Lemma \ref{lem Extinction}), such that
\begin{align*}
\norm{e^{it\Delta} \Pi_{t_k, x_k} T_{N_k \phi^{\pm \infty}}}_{Z(T_{\theta}N_k^{-2} , T_{\theta}^{-1})} \leq \frac{\theta}{4} ,
\end{align*}
when $k$ large enough.

And then we choose $R_{\theta} = R(\phi, \frac{\theta}{2} , T_{\theta})$.
Denote:
\begin{enumerate}[\bf (1)]
\item
$\omega_k^{\theta, \pm \infty} : = \mathbf{1}_{ \{\pm(t-t_k) \geq T_{\theta}N_k^{-2} , \abs{t} \leq T_{\theta}^{-1} \}} (\Pi_{t_k -t, x_k} T_{N_k} \phi^{\theta, \pm \infty})$, where 
\begin{align*}
\norm{\phi^{\theta, \pm \infty}}_{\dot{H}^{\frac{1}{2}} (\R^2)} \lesssim 1 , \phi^{\theta, \pm \infty} = P_{\leq R_{\theta}} \phi^{\theta , \pm \infty},
\end{align*}
which implies $\omega_k^{\theta , \pm \infty} = P_{\leq R_{\theta} N_{\theta}} \omega_k^{\theta, \pm \infty}$.

\item
$\omega_k^{\theta}: = \widetilde{u}_k \cdot \mathbf{1}_{S_k^{\theta}}$.

By the stability and Lemma \ref{lem Compare Profile} we can adjust $\omega_k^{\theta}$ and $\omega_k^{\theta , \pm \infty}$ with an acceptable error, to make 
\begin{align*}
\abs{\nabla_x^m \omega_k^{\theta}} + N_k^{-2} \mathbf{1}_{S_k^{\theta}} \abs{\partial_t \nabla_x^m \omega_k^{\theta}} \leq R_{\theta} N_k^{\abs{m}+1} \mathbf{1}_{S_k^{\theta}} , \quad 0 \leq \abs{m} \leq 10 .
\end{align*}
\item
$\rho_k : = \mathbf{1}_{(-T_{\theta}^{-1} , T_{\theta}^{-1})} (t) U_k^{\alpha} - \omega_k^{\theta} -\omega^{\theta , \infty} - \omega^{\theta, -\infty} $.
\end{enumerate}
By Proposition \ref{prop Euclidean profiles}, we obtain that
\begin{align*}
\norm{\rho_k^{\theta} }_{X^{\frac{1}{2}} (\{ \abs{t} < T_{\theta}^{-1}\})} \leq \frac{\theta}{2},
\end{align*}
and then we have
\begin{align*}
& \norm{\omega_k^{\theta, \pm \infty}}_{Z' ([-T_{\theta}^{-1} , T_{\theta}^{-1} ])} + \norm{\rho_k^{\theta}}_{X^{\frac{1}{2}} ([-T_{\theta}^{-1} , T_{\theta}^{-1} ])} \leq \theta,\\
& \norm{\omega_k^{\theta, \pm \infty}}_{X^{\frac{1}{2}}([-T_{\theta}^{-1} , T_{\theta}^{-1}])} + \norm{\omega_k^{\theta}}_{X^{\frac{1}{2}} ([-T_{\theta}^{-1} , T_{\theta}^{-1}])} \lesssim 1.
\end{align*}
Now we finish the proof of Corollary \ref{cor Decomp nonlinear Euclidean profiles}.
\end{proof}

\section{Profile decomposition}\label{sec Profile decomp}

In this section, we construct the profile decomposition on $\T^2$ for linear
Schr\"odinger equations. In particular, in Lemma \ref{prop:almostorth} we show the almost  orthogonality of nonlinear profiles which will be used in next section.

As in the previous section, given $f \in L^2(\R^2)$, $t_0 \in \R$, and
$x_0 \in \T^2$, we define:
\begin{align*}
(\Pi_{t_0, x_0} ) f(x) &:= (e^{-it_0 \DD} f)(x-x_0)\\
T_N \phi (x) &:= N^{\frac{1}{2}} \widetilde{\phi} (N \Psi^{-1}(x)),
\end{align*}
where $\widetilde{\phi}(y):= \eta(\frac{y}{N^{\frac{1}{2}}})\phi(y)$.

Observe that $T_N : \dot{H}^{\frac{1}{2}}(\R^2) \to H^{\frac{1}{2}}(\T^2)$ is a linear operator with $\norm{T_N \phi}_{H^{\frac{1}{2}}(\T^2)}\lesssim \norm{\phi}_{\dot{H}^{\frac{1}{2}}(\R^2)}$.

\begin{defn}[Euclidean frames]
\begin{enumerate}[\bf (1)]
\item We define a Euclidean frame to be a sequence $\mathcal{F}_e =(N_k, t_k, x_k)_k$ with $N_k\geq 1$, $N_k\to +\infty$, $t_k\in\R$, $t_k\to 0$, $x_k\in\T^2$. We say that two frames, $(N_k, t_k, x_k)_k$ and $(M_k, s_k, y_k)_k$ are orthogonal if
\begin{align*}
\lim_{k\to+\infty} \parenthese{ \ln \abs{\frac{N_k}{M_k}}+ N_k^2 \abs{t_k-s_k} + N_k\abs{x_k-y_k} } =\infty.
\end{align*}
Two frames that are not orthogonal are called equivalent.

\item If $\mathcal{O} =(N_k, t_k, x_k)_k$ is a Euclidean frame and if  $\phi\in\dot{H}^{\frac{1}{2}}(\R^2)$, we define the Euclidean profile associated to  $(\phi, \mathcal{O})$ as the sequence $\widetilde{\phi}_{\mathcal{O}_k}$:
\begin{align*}
\widetilde{\phi}_{\mathcal{O}_k} := \Pi_{t_k, x_k} (T_{N_k}\phi).
\end{align*}
\end{enumerate}
\end{defn}

\begin{prop}[Equivalence of frames \cite{IP2}]\label{prop:equivalenceFrames}
\begin{enumerate}[\bf (1)]
\item
If $\mathcal{O}$ and $\mathcal{O}'$ are equivalent Euclidean frames, then there
exists an isometry $T: \dot{H}^{\frac{1}{2}}(\R^2) \to \dot{H}^{\frac{1}{2}}(\R^2)$ such that for any profile $\widetilde{\phi}_{\mathcal{O}'_k}$, up to a subsequence there holds that
\begin{align*}
\limsup_{k\to\infty} \norm{\widetilde{T \phi}_{\mathcal{O}_k} - \widetilde{\phi}_{\mathcal{O}'_k}}_{H_x^{\frac{1}{2}}(\T^2)} = 0.
\end{align*}

\item
If $\mathcal{O}$ and $\mathcal{O}'$ are orthogonal Euclidean frames and $\widetilde{\phi}_{\mathcal{O}_k}$, $\widetilde{\phi}_{\mathcal{O}'_k}$ are corresponding profiles, then, up to a subsequence:
\begin{align}\label{eq:almostorth1}
\lim_{k\to\infty} \inner{ \widetilde{\phi}_{\mathcal{O}_k},  \widetilde{\phi}_{\mathcal{O}'_k}}_{H^{\frac{1}{2}}\times H^{\frac{1}{2}}(\T^2)} = 0;\\\label{eq:almostorth2}
\lim_{k\to\infty} \inner{ \abs{\widetilde{\phi}_{\mathcal{O}_k}}^3,  \abs{\widetilde {{\phi}}_{\mathcal{O}'_k}}^3}_{L^2\times L^2(\T^2)} = 0.
\end{align}
\end{enumerate}
\end{prop}

\begin{proof}[Proof of Proposition \ref{prop:equivalenceFrames}]
See Lemma 5.4 in \cite{IPS} for the proof.
\end{proof}

\begin{defn}[Definition 5.3 \cite{IP1}]
We say that a sequence of functions $\{ f_k\}_k \subset H^{\frac{1}{2}} (\T^2)$ is absent from a frame $ \OO$ if, up to a subsequence, for every profile $\psi_{\OO_k}$ associated to $\OO$, 
\begin{align*}
& \lim_{k\to\infty} \inner{ f_k,  \wt{\psi}_{\OO_k}}_{L^2\times L^2(\T^2)} = 0 ,\\
& \lim_{k\to\infty} \inner{ f_k,  \wt{\psi}_{\OO_k}}_{H^{\frac{1}{2}}\times H^{\frac{1}{2}}(\T^2)} = 0 .
\end{align*}
\end{defn}

\begin{lem}[Refined Strichartz inequality]\label{lem Refined}
Let $f\in H^1(\T^2)$ and $I\subset [0,1]$. Then 
\begin{align*}
\norm{e^{it\Delta} f}_{Z(I)} \lesssim \norm{f}_{H_x^1(\T^2)}^{\frac{5}{6}} \sup_{N\in 2^{\Z}} (N^{-1} \norm{P_N e^{it\Delta}f}_{L_{t,x}^{\infty}(I\times\T^2)})^{\frac{1}{6}}.
\end{align*}
\end{lem}

\begin{proof}[Proof of Lemma \ref{lem Refined}]
By the definition of $Z$-norm,
\begin{align*}
\norm{e^{it\Delta} f}_{Z(I)} = \parenthese{\sum_N N \norm{P_N e^{it\Delta} f}_{L^6_{t,x}}^6 }^{\frac{1}{6}} = \norm{N^{\frac{1}{6}} \norm{P_N e^{it\Delta}f}_{L^6_{t,x}} }_{l^6_N}.
\end{align*}
By H\"{o}lder inequality,  Lemma \ref{lem Strichartz1} and Lemma \ref{lem Strichartz2}, we have that for any $0< a < 1$ and $4<ap<\infty$, 
\begin{align*}
\norm{N^{\frac{4}{p}-\frac{1}{2}} \norm{P_N e^{it\Delta}f}_{L_{t,x}^p} }_{l_N^p} &  \lesssim  \norm{  \parenthese{N^{\frac{4}{ap}-\frac{1}{2}} \norm{P_Ne^{it\Delta}f}_{L^{ap}_{t,x}}}^{a} \parenthese{N^{-\frac{1}{2}}   \norm{P_N e^{it\Delta} f}_{L_{t,x}^{\infty}}}^{1-a}}_{l_N^p}\\
&  \lesssim  \sup_{N} \parenthese{ N^{-1} \norm{P_N e^{it\Delta} f }_{L_{t,x}^{\infty}} }^{1-a}  \parenthese{ \sum_N N^{4-\frac{ap}{2}} \norm{P_N e^{it\Delta} f }_{L_{t,x}^{ap}}^{ap} }^{\frac{1}{p}}\\
&  \lesssim  \sup_{N} \parenthese{ N^{-1} \norm{P_N e^{it\Delta} f }_{L_{t,x}^{\infty}} }^{1-a}  \parenthese{ \sum_N N^{4-\frac{ap}{2}}  N^{ap(1-\frac{4}{ap} -\frac{1}{2})}\norm{P_N f }_{H^{\frac{1}{2}}}^{ap}}^{\frac{1}{p}}\\
&  \lesssim  \sup_{N} \parenthese{N^{-1} \norm{P_N e^{it\Delta} f}_{L^\infty_{t,x}} }^{1-a}  \norm{f}_{H_x^{\frac{1}{2}}(\T^2)}^{a} . 
\end{align*}
We finish the proof of Lemma \ref{lem Refined}.
\end{proof}

\begin{rmk}
In fact, Lemma \ref{lem Refined} implies that the $Z$-norm of $e^{it\Delta} f$ is weaker than  $\norm{N^{-1} \norm{P_N e^{it\Delta}f}_{L_{t,x}^{\infty}}}_{l_N^{\infty}}$. Hence once we control $\norm{N^{-1} \norm{P_N e^{it\Delta}f}_{L_{t,x}^{\infty}}}_{l_N^{\infty}}$, then we have the control of $Z$-norm of $e^{it\Delta} f$. This fact will be used in Proposition \ref{prop:ProfileDecomposition} and Lemma \ref{lem 5.6}, where we state the smallness in $Z$-norm of $e^{it\Delta} f$, but in fact we will control $\norm{N^{-1} \norm{P_N e^{it\Delta}f}_{L_{t,x}^{\infty}}}_{l_N^{\infty}}$ in their proofs.
\end{rmk}

\begin{prop}[Profile decomposition]\label{prop:ProfileDecomposition}
Consider $\{ f_k\}_{k}$ a sequence of functions in $H^{\frac{1}{2}}(\T^2)$ and $0<A<\infty$ satisfying
\begin{align*}
 \limsup_{k\to +\infty} \norm{f_k}_{H^{\frac{1}{2}}(\T^2)} \leq A
\end{align*}
and a sequence of intervals $I_k=(-T_k, T^k)$ such that $\abs{I_k} \to 0$ as $k\to\infty$. Up to passing to a subsequence, assume that $f_k \rightharpoonup g \in H^{\frac{1}{2}}(\T^2)$. There exists $J^*\in \N$, and a sequence of profile $ \widetilde{\psi}_{\mathcal{O}_k^{\alpha}}^{\alpha}$  associated to pairwise orthogonal Euclidean frames $\mathcal{O}^{\alpha}$ and $\psi^{\alpha} \in {H}^{\frac{1}{2}}(\R^2)$ such that extracting a subsequence, for every $0 \leq J\leq J^*$, we have
\begin{align}\label{5.1}
f_k = g + \sum_{1\leq \alpha \leq J} \widetilde{\psi}_{\mathcal{O}_k^{\alpha}}^{\alpha}  + R_k^J
\end{align}
where $R^J_k$ is small in the sense that
\begin{align}\label{5.2}
\limsup_{J \to J^*}\limsup_{k\to\infty}\norm{e^{it\Delta}R_k^J}_{Z(I_k)} = 0.
\end{align}
Besides, we also have the following orthogonality relations:
\begin{align*}
& \norm{f_k}_{L^2}^2 = \norm{g}_{L^2}^2 + \norm{R_k^J}_{L^2}^2 + o_k (1),\\
& \norm{f_k}_{\dot{H}^{\frac{1}{2}}}^2  = \norm{g}_{\dot{H}^{\frac{1}{2}}}^2 + \sum_{\alpha \leq J} \norm{\psi^{\alpha}}_{\dot{H}^{\frac{1}{2}} (\R^2)}^2 + \norm{R_k^J}_{\dot{H}^{\frac{1}{2}}}^2 + o_k(1)  .
\end{align*}
\end{prop}

In fact, the proof of Proposition \ref{prop:ProfileDecomposition} relies on applying  inductively the following Lemma \ref{lem 5.6}. 
\begin{lem}\label{lem 5.6}
Consider $\{ f_k\}_k$ a sequence of functions in $H^{\frac{1}{2}} (\T^2)$ and $ 0 < E < \infty$ satisfying
\begin{align}\label{eq 5.5}
\limsup_{k \to \infty} \norm{f_k}_{H^{\frac{1}{2}} (\T^2)} \leq E
\end{align}
and a sequence of intervals $I_k = (-T_k, T^k)$ such that $\abs{I_k} \to 0$ as $k \to \infty$. Fix $\delta > 0$, there exists $J \lesssim \delta^{-2}$ profiles $\wt{\psi}_{\OO_k^{\alpha}}^{\alpha}$ associated to pairwise orthogonal Euclidean frames $\OO^{\alpha}$, $\alpha =1 , 2, \cdot , J$, such that
\begin{align*}
f_k = g + \sum_{1 \leq \alpha \leq J} \wt{\psi}_{\OO_k^{\alpha}}^{\alpha} + R_k
\end{align*}
where $R_k$ is small in the sense that
\begin{align}\label{eq small R}
\limsup_{k \to \infty} \norm{e^{it\Delta} R_k}_{Z(I_k)} \leq \delta
\end{align}
Besides, we also have
\begin{align*}
& \norm{f_k}_{L^2}^2 = \norm{g}_{L^2}^2 + \norm{R_k}_{L^2}^2 + o_k (1),\\
& \norm{f_k}_{\dot{H}^{\frac{1}{2}}}^2  = \norm{g}_{\dot{H}^{\frac{1}{2}}}^2 + \sum_{\alpha \leq J} \norm{\psi^{\alpha}}_{\dot{H}^{\frac{1}{2}} (\R^2)}^2 + \norm{R_k}_{\dot{H}^{\frac{1}{2}}}^2 + o_k(1).
\end{align*}

\end{lem}

\begin{proof}[Proof of Lemma \ref{lem 5.6}]
This proof is adapted from Lemma 5.4 in \cite{IP2}.

To proof \eqref{eq small R}, by the refined Strichartz estimates, it is sufficient to find $\{ \wt{\psi}_{\OO_k^{\alpha}}^{\alpha} \}_k$ for $1 \leq \alpha \leq J$ such that
\begin{align*}
\Omega (\{R_k \}) : = \limsup_{k \to \infty} \sup_{N \geq 1, t \in I_k, x \in \T^2} N^{-\frac{1}{2}} \abs{ (e^{it\Delta} P_N R_k) (x)} \leq \delta
\end{align*}
Notice that $H^{\frac{1}{2}} (\T^2)$ is a separable Hilbert space, hence up to a subsequence there exists a unique $g \in H^{\frac{1}{2}} (\T^2)$ such that for $\norm{f_k}_{H^{\frac{1}{2}}(\T^2)} \leq A$,
\begin{align*}
f_k \rightharpoonup g \quad \text{ in } H^{\frac{1}{2}} (\T^2) .
\end{align*}
And using $\inner{f_k -g , g} \to 0$ as $k \to \infty$, it is easy to see that
\begin{align*}
\norm{f_k}_{H^{\frac{1}{2}}(\T^2)}^2 = \norm{g}_{H^{\frac{1}{2}}(\T^2)}^2 + \norm{f_k -g}_{H^{\frac{1}{2}}(\T^2)}^2 .
\end{align*}
Similarly, we can also check that $f_k \rightharpoonup g $ in $L^2 (\T^2)$ and
\begin{align*}
\norm{f_k}_{L^2 (\T^2)}^2 = \norm{g}_{L^2 (\T^2)}^2 + \norm{f_k -g}_{L^2 (\T^2)}^2  .
\end{align*}
Denote that $\wt{f}_k := f_k -g$, then $\wt{f}_k \rightharpoonup  0$ in $H^{\frac{1}{2}} (\T^2)$. 
If $\Omega (\{ \wt{f}_k \}) < \delta$, we prove the lemma. Otherwise, if $\Omega (\{ \wt{f}_k \}) \geq \delta$, we can extract a frame from $\{ \wt{f}_k \}$. By the definition of $\Omega$, up to extracting a subsequence, there exists $(N_k , t_k, x_k)_k$ with $t_k \to 0$ and  $x_k - x_0 \in \T^2$ such that for all $k$
\begin{align*}
\frac{\delta}{2} & \leq N_k^{-\frac{1}{2}} \abs{(e^{it_n \Delta}) P_{N_k} \wt{f}_k (x_k)} = \abs{\inner{N_k^{-\frac{1}{2}} e^{it\Delta} P_{N_k} \wt{f}_k , \delta_{x_k}}_{} } \\
& \leq \abs{ \inner{f_k , N_k^{-\frac{1}{2}} e^{-it_n \Delta} P_{N_k} \delta_k }_{H^{\frac{1}{2}} \times H^{-\frac{1}{2}}} } .
\end{align*}
Now we claim that $N_k \to \infty$. In fact, if $N_k$ remains bounded, then up to a subsequence, one may assume that  
\begin{align*}
N_k \to N_{\infty} < \infty .
\end{align*}
by compactness of $\T^2$. In this case, we define $\psi = (1-\Delta)^{-\frac{1}{2}} N_{\infty}^{-\frac{1}{2}} P_{N_{\infty}} \delta_{x_0}$. By the continuity with respect to $(N,x,t) \in N^{-\frac{1}{2}} e^{-it\Delta} P_N \delta_x$, we have for $k $ large enough
\begin{align*}
\inner{\wt{f}_k , \psi}_{H^{\frac{1}{2}} \times H^{\frac{1}{2}}} \geq \frac{\delta}{4} .
\end{align*}
which contradicts $\wt{f}_k \rightharpoonup  0$ in $H^{\frac{1}{2}} (\T^2)$. We then define the Euclidean frame $\OO = (N_k , t_k, x_k)_k$, and the function
\begin{align*}
\psi = \mathcal{F}_{\R^2}^{-1} (\frac{1}{\abs{\xi}^{\frac{3}{2}}} (\eta(\xi) - \eta(2\xi)) e^{ix_0 \cdot \xi}) .
\end{align*}
Using the Parseval's identity and Sobolev embedding, we have
\begin{align*}
\lim_{k \to \infty} \norm{(1- \Delta)^{\frac{1}{2}} T_{N_k} \psi - N_k^{-\frac{1}{2}} P_{N_k} \delta_{x_0}}_{L^{\frac{4}{3}} (\T^2)} =0 ,
\end{align*}
and thus 
\begin{align*}
\lim_{k \to \infty} \norm{(1- \Delta)^{\frac{1}{2}} T_{N_k} \psi - N_k^{-\frac{1}{2}} P_{N_k} \delta_{x_0}}_{H^{-\frac{1}{2}}}  =0 ,
\end{align*}
then we conclude that
\begin{align}\label{eq 5.9}
\frac{\delta}{2} \lesssim \abs{\inner{ \wt{f}_k , N_k^{-\frac{1}{2}} e^{it\Delta} P_{N_k} \delta_{x_k}}_{H^{\frac{1}{2}} \times H^{-\frac{1}{2}} }} \lesssim \abs{ \inner{\wt{f}_k, \wt{\psi}_{\OO_k}}_{H^{\frac{1}{2}} \times H^{\frac{1}{2}}}}  .
\end{align}
Also it is easy to check that 
\begin{align}\label{eq 5.10}
\limsup_{k \to \infty} \norm{\wt{\psi}_{\OO_k}}_{H^{\frac{1}{2}}} \lesssim 1 .
\end{align}
Now we have selected a Euclidean frame $\OO$. 

For $R > 0$ and $k $ large enough, we consider for $y \in \R^2$
\begin{align*}
\phi_k^R (y) = N_k^{-\frac{1}{2}} \eta(\frac{y}{R}) (\Pi_{-t_k, -x_k} \wt{f}_k)(\psi(\frac{y}{N_k})) .
\end{align*}
This is a sequence of functions in the Hilbert space $\dot{H}^{\frac{1}{2}} (\R^2)$ that are  uniformly bounded in $R$. Hence we can extract a subsequence $\phi_k^R \rightharpoonup \phi^R \in \dot{H}^{\frac{1}{2}} (\R^2)$ satisfying $\norm{\phi^R}_{\dot{H}^{\frac{1}{2}} (\R^2)}$ for every $R$. 

By extracting a further subsequence, we may assume that $\phi^R \rightharpoonup \phi \in \dot{H}^{\frac{1}{2}} (\R^2)$ and by the uniqueness of the weak limit, we see that for every $R$, 
\begin{align*}
\phi^R (x) = \eta(\frac{x}{R}) \phi(x) .
\end{align*}
Now for any $\gamma \in C_0^{\infty} (\R^2)$ supporting in $B (0, \frac{R}{2}) \subset \R^2$ and for $k$ large enough, we have
\begin{align*}
\inner{f_k - \wt{\phi}_{\OO_k} , \wt{\gamma}_{\OO_k}}_{H^{\frac{1}{2}} \times H^{\frac{1}{2}} (\T^2)} & = \inner{\wt{f}_k , \wt{\gamma}_{\OO_k}}_{H^{\frac{1}{2}} \times H^{\frac{1}{2}}} - \inner{\wt{\phi}_{\OO_k} , \wt{\gamma}_{\OO_k}}_{H^{\frac{1}{2}} \times H^{\frac{1}{2}} (\T^2)} \\
& = \inner{\Pi_{-t_k, -x_k} \wt{f}_k , T_{N_k} \gamma}_{H^{\frac{1}{2}} \times H^{\frac{1}{2}}(\T^2)} - \inner{\phi , \gamma}_{\dot{H}^{\frac{1}{2}} \times \dot{H}^{\frac{1}{2}} (\R^2)} \\
& = \inner{\phi^R , \gamma}_{\dot{H}^{\frac{1}{2}} \times \dot{H}^{\frac{1}{2}} (\R^2)} - \inner{\phi , \gamma}_{\dot{H}^{\frac{1}{2}} \times \dot{H}^{\frac{1}{2}} (\R^2)} + o_k(1) = o_k (1) .
\end{align*}
This implies that $h_k = \wt{f}_k- \wt{\phi}_{\OO_k}$ is absent from the frame $\OO$. And by \eqref{eq 5.9} and \eqref{eq 5.10}, we know 
\begin{align}\label{eq 5.13}
\norm{\phi}_{\dot{H}^{\frac{1}{2}}(\R^2)} \gtrsim \delta .
\end{align}
Now by Proposition \ref{prop:equivalenceFrames}, we have
\begin{align}\label{eq 5.12}
\begin{aligned}
& \norm{h_k}_{L^2}^2 =  \norm{\wt{f}_k}_{L^2}^2 + o_k (1) , \\
& \norm{h_k}_{\dot{H}^{\frac{1}{2}}}^2  = \norm{\wt{f}_k}_{\dot{H}^{\frac{1}{2}}}^2 + \norm{\phi}_{\dot{H}^{\frac{1}{2}} (\R^2)}^2 - 2 \inner{\wt{f}_k , \wt{\phi}_{\OO_k}}_{\dot{H}^{\frac{1}{2}} \times \dot{H}^{\frac{1}{2}}} = \norm{\wt{f}_k}_{\dot{H}^{\frac{1}{2}}}^2 + \norm{\phi}_{\dot{H}^{\frac{1}{2}} (\R^2)}^2  + o_k (1) .
\end{aligned}
\end{align}
Finally, we define $f_k^0 =\wt{f}_k$. For $\alpha \geq 0$, when $\Omega (\{ f_k^{\alpha} \}) > \delta$ we proceed as follow: applying what we did above to find a frame $\OO_{\alpha}$ and an associated profile $\wt{\phi}_{\OO_k}^{\alpha}$. Then let 
\begin{align*}
f_k^{\alpha +1} = f_k^{\alpha} - \wt{\phi}_{\OO_k^{\alpha}}^{\alpha} .
\end{align*}
Note that $f^{\alpha +1}$ is absent from $\OO^{\alpha}$ by the construction. By induction and Proposition \ref{prop:equivalenceFrames}, it is also absent from all the frames $\OO^{\beta}$, $\beta \leq \alpha$. Similarly all frames $\OO^{\beta}$, $\beta \leq \alpha$ are pairwise orthogonal. Using \eqref{eq 5.12} inductively, we also obtain that
\begin{align*}
\norm{A f_k}_{L^2}^2 = \norm{Ag}_{L^2}^2 + \sum_{\beta \leq \alpha} \norm{A  \wt{\phi}_{\OO_k^{\beta}}}_{L^2}^2 + \norm{A f_k^{\alpha +1}}_{L^2} + o_k (1)
\end{align*}
for $A =1$ or $\abs{\nabla}^{\frac{1}{2}}$.
By \eqref{eq 5.5}, \eqref{eq 5.13} and the orthogonality property above, we see that this procedure stops after at most $O(\delta^{-2})$ steps. Suppose $J$ is the last index of $\alpha$, we then just define $R_k = f_k - g  - \sum_{1 \leq \alpha \leq J} \wt{\phi}_{\OO_k^{\alpha}}^{\alpha}$.

We finish the proof of Lemma \ref{lem 5.6}.
\end{proof}
\begin{rmk}
$g$ and $\widetilde{\psi}_k^{\alpha}$ for all $\alpha$ are called profiles. In addition, we call $g$ is a scale-one profile, and $\widetilde{\psi}_{k}^{\alpha}$ are called Euclidean profiles.
\end{rmk}

\begin{proof}[Proof of Proposition \ref{prop:ProfileDecomposition}]
We let $\delta_l = 2^{-l}$ and we apply inductively Lemma \ref{lem 5.6} to get the sequence of frames and profiles $\OO^{\alpha}$, $\wt{\phi}_{\OO_k^{\alpha}}^{\alpha}$. 
\end{proof}

\begin{prop}[Almost orthogonality of nonlinear profiles]\label{prop:almostorth}
Define $U_k^{\alpha}$, $U_k^{\beta}$ as the maximal life-span $I_k$  solutions of \eqref{NLS} with  initial data $U_k^{\alpha}(0) = \widetilde{\psi}_{\mathcal{O}_k^{\alpha}}^{\alpha}$, $U_k^{\beta}(0) = \widetilde{\psi}_{\mathcal{O}_k^{\beta}}^{\beta}$, where $\mathcal{O}^{\alpha}$ and $\mathcal{O}^{\beta}$ are orthogonal.  And define $G$ to be the solution of the maximal lifespan $I_0$ of \eqref{NLS} with initial data $G(0) =g$.  And $0\in I_k$ and $\lim_{k\to \infty} \abs{I_k} = 0$. Then 
\begin{align*}
\lim_{k\to \infty} \sup_{t\in I_k}\, \inner{ U_k^{\alpha}(t), U_k^{\beta} (t) }_{\dot{H}^{\frac{1}{2}} \times \dot{H}^{\frac{1}{2}}} = 0, \quad \lim_{k\to \infty} \sup_{t\in I_k\cap I_0} \inner{ U^\alpha_k(t), G(t) }_{\dot{H}^{\frac{1}{2}} \times \dot{H}^{\frac{1}{2}}} = 0.
\end{align*}
\end{prop}

\begin{proof}[Proof of Proposition  \ref{prop:almostorth}]
Set $U^0_k(0) = g$ and $U^0_k = G$ for all $k$, such that $U^0_k$ can be considered as a nonlinear profile with a trivial frame $\mathcal{O} =(1, 0, 0)_k$. 

For any $\theta > 0$, by the decomposition of the nonlinear profiles $U^{\alpha}$ and $U^{\beta}$ (Corollary \ref{cor Decomp nonlinear Euclidean profiles}), there exist $T_{\theta, \alpha}$, $R_{\theta, \alpha}$, $T_{\theta, \beta}$, $R_{\theta, \beta}$ sufficiently large
\begin{align*} 
U^{\alpha}_k &= \w_k^{\alpha, \theta,-\infty}+\w_k^{\alpha, \theta,+\infty}   +\w_k^{\alpha, \theta}+\rho_k^{\alpha, \theta},\\
U^{\beta}_k &= \w_k^{\beta, \theta,-\infty}+\w_k^{\beta, \theta,+\infty}  +\w_k^{\beta, \theta}+\rho_k^{\beta, \theta}.
\end{align*}
For $U_k^0$,   set $U_k^0 := \w_k^{\alpha,\theta,-\infty}+\w_k^{\alpha,\theta,+\infty}  +\w_k^{\alpha,\theta}+\rho_k^{\alpha, \theta}$ where $\rho_{k}^{0, \theta} =  \w_k^{0,\theta} = 0$ and $\w^{0,\theta,+\infty}=\w^{0,\theta,-\infty}=\frac{1}{2}G$. And by taking $T_{\theta,0}$ large, it is easy to make $\norm{G}_{Z'((-T_{\theta,0},T_{\theta,0}))}\leq \theta$. So   $\inner{ U_k^{\alpha} (t), G(t) }_{\dot{H}^{\frac{1}{2}} \times \dot{H}^{\frac{1}{2}}}$  can be considered as a special case of $\inner{ U_k^{\alpha} (t), U_k^{\beta}(t) }_{\dot{H}^{\frac{1}{2}}\times \dot{H}^{\frac{1}{2}}}$ when $\beta = 0$.

Since $\rho_k^{\alpha,\theta}$, $\rho_k^{\beta,\theta}$ are the small term with the $X^{\frac{1}{2}}$-norm  less then $\theta$, for any fixed $t\in I_k$, it will suffice to consider the following three terms:
\begin{enumerate}[\bf (1)]
\item $\inner{ \w_k^{\alpha, \theta,\pm\infty} , \w_k^{\beta, \theta,\pm\infty} }_{\dot{H}^{\frac{1}{2}} \times \dot{H}^{\frac{1}{2}}}$;
\item $\inner{\w_k^{\alpha, \theta,\pm\infty} , \w_k^{\beta, \theta} }_{\dot{H}^{\frac{1}{2}}\times \dot{H}^{\frac{1}{2}}}$;
\item $\inner{\w_k^{\alpha, \theta} , \w_k^{\beta, \theta} }_{\dot{H}^{\frac{1}{2}}\times \dot{H}^{\frac{1}{2}}}$.
\end{enumerate}
\case{1}{$\inner{ \w_k^{\alpha, \theta,\pm\infty} , \w_k^{\beta, \theta,\pm\infty} }_{\dot{H}^{\frac{1}{2}} \times \dot{H}^{\frac{1}{2}}}$.}
By the constructions of $\w_k^{\alpha, \theta,\pm\infty}, \w_k^{\beta, \theta,\pm\infty}$ in the proof of Corollary \ref{cor Decomp nonlinear Euclidean profiles}, we obtain that 
\begin{align*}
\w_k^{\alpha, \theta,\pm\infty} := \mathds{1}_{\{  \pm(t-t_k^{\alpha}) \geq T_{\alpha, \theta}(N^{\alpha}_k)^{-2}, \abs{t} \leq T_{\alpha, \theta}^{-1}\}}  \parenthese{\Pi_{t_k^{\alpha}-t, x_k^{\alpha}} T_{N_k^{\alpha}}\phi^{\alpha, \theta,\pm\infty}},\\
\w_k^{\beta, \theta,\pm\infty} := \mathds{1}_{\{ \pm(t-t_k^{\beta}) \geq T_{\beta, \theta}(N^{\beta}_k)^{-2}, \abs{t} \leq T_{\beta, \theta}^{-1}\}} \parenthese{\Pi_{t_k^\beta-t, x_k^{\beta}} T_{N_k^{\beta}}\phi^{\beta, \theta,\pm\infty} }.
\end{align*}
For any fixed $t\in I_k$, we obtain that 
\begin{align*}
\inner{ \w_k^{\alpha, \theta,\pm\infty}(t) , \w_k^{\beta, \theta,\pm\infty}(t) }_{\dot{H}^\frac{1}{2}\times \dot{H}^{\frac{1}{2}}} = \inner{ \phi^{\alpha, \theta,\pm\infty}_{\mathcal{O}_k^{\alpha}}, \phi^{\beta, \theta,\pm\infty}_{\mathcal{O}_k^{\beta}} }_{\dot{H}^{\frac{1}{2}} \times \dot{H}^{\frac{1}{2}}}.
\end{align*}
By \eqref{eq:almostorth1} of Proposition \ref{prop:equivalenceFrames}, we obtain that 
\begin{align*}
\lim_{k\to\infty} \sup_t \inner{\w_k^{\alpha, \theta,\pm\infty}(t) , \w_k^{\beta, \theta,\pm\infty}(t) }_{\dot{H}^{\frac{1}{2}} \times \dot{H}^{\frac{1}{2}}} = 0.
\end{align*}

\case{2}{$\inner{ \w_k^{\alpha, \theta,\pm\infty} , \w_k^{\beta, \theta} }_{\dot{H}^{\frac{1}{2}} \times \dot{H}^{\frac{1}{2}}}$.} 
By the constructions of $\w_k^{\alpha, \theta,\pm\infty}, \w_k^{\beta, \theta,\pm\infty}$ in the proof of Corollary \ref{cor Decomp nonlinear Euclidean profiles}, we obtain that 
\begin{align*}
\w_k^{\beta, \theta} := \widetilde{u}_k^{\beta} \cdot \mathds{1}_{S_k^{ \beta, \theta}},
\end{align*}
where
\begin{align*}
S_k^{\beta, \theta} :=\{ (t, x)\in  (-T_{\beta,\theta}, T_{\beta, \theta}) \times \T^2   : |t-t_k^{\beta}| < T_{\beta, \theta}(N_k^{\beta})^{-2},\ |x-x_k^{\beta}|\leq R_{\beta, \theta}(N_k^{\beta})^{-1}\}
\end{align*}
and $\widetilde{u}_k^{\beta}$ is defined in \eqref{eq:tildeu}.
Following a similar proof of the \textbf{Case 4} in the proof of \eqref{eq:6.9} in Subsection \ref{ssec A2}, we have that $\lim_{k\to \infty}\sup_t\langle\w_k^{\alpha, \theta,\pm\infty} , \w_k^{\beta, \theta} \rangle_{\dot{H}^\frac{1}{2}\times \dot{H}^\frac{1}{2}} = 0$.
    
\case{3}{$\inner{\w_k^{\alpha, \theta} , \w_k^{\beta, \theta} }_{\dot{H}^{\frac{1}{2}} \times \dot{H}^{\frac{1}{2}}}$.} 
For $\ee > 0$ small.

If $N_{k}^{\alpha}/ N_{k}^{\beta}+ N_{k}^{\beta} / N_{k}^{\alpha} \leq \ee^{-1000}$ and $k$ is large enough then $S_{k}^{\alpha, \theta} \cap S_{k}^{\beta, \theta}= \emptyset$. (By the definition of orthogonality of frames,  $N_{k}^{\alpha} /N_{k}^{\beta} + N_{k}^{\beta}/N_{k}^{\alpha} \leq \ee^{-1000}$ implies $(N_{k}^{\alpha})^2 |t_{k}^{\alpha}-t_{k}^{\beta}| \to \infty$ or $N_{k}^{\alpha} |x_{k}^{\alpha}-x_{k}^{\beta}|\to \infty$, so $S_{k}^{\alpha, \theta} \cap S_{k}^{\beta, \theta} = \emptyset$.)
In this case, $\w_{k}^{\alpha,\theta} \w_{k}^{\beta,\theta} \equiv 0$.

If $N_{k}^{\alpha}/N_{k} ^{\beta}\geq \ee^{-1000}/2$.
Denote that
\begin{align*}
\w_{k}^{\alpha,\theta} \w_{k}^{\beta,\theta}= \w_{k}^{\alpha,\theta} \widetilde{ \w}_{k}^{\beta,\theta}:=
\w_{k}^{\alpha,\theta} \cdot (\w_{k}^{\beta,\theta} \mathds{1}_{(t_{k}^{\alpha}-T_{\alpha,\theta} (N_{k}^{\alpha})^{-2}, t_{k}^{\alpha}+T_{\alpha,\theta} (N_{k}^{\alpha})^{-2})}(t)).
\end{align*}
By $\ee^{10}N_k^{\alpha}>>\ee^{-10}N_k^{\beta}$ and the \textit\textbf{Claim $\dagger$} in the proof of Lemma \ref{lem:7.2}, we obtain that
\begin{align*}
\inner{ \w_{k}^{\alpha,\theta}, \w_{k}^{\beta,\theta}}_{\dot{H}^{\frac{1}{2}} \times\dot{H}^{\frac{1}{2}}} & \leq \inner{ P_{\leq \ee^{10}N_k^{\alpha}}\w^{\alpha,\theta}_{k},  \w_{k}^{\beta,\theta}}_{\dot{H}^{\frac{1}{2}} \times \dot{H}^{\frac{1}{2}}}
  +\inner{ P_{> \ee^{10} N_k^{\alpha}} \w_{k}^{\alpha,\theta}, P_{>\ee^{-10} N_k^{\beta}} \w^{\beta,\theta}_{k} }_{\dot{H}^{\frac{1}{2}} \times \dot{H}^{\frac{1}{2}}}\\
&+\inner{ P_{> \ee^{10}N_k^{\alpha}} \w_{k}^{\alpha,\theta}, \w_{k}^{\beta,\theta}}_{\dot{H}^{\frac{1}{2}} \times \dot{H}^{\frac{1}{2}}} \lesssim \ee.
\end{align*}
Now we finish the proof of Proposition \ref{prop:almostorth}.
\end{proof}

\section{Proof of Theorem \ref{thm Main}}\label{sec Rigidity}
It suffice to prove the solutions remain bounded in $Z$-norm on intervals of
length at most 1. To obtain this, we run the induction on  $\norm{u}_{L_t^{\infty} H_x^{\frac{1}{2}}}$.

\begin{defn}\label{defn Lambda}
Define
\begin{align*}
\Lambda (L,\tau) =  \sup \bracket{ \norm{u}_{Z(I)}:  u: I \times \T^2 \to \C \text{ solves \eqref{NLS} with } \sup_{t\in I} \norm{u(t)}_{H^{\frac{1}{2}}(\T^2)}^2<L, \abs{I}\leq \tau} 
\end{align*}
where $u$ is any strong solution of \eqref{NLS} with initial data $u_0$ in interval $I$ of length $\abs{I} \leq \tau$.
\end{defn}
We observe that
\begin{itemize}
\item
$\Lambda$ is an increasing function of both $L$ and $\tau$;
\item
by the definition we have the sublinearity of $\Lambda$ in $\tau$: $\Lambda (L, \tau + \sigma) \leq \Lambda (L, \tau) + \Lambda (L, \sigma)$.
\end{itemize}
Hence we define
\begin{align*}
\Lambda_0 (L) = \lim_{\tau\to 0} \Lambda (L,\tau),
\end{align*}
then for all $\tau$, we have that  $\Lambda(L, \tau)<+\infty $ is equivalent to $ \Lambda_0 (L)<+\infty$.
Finally, we define
\begin{align}\label{eq Emax}
E_{max} = \sup\{L: \Lambda_0 (L)<+\infty\}.
\end{align}

\begin{thm}\label{thm:main2}
Consider $E_{max}$ defined as in \eqref{eq Emax},  $E_{max} = +\infty$.
\end{thm}

Notice that Theorem \ref{thm:main2} implies Theorem \ref{thm Main}. Hence in the rest of this paper, we will address the proof of Theorem \ref{thm:main2}. To prove this theorem, we argue by contradiction and assume that $E_{max}<+\infty$.  The main strategy is that, by using the profile decomposition, we analyze the following three scenarios respectively,
\begin{enumerate}[\bf (1)]
\item
no Euclidean profiles
\item
exactly one Euclidean profile, but no scale-one profile
\item
multiple (Euclidean/scale-one) profiles
\end{enumerate}
then rule out the possibility of all these cases to obtain the contradiction to the existence of such $E_{max}$. Notice that the first two cases are comparably easier. Now we prove Theorem \ref{thm:main2} via profile decomposition.

\begin{proof}[Proof of Theorem \ref{thm:main2}]
Assume that 
\begin{align}\label{eq Eass}
E_{max}<+\infty. 
\end{align}
By the definition of $E_{max}$, there exists a sequence of solutions $u_k $ with their lifespan $I_k = [-T_ k,T^k]$ such that
\begin{align}\label{6.2}
\begin{aligned}
& \sup_{t\in [-T_k, T^k]}\norm{u_k(t)}_{H^{\frac{1}{2}}(\T^2)}   \to E_{max},   \\
& \norm{u_k}_{Z(-T_k, 0)} \to +\infty ,\\
&  \norm{u_k}_{Z(0,T^k)} \to +\infty .
\end{aligned}
\end{align}
for some $T_k,\ T^k\to 0$ as $k\to +\infty$. For  simplicity of notations, we denote 
\begin{align*}
L(\phi) :=   \sup_{t\in I_k} \norm{u_{\phi}(t)}_{H^{\frac{1}{2}}(\T^2)}^2,
\end{align*}
where $u_{\phi}(t)$ is the solution of \eqref{NLS} with initial data $u_{\phi}(0) =\phi$. By the Proposition \ref{prop:ProfileDecomposition}, after extracting a subsequence, \eqref{6.2} gives a sequence of profiles $\widetilde{\psi}_k^{\alpha}$, where $\alpha, k = 1, 2,\cdots$ and a decomposition
\begin{align*}
u_k(0) = g + \sum_{1 \leq \alpha \leq J}  \widetilde{\psi}_k^{\alpha} + R^J_k ,
\end{align*}
satisfying
\begin{align}\label{63}
\limsup_{J\to\infty} \limsup_{k\to\infty} \norm{e^{it\DD}R_k^J}_{Z(I_k)} =0.
\end{align}
Recall that $g$ is the so-called scale-one profile and $\widetilde{\psi}_k^{\alpha} $'s are Euclidean profiles.
Moreover the almost orthogonality in  Proposition \ref{prop:ProfileDecomposition} and the almost orthogonality of nonlinear profiles in Lemma \ref{prop:almostorth}, we obtain that 
\begin{align}\label{eq:EnergyDecoupling}
\begin{aligned}
& \mathcal{L}(\alpha) := \lim_{k\to+\infty} L(\widetilde{\psi}_{k}^{\alpha})  \in [0,E_{max}],\\
& \lim_{J\to J^*}  \parenthese{  \sum_{1\leq \alpha \leq J} \mathcal{L}(\alpha) +\lim_{k\to\infty} L(R_k^J) }+ L(g)= E_{max},
\end{aligned}
\end{align}

\case{1}{There are no Euclidean profiles.} 
In fact, this case contains two scenarios: (1) there are no scale-one profiles neither ($g=0$), (2) there is a nontrivial scale-one profile ($g\neq 0$). Our argument works for both scenarios. 

Notice that there is no any Euclidean profiles, and  $\norm{g}_{H_x^{\frac{1}{2}}(\T^2)} \lesssim L(g) \leq E_{max}$. Then, by $I_k \to 0$ as $k \to \infty$, there exists $\eta >0$,  such that for $k$ large enough
\begin{align*}
\norm{e^{it\Delta} u_k(0)}_{Z(I_k)} \leq \norm{e^{it\Delta} g}_{Z((-\eta, \eta))} + \ee \leq \delta_0,
\end{align*}
where $\delta_0$ is given by the local theory in Proposition \ref{prop LWP}. In this case, we conclude that $\norm{u_k}_{Z(I_k)} \lesssim 2 \delta$ which contradicts \eqref{6.2}.

\case{2}{There is no scale-one profile, but exactly one Euclidean profile.} That is, $g=0$ and only one Euclidean profile $\widetilde{\psi}_k^1$ such that $\mathcal{L}(1) = E_{max}$.
By Proposition \ref{prop:ProfileDecomposition} and \eqref{eq:EnergyDecoupling}, we obtain that
\begin{align*}
L(\widetilde{\psi}_k^1) \leq E_{max}
\end{align*}
which implies
\begin{align*}
\sup_{t}\norm{u_{\psi}}_{\dot{H}^{\frac{1}{2}}(\R^2)}<\infty .
\end{align*}
Let $U_k^1$ be  the solution of \eqref{NLS} with $U_k^1(0)= \widetilde{\psi}_k^1$. In this case, we use the {\bf Part (1)} of Proposition \ref{prop Euclidean profiles}, given some $\varepsilon>0$, for $k$ large enough, we have that
\begin{align}\label{eq:case2}
\begin{aligned}
& \norm{U^1_k}_{X^{\frac{1}{2}}(I_k)} \leq \norm{U^1_k}_{X^{\frac{1}{2}}((-\dd, \dd))} \lesssim 1, \\
& \norm{U_k^1(0) - u_k(0)}_{H_x^{\frac{1}{2}}(\T^2)} \leq \epsilon .
\end{aligned}
\end{align}
By \eqref{eq:case2} and Proposition \ref{prop Stability}, we obtain that
\begin{align}
\norm{u_k}_{Z(I_k)} \lesssim \norm{u_k}_{X^{\frac{1}{2}}(I_k)} \lesssim 1,
\end{align}
which contradicts \eqref{6.2}.

\case{3}{There are at least two nonzero profiles in the decomposition.}
By \eqref{eq:EnergyDecoupling}, we know that $L(g) <E_{max}$ and $\mathcal{L}(\alpha) <E_{max}$ for all $\alpha$. By almost orthogonality and relabeling the profiles, we can assume that for all $\alpha$,
\begin{align*}
\begin{aligned}
& \mathcal{L}(\alpha) \leq \mathcal{L}(1) < E_{max} -\eta, \\
& L(g)<E_{max} -\eta, 
\end{aligned}
\end{align*}
for  some $\eta>0$.
Define $U_k^{\alpha}$ as the maximal life-span solution of \eqref{NLS} with  initial data $U_k^{\alpha}(0) = \widetilde{\psi}_k^{\alpha}$  and $G$ to be the maximal life-span solution  of \eqref{NLS} with initial  data $G(0) =g$.

By Definition \ref{defn Lambda} and the assumption on $E_{max}$ in \eqref{eq Eass}, we have
\begin{align*}
\norm{G}_{Z((-1,1))} + \lim_{k\to\infty} \norm{U_k^{\alpha}}_{Z((-1,1))} \leq  2 \Lambda(E_{max}-\eta/2, 2) \lesssim 1.
\end{align*}
By Proposition \ref{prop LWP}, it follows that for any $\alpha$ and any $k>k_0(\alpha)$ sufficient large,
\begin{align*}
\norm{G}_{X^{\frac{1}{2}}((-1,1))} +\norm{U_k^{\alpha}}_{X^{\frac{1}{2}}((-1, 1))}\lesssim 1.
\end{align*}
For $J, k \geq 1$, we define
\begin{align*}
U_{prof, k}^J:= G + \sum_{\alpha =1}^J U_k^{\alpha} = \sum_{\alpha = 0}^J U_k^\alpha.
\end{align*}
where we denote $U_k^0 := G$.

Now we claim the following.
\begin{claim}\label{claim2}
There is a constant $Q$ such that
\begin{align*}
\norm{U^J_{prof, k}}_{X^{\frac{1}{2}}((-1,1))}^2 +\sum_{\alpha=0}^J \norm{U_k^{\alpha}}_{X^{\frac{1}{2}}((-1,1))}^2   \leq Q^2,
\end{align*}
uniformly on $J$.
\end{claim}
\begin{proof}[Proof of Claim \ref{claim2}]
From \eqref{63} we know that there are only finite many profiles such that  $\mathcal{L}(\alpha)\geq \frac{\dd_0}{2}$. We may assume that for all $\alpha\geq A$,  $\mathcal{L}(\alpha) \leq\dd_0$. Consider $U_k^{\alpha}$ for $k \geq A$, by  Proposition \ref{prop LWP}, we have that
\begin{align*}
&\quad \norm{U_{prof, k}^J}_{X^{\frac{1}{2}}((-1,1))} = \norm{\sum_{0\leq\alpha\leq J}   U_k^\alpha}_{X^{\frac{1}{2}}((-1,1))} \\
& \leq \sum_{0\leq \alpha \leq A} \norm{U_k^{\alpha}}_{X^{\frac{1}{2}}((-1,1))}  +\norm{\sum_{A\leq \alpha\leq J} (U_k^{\alpha} -e^{it\Delta}U_k^{\alpha}(0))}_{X^{\frac{1}{2}}((-1,1))}  +\norm{e^{it\Delta} \sum_{A\leq \alpha \leq J} U_k^{\alpha}(0)}_{X^{\frac{1}{2}}((-1,1))}\\
& \lesssim (A+1)+\sum_{A\leq\alpha\leq J} \norm{U_k^{\alpha}(0)}_{H^{\frac{1}{2}}}^2 + \norm{\sum_{A\leq\alpha\leq J} U_k^{\alpha}(0)}_{H^{\frac{1}{2}}}\\
& \lesssim (A+1) + \sum_{A\leq\alpha\leq J} \mathcal{L}(\alpha) + E_{max}^{\frac{1}{2}}\\
& \lesssim 1.
\end{align*}
Similarly, we write
\begin{align*}
\sum_{\alpha = 0}^{J} \norm{U_k^{\alpha}}_{X^{\frac{1}{2}}((-1,1))}^2 &= \sum_{\alpha =0}^{A-1} \norm{U_k^{\alpha}}_{X^{\frac{1}{2}}((-1,1))}^2 + \sum_{A\leq\alpha\leq J}\norm{U_k^{\alpha}}_{X^{\frac{1}{2}}((-1,1))}^2  \lesssim A + \sum_{A \leq \alpha \leq J} \mathcal{L}(\alpha) \lesssim 1.
\end{align*}
Now we finish the proof of Claim \ref{claim2}.
\end{proof}

Let us continue the proof of Theorem \ref{thm:main2}. Let 
\begin{align*}
U_{app, k}^J = \sum_{0\leq \alpha\leq J} U_k^{\alpha} + e^{it\Delta}R_k^J
\end{align*}
to be a solution of the approximation equation \eqref{eq Approx NLS} with
the error term:
\begin{align*}
e &= (i\partial_t +\Delta_{\T^2}) U_{app, k}^J - F(U_{app, k}^J) = \sum_{0 \leq \alpha \leq J} F(U_k^{\alpha}) - F(\sum_{0\leq \alpha \leq J} U_k^{\alpha} +e^{it\Delta}R_k^J),
\end{align*}
where $F(u) = \abs{u}^4 u$.

From Claim \ref{claim2}, we know $\norm{U_{app, k}^J}_{X^{\frac{1}{2}}((-1,1))}\leq Q$. To continue, we need the following lemma which will be proved in Appendix \ref{sec Apx}.
\begin{lem}\label{lem:6.2}
With the same notation, we obtain that
\begin{align*}
\limsup_{J \to \infty} \limsup_{k\to \infty}   \norm{\sum_{0\leq \alpha\leq J} F(U_k^{\alpha}) - F(\sum_{0\leq \alpha\leq J} U_k^{\alpha}  +e^{it\Delta}R_k^J)}_{N(I_k)} = 0.
\end{align*}
\end{lem}
Assuming Lemma \ref{lem:6.2}, we obtain that for $J\geq J_0(\ee)$
\begin{align*}
\limsup_{k\to\infty} \norm{e}_{N(I_k)} \leq \ee/2.
\end{align*}
We use Proposition \ref{prop Stability} to conclude that $u_k$ satisfies
\begin{align*}
\norm{u_k}_{X^1(I_k)} \lesssim \norm{U_{app, k}^J}_{X^{\frac{1}{2}}(I_k)} \leq
\norm{U_{prof, k}^J}_{X^{\frac{1}{2}}((-1,1))} \norm{e^{it\Delta}R_k^J}_{X^{\frac{1}{2}}((-1,1))} \lesssim 1.
\end{align*}
which contradicts \eqref{6.2}.

Now we finish the proof of Theorem \ref{thm:main2}. 
\end{proof}

\appendix
\section{Proof of Lemma \ref{lem:6.2}}\label{sec Apx}
We first write that
\begin{align*}
& \quad \norm{\sum_{0\leq \alpha\leq J} F(U_k^\alpha) - F(U_{prof, k}^J +e^{it\DD}R_k^J)}_{N(I_k)} \\
& \leq \norm{F(U_{prof, k}^J+e^{it\Delta}R_k^J))-F(U_{prof, k}^J)}_{N(I_k)}+\norm{F(U_{prof, k}^J) -\sum_{0 \leq \alpha \leq J} F(U_k^{\alpha})}_{N(I_k)}.
\end{align*}
It will suffice to prove
\begin{align}\label{eq:6.8}
\limsup_{J\to\infty}\limsup_{k\to \infty} \norm{F(U_{prof, k}^J+e^{it\Delta}R_k^J))-F(U_{prof, k}^J)}_{N(I_k)} = 0,
\end{align}
and
\begin{align}\label{eq:6.9}
\limsup_{J\to\infty}\limsup_{k\to \infty} \norm{F(U_{prof, k}^J) -\sum_{0\leq \alpha\leq J} F(U_k^\alpha)}_{N(I_k)} = 0.
\end{align}
Lemma \ref{lem:7.1} and Lemma \ref{lem:7.2} are needed in the proof of \eqref{eq:6.8} and \eqref{eq:6.9}. Hence let us first prove Lemma \ref{lem:7.1} and Lemma \ref{lem:7.2}.

\subsection{Lemmas needed in \eqref{eq:6.8} and \eqref{eq:6.9}}
Let  $\oo_{p,q}(a,b)$ be a $(p+q)$ - linear expression with $p$ factors consisting of either $\overline{a}$ or $a$ and $q$ factors consisting of either $\overline{b}$ or $b$.
\begin{lem}[a high-frequency linear solution does not interact significantly with a low-frequency profile]\label{lem:7.1}
Let $B, N\geq 2$ be dyadic numbers. Assume that $\w : (-1, 1) \times \T^2 \to \C$ is a function satisfying
\begin{align*}
\abs{\nabla^j \w} \leq N^{j+\frac{1}{2}} \mathds{1}_{ \{ \abs{t} \leq N^{-2}  , \abs{x} \leq N^{-1}\} }, \quad j=0, 1.
\end{align*}
Then the following holds
\begin{align*}
\norm{\oo_{4,1}(\w, e^{it\Delta} P_{>BN} f)}_{N((-1,1))} \lesssim (B^{-1/{200}}+N^{-1/{200}}) \norm{f}_{H_x^{\frac{1}{2}}(\T^2)}.
\end{align*}
\end{lem}

\begin{proof}[Proof of Lemma \ref{lem:7.1}]
Without loss of generality, we assume that $\norm{f}_{H^{\frac{1}{2}} (\T^2)} =1 $ and $f = P_{> BN} f$. Using \eqref{eq Embed}, we have that 
\begin{align*}
&\quad \norm{\oo_{4,1} (\w, e^{it \Delta} P_{>BN} f)}_{N((-1,1))} \lesssim \norm{\oo_{4,1} (\w , e^{it\Delta} P_{>BN} f)}_{L_t^1 H_x^{\frac{1}{2}}} \\
& \lesssim \norm{\oo_{4,1} (\w , e^{it\Delta} P_{>BN} \abs{\nabla}^{\frac{1}{2}} f)}_{L_t^1 L_x^2} + \norm{e^{it\Delta} f}_{L_t^{\infty} L_x^2} \norm{\w}_{L_t^4 L_x^{\infty}}^3 \norm{\abs{\abs{\nabla}^{\frac{1}{2}} \w } + \abs{\w} }_{L_t^4 L_x^{\infty}} \\
& \lesssim \norm{\oo_{4,1} (\w , \abs{\nabla}^{\frac{1}{2}} e^{it\Delta} f)}_{L_t^1 L_x^2} + B^{-\frac{1}{2}} .
\end{align*}
The last inequality above holds since
\begin{align*}
&\norm{\w}_{L_t^4 L_x^{\infty}}  \leq \parenthese{\int_{\abs{t} \leq N^{-2}} \abs{N^{\frac{1}{2}}}^4 \, dt }^{\frac{1}{4}} \lesssim 1\\
& \norm{\abs{\nabla}^{\frac{1}{2}}  \w}_{L_t^4 L_x^{\infty}}   \leq \parenthese{\int_{\abs{t} \leq N^{-2}} \abs{N}^4 \, dt }^{\frac{1}{4}} \lesssim N^{\frac{1}{2}}\\
& \norm{e^{it\Delta}  f}_{L_t^{\infty} L_x^2}  \leq \norm{P_{> BN} f}_{L_x^2} \leq (BN)^{-\frac{1}{2}} .
\end{align*}
Now let $W (x,t) : = N^4 \eta_{\R^2} (N \psi^{-1} (x)) \cdot \eta_{\R} (N^2 t)$, then we write
\begin{align*}
& \quad \norm{\oo_{4,1} (\w, \nabla e^{it\Delta} f)}_{L^1((-1,1), L_x^2)}^2 \lesssim \norm{\w}_{L_t^8 L_x^{\infty}}^8 N^{-4} \norm{W^{\frac{1}{2}} \nabla e^{it\Delta} f}_{L_{t,x}^2 ((-1, 1) \times \T^2 )}^2\\
& \lesssim N^{-2} \sum_{j=1}^2 \int_{-1}^1 \inner{e^{it\Delta} \partial_j f , W e^{it\Delta} \partial_j f}_{L^2 \times L^2} \, dt\\
& \lesssim N^{-2} \sum_{j=1}^2 \inner{ \partial_j f , \parenthese{ \int_{-1}^1 e^{-it\Delta} W e^{it\Delta} \, dt } \partial_j f}_{L^2 \times L^2(\T^2)} .
\end{align*}
Therefore, it remains to prove that 
\begin{align}\label{eq K bdd}
\norm{K}_{L^2 (\T^2) \to L^2 (\T^2)} \lesssim N^2 (B^{-1/100} + N^{-1/100}),
\end{align}
where $K= P_{> BN} \int_{\R} e^{-it\Delta} W e^{it\Delta} P_{>BN} \, dt$.

We compute the Fourier coefficients of $K$ as follows:
\begin{align*}
c_{p,q} & = \inner{e^{ipx} , Ke^{iqx}} = \int_{\T^2} \overline{P_{>BN} e^{ipx} } \int_{\R} e^{-it\Delta} W e^{it\Delta} P_{>BN} e^{iqx} \, dt dx \\
& = (1- \eta_{\R^2}) (\frac{p}{BN}) (1-\eta_{\R^2}) (\frac{q}{BN}) \int_{(-1,1) \times \T^2} e^{it (\abs{p}^2 -\abs{q}^2) + i (q-p) \cdot x} W(x,t) \, dxdt\\
& = (\mathcal{F}_{x,t} W) (p-q, \abs{q}^2 - \abs{p}^2) (1-\eta_{\R^2}) (\frac{p}{BN}) (1- \eta_{\R^2}) (\frac{q}{BN}).
\end{align*}
By the property of Fourier transform of $W$, we have that
\begin{align*}
\abs{c_{p,q}} \lesssim \parenthese{1+ \frac{\abs{\abs{p}^2 -\abs{q}^2}}{N^2}}^{-10} \parenthese{ 1+ \frac{\abs{p-q}}{N^2}}^{-10} \mathds{1}_{\{ \abs{p} \geq BN \}} \mathds{1}_{\{ \abs{q} \geq BN \}} .
\end{align*}
Then using Schur's lemma, we get
\begin{align*}
\norm{K}_{L^2(\T^2) \to L^(\T^2)} \lesssim \sup_{p \in \Z^2} \sum_{q \in \Z^2} \abs{c_{p,q}} + \sup_{q \in \Z^2} \sum_{p \in \Z^2} \abs{c_{p,q}}.
\end{align*}
Now to obtain \eqref{eq K bdd}, it is sufficient to prove the following:
\begin{align}\label{eq K bdd equiv}
N^{-2} \sup_{\abs{p} \geq BN} \sum_{v \in \Z^2} \parenthese{1 + \frac{\abs{\abs{p}^2 - \abs{p+v}^2}}{N^2}}^{-10}  \parenthese{1+\frac{\abs{v}}{N}}^{-10} \lesssim (B^{-1/100} + N^{-1/100}).
\end{align}
We will prove \eqref{eq K bdd equiv} in the following three cases:
\begin{enumerate}[\bf (1)]
\item $\abs{v} \geq NB^{1/100}$
\item $\abs{v} < NB^{1/100}$, $\abs{v \cdot p} \geq N^2 B^{1/10}$
\item $\abs{v} < NB^{1/100}$, $\abs{v \cdot p} < N^2 B^{1/10}$.
\end{enumerate}

Let us discuss case by case.

\case{1}{$\abs{v} \geq NB^{1/100}$.}
We compute 
\begin{align*}
& \quad \sum_{\abs{v} \geq NB^{1/100}} \parenthese{1 + \frac{\abs{\abs{p}^2 - \abs{p+v}^2}}{N^2}}^{-10}  \parenthese{1+\frac{\abs{v}}{N}}^{-10}  \lesssim \sum_{\abs{v} \geq NB^{1/100}} \parenthese{1+\frac{\abs{v}}{N}}^{-10} \\
& \lesssim \int_{\abs{v} \geq N, v \in \R^2} \parenthese{1+\frac{\abs{v}}{N}}^{-10}  \, dv  \lesssim \parenthese{1+\frac{NB^{1/100}}{N}}^{-8} \lesssim B^{-8/100}. 
\end{align*}

\case{2}{$\abs{v} < NB^{1/100}$, $\abs{v \cdot p} \geq N^2 B^{1/10}$.} 
We write
\begin{align*}
& \quad \sum_{\substack{\abs{v} \leq NB^{1/100} \\ \abs{v \cdot p} \geq N^2 B^{1/100}}}  \parenthese{1 + \frac{\abs{\abs{p}^2 - \abs{p+v}^2}}{N^2}}^{-10}  \parenthese{1+\frac{\abs{v}}{N}}^{-10} \\
& \lesssim \sum_{\substack{\abs{v} \leq NB^{1/100} \\ \abs{v \cdot p} \geq N^2 B^{1/100}}}  \parenthese{1+\frac{2\abs{v \cdot p}}{N}}^{-10} \lesssim (1 + B^{1/10})^{-8} \lesssim B^{-8/10} .
\end{align*}

\case{3}{$\abs{v} < NB^{1/100}$, $\abs{v \cdot p} < N^2 B^{1/10}$.}
Denote $\hat{p} = \frac{p}{\abs{p}}$ and write
\begin{align*}
& \quad N^{-2} \sup_{\abs{p} \geq BN} \sum_{\substack{\abs{v} \leq NB^{1/100} \\ \abs{v \cdot p} \leq N^2 B^{1/10}}}  \parenthese{1 + \frac{\abs{\abs{p}^2 - \abs{p+v}^2}}{N^2}}^{-10}  \parenthese{1+\frac{\abs{v}}{10}}^{-10} \\
& \leq N^{-2} \sup_{\abs{p} > BN} \sum_{\substack{\abs{v} \leq NB^{1/100} \\ \abs{v \cdot \hat{p}} \leq N B^{-9/10}}} 1  \leq N^{-2} \sup_{\abs{p} >BN} \# \bracket{v : \abs{v} \leq NB^{1/100} , \abs{v \cdot \hat{p}} \leq NB^{-9/10} }\\
& \leq N^{-2} (NB^{1/100}) NB^{-9/10} \leq B^{-89/100}.
\end{align*}
These three cases above give \eqref{eq K bdd equiv}. Then Lemma \ref{lem:7.1} follows.
\end{proof}

\begin{lem}\label{lem:7.2}
Assume that $\OO_\alpha = (N_{k,\alpha},t_{k,\alpha}, x_{k,\alpha})_k \in \mathcal{F}_e$, $\alpha\in\{1, 2\}$, are two orthogonal frames, $I \subseteq (-1, 1)$ is a fixed open interval, $0 \in I$, and $T_1$, $T_2$, $R \in [1,\infty)$ are fixed numbers, $R \geq T_1 + T_2$. For $k$ large enough,
for $\alpha \in\{1, 2\}$
\begin{align*}
\abs{ \nabla_x^m \w_k^{\alpha, \theta}}+(N_{k, \alpha})^{-2}\mathds{1}_{S_k^{\alpha, \theta}} \abs{\partial_t \nabla_x^m \w_k^{\alpha, \theta}} \leq R_{\theta, \alpha} (N_k^\alpha)^{\abs{m}+\frac{1}{2}}\mathds{1}_{S_k^{\alpha, \theta}},m \in \N^2, \,  0 \leq \abs{m} \leq 10,
\end{align*}
where
\begin{align*}
S_k^{\alpha,\theta} :=\{ (t,x) \in I \times \T^2  : \abs{t-t_{k,\alpha}} < T_{\alpha}(N_{k,\alpha})^{-2},\ \abs{ x-x_{k,\alpha}} \leq R_{\theta, \alpha} (N_{k,\alpha})^{-1}\}.
\end{align*}
Assume that $(\w_{k,1}, w_{k,2}, f_k, g_k, h_k)_k$ are a sequence of quintuple functions with properties $\norm{f_k}_{X^1(I)}\leq 1$, $\norm{g_k}_{X^1(I)}\leq 1$ and $\norm{h_k}_{X^1(I)}\leq 1$ for all $k$ large enough. 

Then
\begin{align*}
 \limsup_{k\to\infty} \norm{\w_{k, 1}\, \w_{k, 2}\, f_k\,g_k\, h_k}_{N(I)} = 0.
\end{align*}
\end{lem}

\begin{proof}[Proof of Lemma \ref{lem:7.2}]
Fix $\varepsilon >0$ small. We consider the following two cases.
\begin{enumerate}
\item
If 
\begin{align*}
\frac{N_{k,1}}{N_{k,2}} + \frac{N_{k,1}}{N_{k,1}} \leq \varepsilon^{-1000}
\end{align*}
and $k$ is large enough then $S_{k,1} \cap S_{k,2} =  \varnothing$. In this case $\w_{k, 1}\, \w_{k, 2}\, f_k\,g_k\, h_k \equiv 0$. 
\item
Otherwise, if
\begin{align*}
\frac{N_{k,1}}{N_{k,2}} \geq \frac{1}{2} \varepsilon^{-1000},
\end{align*}
we denote that
\begin{align*}
\w_{k,1} \w_{k,2} = \w_{k.1} \widetilde{\w}_{k,2} : = \w_{k,1} \parenthese{\w_{k,2} \mathds{1}_{\{ (t_{k,1}-T_1N_{k,1}^{-2} , t_{k,1} + T_1 N_{k,1}^{-2}) \} } (t)} .
\end{align*}
\end{enumerate}

Here we claim the following properties of the quintuple functions.
\begin{claim}\label{claim1}
For $k$ large enough, we have
\begin{enumerate}[\bf (1)]
\item
$\norm{\widetilde{\w}_{k,2}}_{X^{\frac{1}{2}} (I)} \lesssim_R 1$,
\item
$\norm{P_{> \varepsilon^{-10} N_{k,2}} \widetilde{\w}_{k,2}}_{X^{\frac{1}{2}}(I)} \lesssim_R \varepsilon$,
\item
$\norm{\widetilde{\w}_{k,2}}_{Z(I)} \lesssim_R \varepsilon$,
\item
$\norm{\w_{k,1}}_{X^{\frac{1}{2}}(I)} \lesssim_R 1$,
\item
$\norm{P_{\leq \varepsilon^{10} N_{k,1}} \w_{k,1}}_{X^{\frac{1}{2}}(I)} \lesssim \varepsilon$.
\end{enumerate}
\end{claim}

Assuming Claim \ref{claim1} and using Lemma \ref{lem Nonlinear est} and $\varepsilon^{10} N_1 \gg \varepsilon^{-10} N_2$, we obtain that
\begin{align*}
 \norm{\w_{k,1} \, \w_{k,2} \, f_k \, g_k \, h_k}_{N(I)} & \leq \norm{(P_{\leq \varepsilon^{10} N_{k,1}} \w_{k,1}) \widetilde{\w}_{k,2} \, f_k \, g_k \, h_k}_{N(I)}\\
& \quad + \norm{(P_{> \varepsilon^{10} N_{k,1}} \w_{k,1}) (P_{> \varepsilon^{-10} N_{k,2}} \w_{k,2} ) \, f_k \, g_k \, h_k}_{N(I)}  \\
& \quad + \norm{(P_{> \varepsilon^{10} N_{k,1}} \w_{k,1}) (P_{\leq \varepsilon^{-10} N_{k,2}} \w_{k,2} ) \, f_k \, g_k \, h_k}_{N(I)}\\
& \lesssim _R \varepsilon,
\end{align*}
which implies Lemma \ref{lem:7.2}.

Now we are left to prove Claim \ref{claim1}.
\begin{proof}[Proof of Claim \ref{claim1}]
We show the properties in  this claim one by one.
\begin{enumerate}[\bf (1)]
\item
Using the integral equation, we write
\begin{align*}
\norm{\widetilde{\w}_{k,2}}_{X^{\frac{1}{2}}(I)} & \lesssim \norm{\widetilde{\w}_{k,2} (0)}_{H^{\frac{1}{2}}} + \norm{(i\partial_t + \Delta) \widetilde{\w}_{k,2}}_{L_t^1 H_x^{\frac{1}{2}} (I \times \T^2)} \\
& \lesssim \parenthese{\int_{\abs{x-x_{k,2}} \leq R N_{k,2}^{-1}} \abs{\abs{\nabla}^{\frac{1}{2}} \widetilde{\w}_{k,2} (0)}^2 \, dx }^{\frac{1}{2}} + \norm{(i \partial_t + \Delta) \widetilde{\w}_{k,2} }_{L_t^1 H_x^{\frac{1}{2}} (I \times \T^2)}\\
& \lesssim R^2 + \int_{\abs{t- t_{k,2} }\leq T_2 N_{k,2}^{-2} } \norm{\partial_t \widetilde{\w}_{k,2}}_{H_x^{\frac{1}{2}}} + \norm{\Delta \widetilde{\w}_{k,2}}_{H_x^{\frac{1}{2}}} \, dt \\
& \lesssim R^2 + R^2 T \lesssim_R 1 .
\end{align*}

\item
For the high frequency part of $\widetilde{\w}_{k,2}$, using the integral equation again, we have
\begin{align*}
\norm{P_{> \varepsilon^{-10}N_{k,2}} \widetilde{\w}_{k,2}}_{X^{\frac{1}{2}} (I)} & \lesssim \norm{P_{> \varepsilon^{-10} N_{k,2}}  \widetilde{\w}_{k,2} (0)}_{H^{\frac{1}{2}}} + \norm{(i\partial_t + \Delta) P_{> \varepsilon^{-10} N_{k,2} } \widetilde{\w}_{k,2}}_{L_t^1 H^{\frac{1}{2}}} \\
& \lesssim \frac{\varepsilon^{10}}{N_{k,2}} \norm{\abs{\nabla}^{\frac{3}{2}} \widetilde{\w}_{k,2} (0)}_{L_x^2} + \frac{\varepsilon^{10}}{N_{k,2}} \norm{(i\partial_t +\Delta) \nabla \widetilde{\w}_{k,2}}_{L_t^1 H_x^{\frac{1}{2}}}\\
& \lesssim \frac{\varepsilon^{10}}{N_{k,2}} N_{k,2} (R^2 + R^2 T) \lesssim_R \varepsilon^{10}.
\end{align*}

\item
Consider the $Z$-norm of $\widetilde{\w}_{k,2}$. For $p= p_0, p_1$, 
\begin{align*}
\norm{\widetilde{\w}_{k,2}}_{Z(I)} & \lesssim \parenthese{\sum_N  N^{4-\frac{p}{2}} \norm{P_N \widetilde{\w}_{k,2}}_{L_{t,x}^p (\T^2 \times (t_{k,1} - RN_{k,1}^{-2} , t_{k,1} + RN_{k,1}^{-2}))}^p  }^{\frac{1}{p}}\\
& \lesssim \parenthese{\sum_N  \norm{\abs{\nabla}^{\frac{4}{p}-\frac{1}{2}}  P_N \widetilde{\w}_{k,2}}_{L_{t,x}^p (\T^2 \times (t_{k,1} -RN_{k,1}^{-2} , t_{k,1} + RN_{k,1}^{-2}))}^p  }^{\frac{1}{p}}\\
& \lesssim   \norm{ \parenthese{ \sum_N \abs{\nabla}^{\frac{4}{p}-\frac{1}{2}}  P_N \widetilde{\w}_{k,2} }^{\frac{1}{2}}}_{L_{t,x}^p (\T^2 \times (t_{k,1} -RN_{k,1}^{-2} , t_{k,1} + RN_{k,1}^{-2}))}\\
& \lesssim  \norm{  \abs{\nabla}^{\frac{4}{p}-\frac{1}{2}}   \widetilde{\w_{k,2}} }_{L_{t,x}^p ( (t_{k,1} -RN_{k,1}^{-2} , t_{k,1} + RN_{k,1}^{-2}) \times \T^2 )} \lesssim R^{1+ \frac{3}{p}} \parenthese{\frac{N_{k,2}}{N_{k,1}}}^{\frac{2}{p}} \lesssim_R \varepsilon^{\frac{2000}{p}} .
\end{align*}

\item
The proof of this part is same as in {\bf Part (1)}.

\item
Finally, using the integral equation and Bernstein inequality, we have
\begin{align*}
\norm{P_{\leq \varepsilon^{10} N_{k,1} }\w_{k,1} }_{X^{\frac{1}{2}} (I)} & \lesssim \norm{P_{\leq \varepsilon^{10} N_{k,1}} \w_{k,1} (0)}_{H^{\frac{1}{2}}} + \norm{P_{\leq \varepsilon^{10} N_{k,1}} (i\partial_t + \Delta) \w_{k,1}}_{L_t^1 H_x^{\frac{1}{2}}} \\
& \lesssim (\varepsilon^{10} N_{k,1})^{\frac{1}{2}} \norm{\w_{k,1} (0)}_{L_x^2} +  (\varepsilon^{10} N_{k,1})^{\frac{1}{2}} \norm{(i\partial_t+ \Delta) \w_{k,1}}_{L_t^1 L_x^2}\\
& \lesssim_R \varepsilon^2 N_{k,1}^{\frac{1}{2}} N_{k,1}^{-\frac{1}{2}} =\varepsilon^5 .
\end{align*}
\end{enumerate}
We then finish the proof of Claim \ref{claim1}.
\end{proof}
Hence the proof Lemma \ref{lem:7.2} of is complete. 
\end{proof}

With Lemma \ref{lem:7.1} and Lemma \ref{lem:7.2} proved in the previous subsection, we get back to the proof of \eqref{eq:6.8} and \eqref{eq:6.9}.

\subsection{Proof of \eqref{eq:6.8}}
Let us start with the proof of \eqref{eq:6.8}. Write
\begin{align*}
\norm{F(U_{prof, k}^J + e^{it\Delta} R_k^J) - F(U_{prof,k}^J)}_{N(I_k)} \lesssim \sum_{p=0}^4 \norm{\oo_{p, 5-p} (U_{prof,k}J , e^{it\Delta} R_k^J)}_{N(I_k)}. 
\end{align*}
Using the nonlinear estimate (Lemma \ref{lem Nonlinear est}) and the uniform bound of $U_{prof, k}^J$, we can control these terms with $p \leq 3$ as follows:
\begin{align*}
\norm{ \oo_{p, 5-p} (U_{prof, l}^J , e^{it\Delta} R_k^J)}_{N(I_k)} & \lesssim \norm{e^{it\Delta} R_k^J}_{X^{\frac{1}{2}} (I_k)} \norm{e^{it\Delta}R_k^J}_{Z' (I_k)} \norm{U_{prof,k}^J}_{X^{\frac{1}{2}}(I_k)}^p\\
& \lesssim \norm{e^{it\Delta} R_k^J}_{Z' (I_k)} \to 0,
\end{align*}
as $k \to \infty$, $J \to \infty$.

Now we only need to treat the case when $p=4$. As in the proof of Claim \ref{claim2}, fix $\varepsilon >0$, then there exists $A = A(\varepsilon)$ sufficiently large, such that for all $J \geq A$ and $k \geq k_0 (J)$
\begin{align*}
\norm{U_{prof ,k}^J -U_{prof ,k}^A}_{X^{\frac{1}{2}} ((-1,1))} \leq \varepsilon .
\end{align*}
Then it remains to prove that
\begin{align*}
\limsup_{J \to \infty} \limsup_{k \to \infty} \norm{ \oo_{4,1} (U_{prof ,k}^A , e^{it\Delta}R_k^J )}_{N(I_k)} \lesssim \varepsilon.
\end{align*}
By the definition of $U_{prof,k}^A$, it suffices to prove that for any $\alpha_1, \alpha_2, \alpha_3$ and $\alpha_4 \in \{ 0, 1, \dots , A \}$
\begin{align}\label{9.19}
\limsup_{J \to \infty} \limsup_{k \to \infty} \norm{ \oo_{1,1,1,1,1} (U_k^{\alpha_1}, U_k^{\alpha_2}, U_k^{\alpha_3}, U_k^{\alpha_4}, e^{it\Delta} R_k^J )}_{N(I_k)} \lesssim \varepsilon A^{-4} .
\end{align}
If $0 \in \{ \alpha_1 , \alpha_2, \alpha_3, \alpha_4\}$, without loss of generality, we suppose $\alpha_1=0$. By Lemma \ref{lem Nonlinear est}, we have
\begin{align*}
& \quad \norm{\oo_{1,1,1,1,1} (G, U_k^{\alpha_2}, U_k^{\alpha_3}, U_k^{\alpha_4}, e^{it\Delta} R_k^J)}_{N(I_k)} \\
& \lesssim \parenthese{\norm{G}_{Z' (I_k)} \norm{e^{it\Delta} R_k^J}_{X^{\frac{1}{2}} (I_k)} + \norm{G}_{X^{\frac{1}{2}} (I_k)} \norm{e^{it\Delta} R_k^J}_{Z' (I_k)} } \norm{U_k^{\alpha_2}}_{X^{\frac{1}{2}} (I_k)} \norm{U_k^{\alpha_3}}_{X^{\frac{1}{2}} (I_k)} \norm{U_k^{\alpha_4}}_{X^{\frac{1}{2}} (I_k)}\\
& \lesssim \norm{G}_{Z' (I_k)} + \norm{e^{it\Delta} R_k^J}_{Z' (I_k)} \to 0 ,
\end{align*}
as $k \to \infty $, $J \to \infty$.

So it remains to prove \eqref{9.19} when $0 \notin  \{ \alpha_1 , \alpha_2, \alpha_3, \alpha_4\}$. We choose $\theta = \varepsilon A^{-4} /100$ and apply the decomposition in  Euclidean profiles. Then we have for all $k$, $\alpha \in \{ 1, \cdots , A \}$
\begin{align*}
\mathds{1}_{I_k} (t) U_k^{\alpha} = \w_k^{\alpha, \theta , -\infty} + \w_k^{\alpha, \theta , +\infty} \w_k^{\alpha, \theta } + \rho_k^{\alpha, \theta},
\end{align*}
satisfying the conditions \eqref{eq cor4.8} in Corollary \ref{cor Decomp nonlinear Euclidean profiles}. Since $\norm{\rho_k^{\alpha,\theta}}_{X^{\frac{1}{2}}} + \norm{\w_k^{\theta, \pm \infty}}_{Z' } \lesssim \theta$ for $k$ large enough and Lemma \ref{lem Nonlinear est}, it is sufficient to prove that
\begin{align*}
\limsup_{J \to \infty} \limsup_{k \to \infty} \norm{ \oo_{1,1,1,1,1} (\omega_k^{\alpha_1 , \theta} , \omega_k^{\alpha_2 , \theta} , \omega_k^{\alpha_3 , \theta}, \omega_k^{\alpha_4 , \theta}, e^{it\Delta} R_k^J)}_{N(I_k)} \lesssim \varepsilon A^{-4},
\end{align*}
where $\theta= \frac{1}{10} \varepsilon A^{-4}$.

If any two $\alpha_i's$ are different (without loss of generality, suppose $\alpha_1 \neq \alpha_2$), then by Lemma \ref{lem:7.2}
\begin{align*}
\limsup_{k \to \infty} \norm{ \oo_{1,1,1,1,1} (\omega_k^{\alpha_1 , \theta} , \omega_k^{\alpha_2 , \theta} , \omega_k^{\alpha_3 , \theta}, \omega_k^{\alpha_4 , \theta}, e^{it\Delta} R_k^J)}_{N(I_k)} =0.
\end{align*} 
So we only need to prove
\begin{align*}
\limsup_{J \to \infty} \limsup_{k \to \infty} \norm{\oo_{4,1} (\omega_k^{\alpha , \theta} , e^{it\Delta} R_k^J)}_{N(I_k)} \lesssim \varepsilon A^{-4}.
\end{align*}

We apply Lemma \ref{lem:7.1} with $B$ sufficiently large (depending on $R_{\theta}$); thus, fot any $J \geq A$, 
\begin{align}\label{eq 7.14}
\lim_{k \to \infty} \norm{\oo_{4,1} (\omega_k^{\alpha ,\theta} P_{> BN_k^{\alpha}} e^{it\Delta} R_k^J)}_{N(I_k)} \lesssim \varepsilon A^{-4}.
\end{align}

We may also assume that $B$ is sufficiently large, such that for $k$ large enough, using a similar estimate in the proof of Lemma \ref{lem:7.2}, we have
\begin{align}\label{eq 7.15}
\norm{P_{\leq B^{-1} N_k^{\alpha}} \omega_k^{\alpha,\theta}}_{X^{\frac{1}{2}} (I_k)} \lesssim \frac{1}{4} \varepsilon A^{-4}.
\end{align}

In the end, by \eqref{eq 7.14} and \eqref{eq 7.15}, it remains to prove that
\begin{align}\label{eq HL}
\limsup_{J \to \infty} \limsup_{k \to \infty} \norm{\oo_{4,1} (P_{B^{-1} N_k^{\alpha}} \omega_k^{\alpha , \theta} , P_{\leq BN_k^{\alpha}} e^{it\Delta} R_k^J)}_{N(I_k)} =0.
\end{align}
By using \eqref{lem3.8 2} in Lemma \ref{lem Nonlinear est}, we have the following estimates
\begin{align}\label{eq HL'}
\norm{\oo_{4,1} (P_{B^{-1} N_k^{\alpha}} \omega_k^{\alpha , \theta} , P_{\leq BN_k^{\alpha}} e^{it\Delta} R_k^J)}_{N(I_k)} \lesssim \norm{P_{\leq BN_k^{\alpha}} e^{it\Delta} R_k^J)}_{Z'(I_k)} \cdot \norm{P_{B^{-1} N_k^{\alpha}} \omega_k^{\alpha , \theta}}^4_{X^{\frac{1}{2}}(I_k)}.
\end{align}
Together with \[\limsup_{J \to \infty} \limsup_{k \to \infty} \norm{e^{it\Delta} R_k^J}_{Z'  (I_k)} =0,\] \eqref{eq HL'} implies \eqref{eq HL}.  Now we finish the proof of \eqref{eq:6.8}.

\subsection{Proof of \eqref{eq:6.9}}\label{ssec A2}
In this subsection, we work on the proof of \eqref{eq:6.9}. First we write
\begin{align*}
F(U_{prof,k}^J) - \sum_{0 \leq \alpha \leq J} F(U_k^{\alpha}) = \sum_{\{ \alpha_1, \cdots , \alpha_5\} \in \mathcal{G}_J} \oo_{1,1,1,1,1} (U_k^{\alpha_1}, \cdots, U_k^{\alpha_5}) 
\end{align*}
where 
\begin{align*}
\mathcal{G}_J = \bigg\{  \{ \alpha_1, \alpha_2 , \alpha_3, \alpha_4 , \alpha_5\}  : 0 \leq \alpha_1, \cdots , \alpha_5 \leq J   \text{ and } \{ \alpha_1, \cdots , \alpha_5\} \text{ contains at least two distinct numbers} \bigg\}.
\end{align*}
By Claim \ref{claim2}, for any $\varepsilon > 0$, we can find $A(\varepsilon)$ such that
\begin{align*}
\sum_{A \leq \alpha J} \norm{U_k^{\alpha}}_{X^{\frac{1}{2}}} \leq \varepsilon .
\end{align*}
So we have 
\begin{align*}
\norm{ \sum_{\{ \alpha_1, \cdots , \alpha_5\} \in \mathcal{G}_J} \oo_{1,1,1,1,1} (U_k^{\alpha_1}, \cdots, U_k^{\alpha_5})}_{N(I_k)}  \lesssim \norm{ \sum_{\{ \alpha_1, \cdots , \alpha_5\} \in \mathcal{G}_A} \oo_{1,1,1,1,1} (U_k^{\alpha_1}, \cdots, U_k^{\alpha_5}) }_{N(I_k)} + \varepsilon .
\end{align*}
Using Corollary 6.6 (decomposition of $U_k^{\alpha}$) and $\theta = \varepsilon A^{-4}/100$, we have
\begin{align*}
\norm{ \sum_{\{ \alpha_1, \cdots , \alpha_5\} \in \mathcal{G}_J} \oo_{1,1,1,1,1} (U_k^{\alpha_1}, \cdots, U_k^{\alpha_5})}_{N(I_k)} \lesssim \sum_F \norm{\oo_{1,1,1,1,1} (W_k^1, W_k^2,  W_k^3, W_k^4, W_k^5 )}_{N(I_k)},
\end{align*}
where 
\begin{align*}
F : = \bigg\{
(W_k^1, W_k^2,  W_k^3, W_k^4, W_k^5 ): W_k^i \in \{\w_k^{\alpha,\theta,+\infty}, \w_k^{\alpha,\theta,-\infty}, \w_k^{\alpha,\theta}, \rho_k^{\alpha,\theta}\}, 0\leq \alpha \leq A, \\ \text{ for } i=1, \cdots , 5, \text{ at least two different } \alpha's
\bigg\}.
\end{align*}
Note that for $\alpha =0$, we see that $\w_k^{0,\theta, +\infty} = G \mathds{1}_{(-T_{\theta}^{-1},T_{\theta}^{-1})}$, $\w_k^{\alpha, \theta, -\infty} = \w_k^{\alpha,\theta} = \rho_k^{\alpha,\theta} =0$. Let us consider the following cases:

\case{1}{$(W_k^1 , \cdots, W_k^5)$ contains at least one error component $\rho_k^{\alpha, \theta}$.} 
By the nonlinear estimate Lemma \ref{lem Nonlinear est}, we have 
\begin{align*}
\norm{\oo_{1,1,1,1,1} (W_k^1, W_k^2,  W_k^3, W_k^4, W_k^5)}_{N(I_k)} \lesssim \norm{\rho_k^{\alpha ,\theta}}_{X^{\frac{1}{2}}(I_k)} \lesssim \theta .
\end{align*}

\case{2}{$(W_k^1 , \cdots, W_k^5)$ contains at least two scattering components $\w_k^{\alpha, \theta, \pm \infty}$ and $\w_k^{\beta, \theta ,\pm \infty}$.} 
By Lemma \ref{lem Nonlinear est},
\begin{align*}
\norm{\oo_{1,1,1,1,1} (\w_k^{\alpha, \theta , \pm \infty} , w_k^{\beta, \theta ,\pm \infty} , W_k^3 , W_k^4 , W_k^5 )}_{N(I_k)}  \lesssim \norm{\w_k^{\alpha ,\theta , \pm \infty}}_{Z' (I_k)} + \norm{\w_k^{\beta ,\theta , \pm \infty}}_{Z' (I_k)} \lesssim 2 \theta.
\end{align*}

\case{3}{$(W_k^1 , \cdots, W_k^5)$ contains two orthogonal case $\w_k^{\alpha,\theta}$, $\w_k^{\beta,\theta}$ with $\alpha \neq \beta$.} 
By Lemma \ref{lem:7.2}, we obtain that 
\begin{align*}
\limsup_{k \to \infty} \norm{\oo_{1,1,1,1,1} (\w_k^{\alpha,\theta} , \w_k^{\beta,\theta}, W_k^3, W_k^4, W_k^5)}_{N(I)}.
\end{align*}

\case{4}{Other scenarios.}
More precisely, $\oo_{4,1} (\w_k^{\alpha ,\theta} , \w_k^{\beta ,\theta , \pm \infty})$ for  $\alpha \neq \beta$, $\alpha = 1,2,\cdots, A$, $\beta = 0, 1, \cdots, A$ and for $\alpha=0$, $\w^{0,\theta, \pm \infty} : = G \mathds{1}_{(-T_{\theta}^{-1} , T_{\theta}^{-1}) }(t)$.

Let us discuss the sub-cases in this regime.
\begin{enumerate}[\bf (a)]
\item
\scase{4.1}{$\limsup_{k \to \infty} \frac{N_{k,\beta}}{N_{k,\alpha}} = \infty$.}
By a direct consequence of Lemma \ref{lem:7.1}, choosing $B$ and $k$ large enough, we have that
\begin{align*}
\norm{\oo_{4,1} (\w^{\alpha, \theta} , P_{>BN_{k,\alpha}} \w^{\beta, \theta, \pm \infty}) }_{N(I_k)} \lesssim B^{-1/200} + N_{k,\alpha}^{-1/200} \lesssim \theta .
\end{align*}
For the other part, we write
\begin{align*}
\norm{\oo_{4,1} (\w^{\alpha,\theta} , P_{\leq BN_{k,\alpha}} \w^{\beta, \theta, \pm \infty})}_{N(I_k)} \lesssim \norm{P_{\leq BN_{k,\alpha}} \w^{\beta , \theta , \pm \infty}}_{X^{\frac{1}{2}}(I_k)} \norm{\w^{\alpha , \theta}}_{X^{\frac{1}{2}}(I_k)}^4 \lesssim \theta,
\end{align*}
which follows 
\begin{align*}
\norm{P_{\leq BN_{k,\alpha}} \w^{\beta , \theta , \pm \infty}}_{X^{\frac{1}{2}}(I_k)} = \norm{P_{\leq BN_{k,\alpha}} \Pi_{x_k^{\beta}} T_{N_{k,\beta}} (\phi^{\beta , \theta , \pm \infty})}_{X^{\frac{1}{2}}(I_k)} = \norm{P_{\leq B \frac{N_{k,\alpha}}{N_{k,\beta}}}  \phi^{\beta , \theta , \pm \infty}}_{\dot{H}^{\frac{1}{2}} (\R^2)} \to 0 ,
\end{align*}
as $k \to \infty$.

\item
\scase{4.2}{$\limsup_{k \to \infty} \frac{N_{k,\alpha}}{N_{k,\beta}} = + \infty$.}
By a similar estimate as in {\bf (5)}  in Claim \ref{claim1}, we have
\begin{align}\label{eq Case 4 1}
\norm{P_{\leq N_{k,\beta}} \w_k^{\alpha ,\theta}}_{X^{\frac{1}{2}}(I_k)} = \norm{P_{\leq N_{k,\alpha} \frac{N_{k,\beta}}{N_{k,\alpha}}}\w_k^{\alpha ,\theta}}_{X^{\frac{1}{2}} (I_k)}  \lesssim \theta
\end{align}
as $k$ is large enough. We also have
\begin{align}\label{eq Case 4 2}
\norm{P_{\geq BN_{k,\beta} \w^{\beta, \theta , \pm \infty}}}_{X^{\frac{1}{2}}(I_k)} = \norm{P_{\geq BN_{k,\beta}} (T_{N_k^{\beta}} (P_{\leq R_{\theta ,\beta}} \phi^{\beta, \theta, \pm \infty})  )}_{H^{\frac{1}{2}} (\T^2)} \lesssim B^{-1} + \theta \lesssim \theta
\end{align}
by taking $k$ and $B$ large enough.

By Lemma \ref{lem Nonlinear est}, \eqref{eq Case 4 1} and \eqref{eq Case 4 2} , we have
\begin{align*}
\norm{\oo_{4,1} (\w_k^{\alpha ,\theta} , \w_k^{\beta, \theta, \pm \infty})}_{N(I_k)} & \lesssim \theta + \norm{\oo_{4,1} (P_{>N_{k,\beta}} \w_k^{\alpha ,\theta} , P_{\leq BN_{k,\beta}} \w^{\beta, 
\theta , \pm \infty})}_{N(I_k)} \\
& \lesssim \norm{P_{\leq BN_{k,\beta}} \w^{\beta ,\theta , \pm \infty}}_{Z' (I_k)} \norm{P_{>N_{k,\beta}} \w_k^{\alpha , \theta}}_{X^{\frac{1}{2}}(I_k)}^4 \lesssim \theta.
\end{align*}

\item
\scase{4.3}{$N_{k,\alpha} = N_{k,\beta}$ and $t_k^\alpha = t_k^\beta$ as $k \to \infty$.}
By definition of $\w_k^{\alpha , \theta}$ and $\w_k^{\beta , \theta , \pm \infty}$ in Corollary 6.6, for $k$ large enough, we have that $\w_k^{\alpha ,\theta} \w_k^{\beta , \theta , \pm \infty} =0$.

\item
\scase{4.4}{$N_{k,\alpha} = N_{k,\beta}$ and $N_{k,\alpha}^2 \abs{t_k^{\alpha} - t_k^{\beta}} \to \infty$, as $k\to \infty$.} 
By \eqref{eq Claim2} and \eqref{eq Claim3}, for any $T \leq N_k$, we obtain that 
\begin{align}\label{eq Case 4 3}
\norm{\w_k^{\beta, 
\theta , \pm \infty}}_{L_x^2(\T^2)} = \norm{P_{\leq R_{\theta} N_{k,\beta}} \w_k^{\beta , \theta, \pm \infty}}_{L_x^2(\T^2)} \lesssim (1+R_{\theta})^{-10} N_{k,\beta}^{-\frac{1}{2}},
\end{align}
and
\begin{align}\label{eq Case 4 4}
\sup_{\abs{t-t_k^{\beta}} \in [TN_{k,\beta}^{-2} , T^{-1}]} \norm{\w_k^{\beta , \theta , \pm \infty}}_{L_x^{\infty} (\T^2)} \lesssim T^{-1} N_{k,\beta}^{\frac{1}{2}} .
\end{align}
Interpolating \eqref{eq Case 4 3} and \eqref{eq Case 4 4}, we have that
\begin{align}\label{8.11}
\sup_{\abs{t-t_k^{\beta}} \in [TN_{k,\beta}^{-2} , T^{-1}]} \norm{\w_k^{\beta , \theta , \pm \infty}}_{L_x^p (\T^2)} \lesssim T^{\frac{2}{p}-1} N_{k,\beta}^{\frac{1}{2} -\frac{2}{p}} .
\end{align}
Denote $N_{k,\alpha} = N_{k,\beta} := N_k$. By choosing $T_k = N_k \abs{t_k^{\alpha} - t_k^{\beta}}^{\frac{1}{2}} \to \infty$ as $k \to \infty$, we have also
\begin{align*}
\sup_{t \in [t_k^{\alpha} - \frac{T_\theta}{N_{k,\alpha}^2} , t_k^{\alpha} + \frac{T_\theta}{N_{k,\alpha}^2} ]} \norm{\w^{\beta, \theta , \pm \infty}}_{L_x^{\infty} (\T^2)} \lesssim T_k^{-1} N_{k}^{\frac{1}{2}}.
\end{align*}

So by Leibniz rule, \eqref{8.11} and Corollary \ref{cor Decomp nonlinear Euclidean profiles}, we have that
\begin{align}
& \quad \norm{\oo_{4,1} (\w_k^{\alpha ,\theta} , \w_k^{\beta , \theta , \pm \infty})}_{N([t_k^{\alpha } - \frac{T_{\theta}}{N_{k,\alpha}^2} , t_k^{\alpha } + \frac{T_{\theta}}{N_{k,\alpha}^2} ])} \\
&  \lesssim \norm{\oo_{4,1} (\w_k^{\alpha ,\theta} , \w_k^{\beta , \theta , \pm \infty})}_{L_t^1 H_x^{\frac{1}{2}} ([t_k^{\alpha } - \frac{T_{\theta}}{N_{k,\alpha}^2} , t_k^{\alpha } + \frac{T_{\theta}}{N_{k,\alpha}^2} ] \times \T^2)} \notag\\
& \lesssim \int_{t_k^{\alpha } - \frac{T_{\theta}}{N_{k,\alpha}^2}}^{t_k^{\alpha } + \frac{T_{\theta}}{N_{k,\alpha}^2}} \norm{\w_k^{\alpha ,\theta}}_{L_x^8(\T^2)}^4 \norm{\inner{\nabla}^{\frac{1}{2}} \w_k^{\beta, \theta, \pm \infty}}_{L_x^{\infty}(\T^2)} \, dt  \notag\\
& \quad + \int_{t_k^{\alpha } - \frac{T_{\theta}}{N_{k,\alpha}^2}}^{t_k^{\alpha } + \frac{T_{\theta}}{N_{k,\alpha}^2}} \norm{\inner{\nabla}^{\frac{1}{2}} \w_k^{\alpha ,\theta}}_{L_x^8(\T^2)} \norm{\w_k^{\alpha ,\theta}}_{L_x^8(\T^2)}^3 \norm{\w_k^{\beta, \theta, \pm \infty}}_{L_x^{\infty}(\T^2)} \, dt \label{eq Case 4 5}
\end{align}
Using Corollary \ref{cor Decomp nonlinear Euclidean profiles}, we have
\begin{align}\label{eq Case 4 6}
\begin{aligned}
&\norm{\w_k^{\alpha ,\theta}}_{L_x^8} \lesssim R_{\theta}^{\frac{5}{4}} N_k^{\frac{1}{4}} ,\\
&\norm{\inner{\nabla}^{\frac{1}{2}} \w_k^{\alpha ,\theta}}_{L_x^8} \lesssim R_{\theta}^{\frac{5}{4}} N_k^{\frac{3}{4}} .
\end{aligned}
\end{align}
Continuing from \eqref{eq Case 4 5}, \eqref{eq Case 4 6} and \eqref{8.11} give us
\begin{align*}
\eqref{eq Case 4 5} \lesssim \frac{T_{\theta}}{N_k^2} (R_{\theta}^5 N_k \cdot T_k^{-1} N_k^{\frac{1}{2}} (R_{\theta} N_k)^{\frac{1}{2}} + R_{\theta}^5 N_k^{\frac{3}{2}} \cdot T_k^{-1} N_k^{\frac{1}{2}}) \lesssim T_{\theta} R_{\theta}^5 T_k^{-1} \to 0 ,
\end{align*}
as $k \to \infty $.
Summarizing all above cases, we have by taking$\theta = \frac{\varepsilon}{100 A^4}$. So we prove that
\begin{align*}
\norm{F(U_{prof,k}^J) - \sum_{0 \leq \alpha \leq J} F(U_k^{\alpha})}_{N(I_k)} \lesssim \varepsilon ,
\end{align*}
for arbitrary small $\varepsilon$.
\end{enumerate}
Now we finish the proof of \eqref{eq:6.9}.

\bibliography{NLS}
\bibliographystyle{plain}
\end{document}